\PassOptionsToPackage{top=2cm,bottom=2cm,left=2cm,right=2cm}{geometry}
\documentclass[preprint,11pt,3p,times]{elsarticle}
\usepackage{geometry}

\usepackage[utf8]{inputenc}
\usepackage[USenglish]{babel}
\usepackage[T1]{fontenc}

\DeclareUnicodeCharacter{2009}{\,} 

\usepackage{lmodern}
\usepackage{amsmath}
\usepackage{bm}
\usepackage{amssymb}
\usepackage{mathtools}
\usepackage{amsthm}
\usepackage{latexsym}
\usepackage{fixmath}
\usepackage{upgreek}

\usepackage{natbib}

\biboptions{sort&compress}

\usepackage{comment}

\usepackage{titlesec}
\usepackage{hyperref}
\usepackage[capitalise,nameinlink]{cleveref}

\usepackage{graphicx}
\usepackage{float}
\usepackage{subcaption}
\numberwithin{equation}{section}

\newtheorem{theorem}{Theorem}[section]

\newtheorem{lemma}{Lemma}[section]
\newtheorem{remark}{Remark}[section]

\usepackage{xcolor}

\usepackage{array}

\hypersetup{
	colorlinks=true,
	citecolor=blue,
	linkcolor=blue,
	filecolor=magenta,      
	urlcolor=cyan,
}

\def\div{\textnormal{div}}
\def\bV{\mathbf{V}}
\def\bv{\mathbf{v}}
\def\bu{\mathbf{u}}
\def\bL{\mathbf{L}}
\def\bn{\mathbf{n}}
\def\bq{\mathbf{q}}
\def\br{\mathbf{r}}

\def\bQ{\boldsymbol{Q}}
\def\bPi{\boldsymbol{\Pi}}
\def\bsi{\boldsymbol{\sigma}}
\def\btau{\boldsymbol{\tau}}
\def\bxi{\boldsymbol{\xi}}

\def\<{\langle}
\def\>{\rangle}

\def\esi{e_{\bsi}}
\def\etu{e_{\tilde\bu}}
\def\eu{e_{\bu}}
\def\eq{e_{\bq}}

\newcommand{\e}[1]{\cdot 10^{#1}}
\newcommand{\eg}{e.g., }
\newcommand{\ie}{i.e., }
\newcommand{\REV}{\textit{REV}}

\newcommand{\spd}{\textit{spd}}
\newcommand{\st}{s.t.}

\newcommand{\PP}{\textit{PP}}
\newcommand{\PjP}{\textit{PjP}}

\newcommand{\DOFs}{\textit{DOFs}}

\newcommand{\nuu}{$\frac{\textnormal{m}^2}{s}$}
\newcommand{\msq}{$\textnormal{m}^2$}

\begin{document}

\begin{frontmatter}



\title{Numerical solution of the unsteady Brinkman equations in the framework of $H$(div)-conforming finite element methods}

\author[label1]{Costanza Aricò\corref{cor1}}

\ead{costanza.arico@unipa.it}
\cortext[cor1]{Corresponding Author}
\affiliation[label1]{organization={Department of Engineering, University of Palermo},
	addressline={viale delle Scienze},
	city={Palermo},
	postcode={90128},
	country={Italy}}

\author[label2]{Rainer Helmig}
\affiliation[label2]{organization={Institute for Modelling Hydraulic and Environmental Systems (IWS),
		Department of Hydromechanics and Modelling of Hydrosystems, University of Stuttgart},
	addressline={Pfaffenwaldring 61},
	city={Stuttgart},
	postcode={D-70569},
	country={Germany}}

\author[label3]{Ivan Yotov}
\affiliation[label3]{organization={Department of Mathematics,
		University of Pittsburgh},
	addressline={301 Thackeray Hall},
	city={Pittsburgh},
	postcode={PA 15260},
	country={USA}}




\begin{abstract}
We present projection-based mixed finite element methods for the solution of the unsteady Brinkman equations for incompressible single-phase flow with fixed in space porous solid inclusions. At each time step the method requires the solution of a predictor and a projection problem. The predictor problem, which uses a stress-velocity mixed formulation, accounts for the momentum balance, while the projection problem, which is based on a velocity-pressure mixed formulation, accounts for the incompressibility. The spatial discretization is $H$(div)-conforming and the velocity computed at the end of each time step is pointwise divergence-free. Unconditional stability of the fully-discrete scheme and first order in time accuracy are established. Due to the $H$(div)-conformity of the formulation, the methods are robust in both the Stokes and the Darcy regimes. In the specific code implementation, we discretize the computational domain using the Raviart--Thomas space \(RT_1\) in two and three dimensions, applying a second-order accurate multipoint flux mixed finite element scheme with a quadrature rule that samples the flux degrees of freedom. In the predictor problem this allows for a local elimination of the viscous stress and results in element-based symmetric and positive definite systems for each velocity component with \(\left(d+1\right)\) degrees of freedom per simplex (where \(d\) is the dimension of the problem). In a similar way, we locally eliminate the corrected velocity in the projection problem and solve an element-based system for the pressure. Numerical experiments are presented to verify the convergence of the proposed scheme and illustrate its performance for several challenging applications, including one-domain modeling of coupled free fluid and porous media flows and heterogeneous porous media with strong discontinuity of the porosity and permeability values.
\end{abstract}



\end{frontmatter}


\section{Introduction}

The Brinkman equations \cite{brinkman1949calculation} can be considered as a combination of the Stokes and Darcy models, with a viscous term added to Darcy's law. The model has important industrial, biomedical, and environmental applications, when the porous medium has heterogeneous porosity/permeability values, in such a way that in some portions of the domain the flow is governed by Darcy's law and elsewhere by the Stokes equations, \eg flow around or within foams, fuel cells, heat exchangers, drying and filtration processes, flows in biological tissues, and groundwater flow in fractured porous media. The Brinkman equations are derived from the volume averaging of the pore-scale Stokes equations at the continuum representative elementary volume (\REV) scale \cite{CHANDESRIS20062137, LeBarsWorster2006}. Pore-scale simulations are accurate, but computationally expensive due to the high number of grid elements needed to discretize the void spaces in the porous medium. Additionally, a detailed knowledge of the solid inclusion geometry is often missing. For these reasons continuum \REV\ scale approaches have become popular. The Brinkman model includes porosity and permeability parameters, which can vary spatially. Thus, the model allows for flow in either the Stokes or the Darcy regimes in different regions, using the so called ``\textit{one-domain} approach'' (ODA) with a single ``fictitious'' medium and one set of governing equations. The Brinkman equations represent an alternative way to model the coupling of Stokes and Darcy equations through interface conditions \cite{LSY, Riv-Yot, Vas-Yot, Lipnikov2014, KANSCHAT20105933, CAMANO2015362, Boon2024, gos2011}.\par

One of the most challenging aspects of the numerical approximation of the Brinkman equations is the construction of stable and robust methods in both the Stokes and the Darcy regimes. Standard inf-sup stable velocity-pressure mixed finite element (MFE) methods for the Stokes equations, such as the MINI and the Taylor--Hood elements, provide suboptimal convergence orders in the Darcy regime \cite{MarTaiWin}. On the other hand, stable $H$(div)-conforming mixed Darcy elements, like the Raviart--Thomas or the Brezzi--Douglas--Marini elements, do not satisfy the required \(H^1\)-conformity in the Stokes regime. Various approaches have been developed, including enhancing stable Darcy elements to enforce some tangential continuity \cite{MarTaiWin,GuzmNeil-Brinkman,XieXuXue}, penalized $H$(div)-conforming methods \cite{Konno-Stenberg, Konno2012-lt, KANSCHAT20105933},
stabilized methods based on Stokes elements \cite{Burman-Brinkman,Burman-SD,Badia-Codina,Nafa,Masud,JunSten-Brinkman,CORREA20092710,Arbogast-SD}, and weak Galerkin methods \cite{Mu-Hdiv-WG, Mu-Wang-Ye}. 

The above mentioned methods are based on the standard velocity-pressure formulation, where the velocity is approximated in the (possibly broken) $H^1$-norm, which makes it difficult to obtain robustness in both the Stokes and the Darcy regimes.
An alternative approach is to consider dual mixed formulations, where the stress or the pseudostress is introduced as a primary variable and the velocity is approximated in the $L^2$-norm \cite{Qian, Gatica-pseudostress-Brinkman, BustGatGonz, Howel-dual-mixed-Brinkman, Zhao-Brinkman, Caucao-Yotov, BF-three-field}. In these schemes, the mass conservation is only weakly imposed and the approximation of the velocity is discontinuous. Vorticity-based mixed formulations with $H$(div)-conforming velocity are developed in \cite{Vassilevski, Anaya}, while an augmented vorticity-based mixed method with $H^1$-conforming velocity is proposed in \cite{BF-vorticity}. 

In this paper we develop a family of new $H$(div)-conforming projection-based MFE methods for the unsteady Brinkman equations, which provide pointwise divergence-free velocity. Moreover, due to the $H$(div)-conformity of the formulation, the methods
are robust in both the Stokes and the Darcy regimes. Extending the methodology proposed in \cite{ARICO2025117616} for the unsteady Stokes equations (see also \cite{Arico-NS-RT0, Arico-ODA}), we apply an incremental pressure correction method in the framework of projection schemes \cite{BROWN2001464, GUERMOND2006}. At each time step we solve a predictor problem for the viscous term, followed by a projection problem to incorporate the incompressibility constraint. The predictor problem uses a mixed stress-velocity formulation \cite{Cai-MFE-NS,gos2011}, whereas the projection problem is based on a mixed velocity-pressure formulation. The stable MFE spaces Brezzi--Douglas--Marini (\(BDM_k\)), \(k \ge 1\), or Raviart--Thomas (\(RT_k\)), \(k \ge 0\), on simplicial grids \cite{BBF} are considered for the components of the viscous stress in the predictor problem and the velocity in the projection problem. The velocity at the end of each time iteration is $H$(div)-conforming and divergence-free polynomial of degree \(\le k\). We establish unconditional stability of the fully discrete method, obtaining bounds on the viscous stress and the velocity that are robust in both the Stokes and Darcy limits. We then derive a first-order in time error bound for the semi-discrete continuous-in-space scheme.

A specific second-order in space method based on the \(RT_1\) space and discontinuous piecewise linear polynomials on unstructured triangular or tetrahedral grids is developed in the second part of the paper. We avoid the solution of saddle point problems by employing the multipoint flux mixed finite element (MFMFE) method, which was originally proposed for Darcy flow as a first-order method in \cite{W-Y} using the \(BDM_1\) space and extended to a second-order method based on the $RT_1$ space in \cite{Radu}, see also \cite{Caucao2020AMS} for the multipoint stress mixed finite element for the Stokes equations. A quadrature rule that samples the flux degrees of freedom allows for mass lumping and local elimination of the stress in the predictor problem and the velocity in the projection problem. As a result, at each time step only symmetric and positive definite algebraic systems need to be solved for the velocity components and the pressure in the predictor and projection problems, respectively. This makes the numerical algorithm very efficient in terms of computational effort.\par

The structure of the paper is as follows. The governing equations are presented in \cref{ge}. The numerical algorithms with discretization in space and time and the MFE projection methods are presented in \cref{numalg}. In \cref{stab_err_an} we develop the stability and the time discretization error analysis. In \cref{implem} we present the implementation of a second order MFMFE scheme on unstructured triangular or tetrahedral grids and in \cref{tests} we show four numerical applications of the MFMFE scheme. In the first test we verify the second and first order convergence in space and time, respectively. In tests two and three we model coupled free fluid and porous media flows and compare our results to numerical solutions available in the literature. We also consider heterogeneous porous media with strong contrasts of porosity and permeability values The last test is a challenging biomedical application for a very irregular geometry domain. Finally, we come to the conclusions in \cref{conclusion}.
\section{Governing equations} \label{ge}

We consider a single-phase, incompressible and Newtonian fluid within and around a porous medium with rigid and fixed in space solid inclusions. Let \(\Omega \subset \mathbf{R}^d \), \(d = 2, 3\), and \(\mathcal{T}\) be the computational domain and the final simulation time, respectively. The Brinkman equations are
\begin{subequations}  \label{eq:governing_Eqq}
	\begin{gather}
		\nabla\ \cdot \mathbf{u} = 0 \quad \textnormal{in} \quad \Omega \times \left(0, \mathcal{T}\right],  \label{eq:continuity}\\
		\frac{1}{\phi} \frac{\partial \mathbf{u}}{\partial t} + \nabla \Psi -  \nabla \cdot \left(\frac{\nu}{\phi} \nabla \mathbf{u}\right) + \nu \mathfrak{K} \mathbf{u} = 0 \quad \textnormal{in} \quad \Omega \times \left(0, \mathcal{T}\right], \label{eq:momentum}
	\end{gather}  
\end{subequations}
where \(\mathbf{u}\) is the surface average fluid velocity vector (or Darcy velocity), \(\Psi = \frac{p}{\rho}\) is the intrinsic average kinematic pressure (with \(p\) the intrinsic average fluid pressure), \(\nu = \frac{\mu}{\rho} > 0\) is the kinematic fluid viscosity (with \(\rho\) and \(\mu\) the fluid density and the dynamic fluid viscosity), \(\phi\) is the porosity of the fictitious medium and \(\mathfrak{K}\) is the inverse of the symmetric and positive definite (\spd) permeability tensor \(\mathbf{K}\). We assume that \(\phi\) and \(\mathbf{K}\) vary in space. For simplicity we assume in the analysis that $\phi$ is constant on each element of the mesh. The last term on the l.h.s. of \cref{eq:momentum} represents the drag force due to the pore scale momentum transfer of the fluid and solid phases in the porous domain \cite{CHANDESRIS20062137, OCHOATAPIA19952635,osti_6416113}. Some basic concepts of the volume averaging techniques that can be used to derive \eqref{eq:momentum} can be found in \cite{GRAY1975229,osti_6416113}. \par

We note that it is possible to insert or not the ratio \(\frac{\nu}{\phi}\) within the divergence operator in \eqref{eq:momentum}. Both versions have been proposed in the literature \cite{CHANDESRIS20062137,OCHOATAPIA19952635,GOYEAU20034071}. As discussed after Eq. (65) in \cite{CHANDESRIS20062137}, this is a modeling choice that leads to different stress continuity conditions in the case of a sharp interface, either in terms of the ``effective'' viscous stress \(-\frac{\nu}{\phi} \nabla \mathbf{u}\) or the ``standard'' viscous stress \(-\nu \nabla \mathbf{u}\).
With our choice we obtain a symmetric algebraic system, while keeping \(\frac{\nu}{\phi}\) outside the divergence operator results in a non-symmetric system, see Remarks \ref{choice} and \ref{sym-system}. 

For the problem in \eqref{eq:governing_Eqq} be well posed, Boundary Conditions (BCs) and Initial Conditions (ICs) need to be properly assigned. \(\Gamma=\Gamma_d\cup\Gamma_n\) is the boundary of \(\Omega\) and \(\mathbf{n}\) its unitary orthogonal vector, outward oriented. We assign Dirichlet BCs for the velocity and the normal stress component on \(\Gamma_d\) and \(\Gamma_n\), respectively. The ICs and BCs needed for the solution of system \eqref{eq:governing_Eqq} are 
\begin{subequations} \label{eq:IBCs}
	\begin{gather}
		\mathbf{u}=\mathbf{u}_b, \qquad \textnormal{on} \quad \Gamma_d, \quad t \in \left[0, \mathcal{T}\right], \label{eq:IBCs1}\\
		\boldsymbol{\Sigma}_b = (\Psi \mathbf{I} - \frac{\nu}{\phi} \nabla \mathbf{u}) \mathbf{n},  \qquad \textnormal{on} \quad \Gamma_n, \quad t \in \left[0, \mathcal{T}\right], \label{eq:IBCs2} \\
		\mathbf{u}=\mathbf{u}_0 \quad \textrm{with} \quad \nabla \cdot \mathbf{u}_0=0, \qquad \Psi = \Psi_0, \quad \textnormal{in} \quad \Omega, \quad t= 0,
	\end{gather}
\end{subequations} 
\noindent where \(\mathbf{u}_b\) is the velocity vector imposed at the boundary and \(\boldsymbol{\Sigma}_b\) is the total stress imposed at the boundary along the direction orthogonal to \(\Gamma_n\), \(\mathbf{I}\) is the identity matrix and \(\mathbf{u}_0\) and \(\Psi_0\) are the initial value of \(\mathbf{u}\) and \(\Psi\) in \(\Omega\), respectively. \par	

We will use the following standard notation. For a domain \(D \subset \mathbf{R}^d\), the \(L^2(D)\)-inner product and the $L^2(D)$-norm for scalar, vector, and tensor valued functions are denoted by \(\left(\cdot,\cdot\right)_D\) and $\|\cdot\|_D$, respectively. The subscript \textit{D} will be omitted if \(D = \Omega\). For a section of the boundary $S \in \mathbf{R}^{d-1}$, $\< \cdot,\cdot \>_S$ marks the $L^2(S)$ inner product (or duality pairing). We will also use the space  
\begin{equation*}
	H(\div,\Omega) = \left\{\mathbf{v}\in (L_2\left(\Omega\right))^d: \ \nabla \cdot \mathbf{v} \in L_2\left(\Omega\right) \right\},
\end{equation*}
and we define
\begin{equation*}
	\mathbf{V} = H(\div,\Omega), \quad \mathbf{V}_{0,\Gamma_s} = \left\{\mathbf{v} \in \mathbf{V} : \mathbf{v}\cdot \mathbf{n} = 0 \ \textnormal{on} \ \Gamma_s\right\}, \ s \ \textnormal{is } d \ \textnormal{or } n, \quad W = L_2(\Omega).
\end{equation*}  

\section{Numerical algorithm} \label{numalg}
We apply the numerical procedure proposed in \cite{ARICO2025117616} for the unsteady Stokes equations, here specifically modified for the solution of the Brinkman equations. 

\subsection{Discretization in time and space} \label{discr}

The simulation time \(\mathcal{T}\) is subdivided into \(N\) intervals with uniform size $\Delta t = \frac{\mathcal{T}}{N}$ and vertices \(t^n = n \Delta t\), $n = 0, \ldots, N$. Let $\varphi^n$ denote the value a variable $\varphi$ at time $t^n$. System \eqref{eq:governing_Eqq} is solved by an incremental pressure correction scheme (\eg \cite{GUERMOND2006}), by solving sequentially a predictor and a projection problem at each time step. Introducing the notation
\begin{equation} \label{tensor_L}
	\mathbf{L} = \nu \mathfrak{K},
\end{equation}
the discretized-in-time form of \eqref{eq:momentum} is
\begin{subequations} \label{FTS}
	\begin{gather}
		\frac{1}{\phi}\frac{\tilde{\mathbf{u}}^{n+1} - \mathbf{u}^n}{\Delta t} + \nabla \Psi^n - \nabla \cdot \left(\frac{\nu}{\phi} \nabla \tilde{\mathbf{u}}^{n+1} \right) + \mathbf{L} \tilde{\mathbf{u}}^{n+1} = 0 ,  \label{eq:Ps} \\
		\frac{1}{\phi} \frac{\mathbf{u}^{n+1} - \tilde{\mathbf{u}}^{n+1}}{\Delta t} + \nabla \left(\Psi^{n+1} - \Psi^n\right) + \mathbf{L} \left(\mathbf{u}^{n+1} - \tilde{\mathbf{u}}^{n+1} \right) = 0, \label{eq:Cs}
	\end{gather}
\end{subequations}
where \eqref{eq:Ps} is the predictor problem (\PP) and \eqref{eq:Cs} is the projection problem (\PjP). For the mixed formulation of the problem, we introduce the variables
\begin{equation} \label{eq:notations}
	\boldsymbol{\sigma}= -\frac{\nu}{\phi} \nabla \mathbf{u}, \quad \qquad \mathbf{q}= \nabla \Psi,
\end{equation}
\noindent where \(\boldsymbol{\sigma}\) is the $\left(d \times d\right)$ viscous pseudostress tensor. Thanks to \eqref{eq:notations}, \eqref{eq:Ps} becomes
\begin{equation} \label{eq:Ps2}
	\frac{1}{\phi} \frac{\tilde{\mathbf{u}}^{n+1} - \mathbf{u}^n}{\Delta t} + \mathbf{q}^n + \nabla \cdot \boldsymbol{\sigma}^{n+1} + \mathbf{L} \tilde{\mathbf{u}}^{n+1} = 0.  
\end{equation}

The spatial discretization of \(\Omega\) is based on a geometrically conforming grid \(T_h\) made of $N_T$ non-overlapping simplices \(E\) (triangles if \(d = 2\) or tetrahedra if \(d = 3\)). Two neighboring simplices share a common face \(e\). We denote by \(\mathfrak{S}_T\) the total number of faces in \(T_h\).
We will use either the Brezzi--Douglas--Marini (BDM) or the Raviart--Thomas (RT) pairs of MFE spaces on \(T_h\), \(\mathbf{V}_h \times W_h \subset \bV\times W\) \cite{BBF}, which satisfy
\begin{equation}\label{div-Vh}
	\nabla\cdot \bV_h = W_h.
\end{equation}
We will also utilize the tensor-valued space $(\bV_h)^d$ with each row in $\bV_h$. Let $Q_h:L^2(\Omega) \to W_h$ be the $L^2$-orthogonal projection operator onto $W_h$, such that for any $w \in L^2(\Omega)$,
\begin{align}
	(Q_h w - w, w_h) = 0 \ \ \forall w_h \in W_h. \label{Qh-defn}
\end{align}
By $Q_h^\Gamma:L^2(\Gamma) \to \bV_h\cdot\bn$ and $\bQ_h^\Gamma:(L^2(\Gamma))^d \to (\bV_h)^d\,\bn$ we denote the $L^2$-orthogonal projection operators onto $\bV_h\cdot\bn$ and $(\bV_h)^d\,\bn$, respectively, such that for any 
$\varphi \in L^2(\Gamma)$ and any $\bv \in (L^2(\Gamma))^d$,
\begin{align*}
	\<\varphi - Q_h^\Gamma\varphi,\bv_h\cdot\bn\>_{\Gamma} = 0 \ \ \forall \bv_h \in \bV_h, \quad
	\<\bv - \bQ_h^\Gamma\bv,\btau_h \,\bn\> _{\Gamma} = 0 \ \ \forall \btau_h \in (\bV_h)^d.
\end{align*}
The mixed interpolant $\bPi_h:(H^1(\Omega))^d \to \bV_h$ \cite{BBF} is also used, which satisfies for any $\bv \in (H^1(\Omega))^d$,
\begin{align}
	&  (\nabla\cdot(\bPi_h \bv - \bv),w_h) = 0 \ \ \forall w_h \in W_h, \label{Pi-div}\\
	& \<(\bPi_h \bv - \bv)\cdot\bn,\bv_h\cdot\bn\>_{\Gamma} = 0 \ \ \forall \bv_h \in \bV_h. \label{Pi-n}
\end{align}
The vector variant of $Q_h$, $Q_h^d: (L^2(\Omega))^d \to (W_h)^d$ will also be used.

\subsection{The mixed finite element projection method} \label{MFEPjM}

We initialize the solution at $t = 0$ as: $\bu_h^0 = \bPi_h \bu_0$, $\Psi_h^0 = Q_h \Psi_0$, $\bq_h^0 = Q_h^d \,\nabla \Psi_0$, $\Psi^0_{b} = \Psi^0|_{\Gamma}$. The numerical method is as follows.\par 
Time loop : do \(n = 0,\ldots, N-1\):

\begin{itemize}	
	\item MFE discretization of the \PP\ \eqref{eq:Ps2}:\\
	Find \(\bsi_h^{n+1} \in (\mathbf{V}_h)^d  : \bsi_h^{n+1} \mathbf{n} = \bQ_h^\Gamma(\mathbf{\Sigma}_b^{n+1} - \Psi_b^n \mathbf{n})\) on \(\Gamma_n\) and \(\tilde{\mathbf{u}}_h^{n+1} \in (W_h)^d\), \st
	\begin{subequations} \label{eq:Predproblem}
		\begin{align}
			& \hskip - .1cm \left( \frac{\phi}{\nu}  \bsi_h^{n+1}, \boldsymbol{\tau}_h \right)
			- \left( \tilde{\mathbf{u}}_h^{n+1}, \nabla \cdot \boldsymbol{\tau}_h \right) =
			- \langle \mathbf{u}_b^{n+1}, \boldsymbol{\tau}_h \, \mathbf{n} \rangle _{\Gamma_d} \quad
			\forall \boldsymbol{\tau}_h \in (\mathbf{V}_{h,0,\Gamma_n})^d, \label{eq:PP1} \\
			&\hskip - .1cm \left( \frac{1}{\phi}\frac{\tilde{\mathbf{u}}_h^{n+1} - \mathbf{u}_h^n}{\Delta t}, \boldsymbol{\xi_h} \right)
			+ \left( \nabla \cdot \bsi_h^{n+1}, \boldsymbol{\xi_h}\right)
			+ \left( \mathbf{q}_h^n, \boldsymbol{\xi_h}\right) + \left(\mathbf{L} \, \tilde{\mathbf{u}}_h^{n+1}, \boldsymbol{\xi_h}\right) = 0 \quad \forall \boldsymbol{\xi_h} \in (W_h)^d.
			\label{eq:PP2}              
		\end{align}
	\end{subequations} 
	
	\item Compute $\Psi_b^{n+1}$ on $\Gamma_n$:
	\begin{equation}
		\Psi_b^{n+1} = \left(\mathbf{\Sigma}_b^{n+1} + \nu\phi^{-1} \nabla \tilde{\mathbf{u}}_h^{n+1} \mathbf{n} \right) \cdot \mathbf{n} \quad \textnormal{on} \quad \Gamma_n.
	\end{equation}
	
	\item MFE discretization of the \PjP\ \eqref{eq:Cs}: 
	
	Find \(\mathbf{u}_h^{n+1} \in \mathbf{V}_h : \mathbf{u}_h^{n+1} \cdot \mathbf{n} = Q_h^\Gamma\left(\mathbf{u}_b \cdot \mathbf{n}\right)\) on \(\Gamma_d\) and \(\Psi_h^{n+1} \in W_h\) \st 
	\begin{subequations} \label{eq:Projecproblem}
		\begin{align}
			&  \bigg( \frac{1}{\phi} \frac{\mathbf{u}_h^{n+1}-\tilde{\mathbf{u}}_h^{n+1}}{\Delta t}, \mathbf{v}_h\bigg) + \left( \mathbf{L} \left(\mathbf{u}_h^{n+1} - \tilde{\mathbf{u}}_h^{n+1}\right), \mathbf{v}_h\right)
			- ( \Psi_h^{n+1} - \Psi_h^n,\nabla \cdot \mathbf{v}_h )
			\notag \\
			&\qquad
			=  - \langle \Psi_b^{n+1} - \Psi_b^n, \mathbf{v}_h \cdot \mathbf{n} \rangle_{\Gamma_n}
			\quad \forall \mathbf{v}_h \in \mathbf{V}_{h,0,\Gamma_d}, \label{eq:Projecproblem_1} \\
			& \left( \nabla \cdot \mathbf{u}_h^{n+1},w_h\right) = 0 \quad \forall w_h \in W_h.
			\label{eq:Projecproblem_2} 
		\end{align}
	\end{subequations}
	
	\item Pressure gradient update. Find \(\mathbf{q}_h^{n+1}\in (W_h)^d\) \st 
	%
		\begin{gather} 
			\left( \mathbf{q}_h^{n+1}, \bxi_h \right) =  \left( \mathbf{q}_h^n, \bxi_h \right)  
			- \left( \frac{1}{\phi} \frac{\mathbf{u}_h^{n+1}-\tilde{\mathbf{u}}_h^{n+1} }{\Delta t} ,\bxi_h \right) - \left(\mathbf{L} \left(\mathbf{u}_h^{n+1}-\tilde{\mathbf{u}}_h^{n+1}\right), \bxi_h \right) \quad \forall \bxi_h \in (W_h)^d. \label{eq:up_q}
		\end{gather}
	
	\item Go to the next time step.
	
\end{itemize}

End do time loop

\begin{remark}\label{rem:div-free}
	The intermediate velocity $\tilde{\mathbf{u}}_h^{n+1} \in (W_h)^d$ computed in the \PP\ \eqref{eq:Predproblem} is discontinuous and the pseudostress $\bsi_h^{n+1} \in (\mathbf{V}_h)^d$ is \(H(\div)\)-conforming. The new pressure $\Psi_h^{n+1} \in W_h$ is computed in the projection step \eqref{eq:Projecproblem} and the new velocity $\bu_h^{n+1} \in \bV_h$ is \(H(\div)\)-conforming and pointwise divergence-free. In particular, \eqref{div-Vh} and \eqref{eq:Projecproblem_2} imply that
	\begin{equation}\label{div-free}
		\nabla \cdot \mathbf{u}_h^{n+1} = 0.
	\end{equation}
	In addition,
	\begin{equation}\label{uh-in-Wh}
		\bu_h^{n+1} \in (W_h)^d.
	\end{equation}  
	This holds trivially for the BDM spaces, since $\bV_h$ and $W_h$ contain polynomials of the same degree, while for the RT spaces it follows from \eqref{div-free} and \cite[Corollary~2.3.1]{BBF}. We further note that \eqref{eq:up_q} gives an approximation of $\nabla \Psi(t_{n+1})$, since, using \eqref{uh-in-Wh} and \eqref{eq:Projecproblem_1}, it gives
	\begin{equation*} 
		\mathbf{q}_h^{n+1} \simeq \mathbf{q}_h^n - \frac{1}{\phi}\frac{\mathbf{u}_h^{n+1}-\tilde{\mathbf{u}}_h^{n+1} }{\Delta t}
		- \mathbf{L} \left(\mathbf{u}_h^{n+1}-\tilde{\mathbf{u}}_h^{n+1}\right)
		\simeq \nabla \Psi(t_n) + \nabla \big(\Psi(t_{n+1})- \Psi(t_n)\big)
		= \nabla \Psi(t_{n+1}).
	\end{equation*}
	Finally, adding \eqref{eq:PP2} and \eqref{eq:up_q} gives
	$$
	\left( \frac{1}{\phi}\frac{\mathbf{u}_h^{n+1} - \mathbf{u}_h^n}{\Delta t}, \boldsymbol{\xi_h} \right)
	+ \left( \nabla \cdot \bsi_h^{n+1}, \boldsymbol{\xi_h}\right)
	+ \left( \mathbf{q}_h^{n+1}, \boldsymbol{\xi_h}\right) + \left(\mathbf{L} \, \mathbf{u}_h^{n+1}, \boldsymbol{\xi_h}\right) = 0 \quad \forall \boldsymbol{\xi_h} \in (W_h)^d,
	$$
	which is the backward Euler approximation of the momentum conservation equation \eqref{eq:momentum} at $t^{n+1}$.
\end{remark}

\begin{remark}
	It is possible to consider a simplified version of the numerical method, where the ``permeability correction term'' $\mathbf{L} \left(\mathbf{u}_h^{n+1} - \tilde{\mathbf{u}}_h^{n+1}\right)$ is not included in \eqref{eq:Projecproblem} and \eqref{eq:up_q}. This also results in a stable and convergent method. However, we illustrate in \cref{tests} that this method produces less accurate results and may result in non-physical solution features with coarse spatial and/or temporal discretizations. 
\end{remark}

\begin{remark}\label{choice}
	With the choice of keeping \(\frac{\nu}{\phi}\) inside the divergence operator in \eqref{eq:momentum}, our formulation is based on the effective viscous stress $\bsi = - \frac{\nu}{\phi} \nabla \mathbf{u}$. Since we approximate $\bsi$ in $H(\div,\Omega)$, we enforce numerically the continuity of $\bsi\bn$ on every finite element face, including on sharp interfaces with discontinuity in $\phi$. This is consistent with the discussion after Eq. (65) in \cite{CHANDESRIS20062137}. Also, our choice leads to a symmetric algebraic system, see Remark~\ref{sym-system}.
\end{remark}

\section{Stability and error analysis} \label{stab_err_an}
In the analysis we consider the case of zero velocity boundary condition on the entire boundary. In addition, to simplify the presentation, we assume that the porosity $\phi$ is a piecewise constant function on the finite element partition. Let
$$
\mathbf{V}_{0} = \left\{\mathbf{v} \in \mathbf{V} : \mathbf{v}\cdot \mathbf{n} = 0 \ \textnormal{on} \ \partial\Omega\right\}, \quad W_0 = \left\{w \in W: \int_\Omega w = 0\right\}.
$$
The restriction of the pressure space is needed to guarantee uniqueness of the pressure $\Psi$. The corresponding MFE spaces are denoted by $\bV_{h,0}$ and $W_{h,0}$. It holds that
\begin{equation}\label{div-Vh0}
	\nabla\cdot \bV_{h,0} = W_{h,0}.
\end{equation}
Algorithm \eqref{eq:Predproblem}--\eqref{eq:up_q} takes the following form.

\begin{itemize}
	
	\item Predictor Problem: Find $\bsi_h^{n+1} \in (\mathbf{V}_h)^d$
	and \(\tilde{\mathbf{u}}_h^{n+1} \in (W_h)^d\) \st
	\begin{subequations} \label{eq:Predproblem-0}
		\begin{align}
			& \left( \frac{\phi}{\nu}  \bsi_h^{n+1}, \boldsymbol{\tau}_h \right)
			- \left( \tilde{\mathbf{u}}_h^{n+1}, \nabla \cdot \boldsymbol{\tau}_h \right)
			= 0  \quad
			\forall \boldsymbol{\tau}_h \in (\mathbf{V}_h)^d, \label{pred1} \\
			& \left( \frac{1}{\phi}\frac{\tilde{\mathbf{u}}_h^{n+1} - \mathbf{u}_h^n}{\Delta t}, \boldsymbol{\xi_h} \right)
			+ \left( \nabla \cdot \bsi_h^{n+1}, \boldsymbol{\xi_h}\right)
			+ \left( \mathbf{q}_h^n, \boldsymbol{\xi_h}\right) + \left(\mathbf{L} \, \tilde{\mathbf{u}}_h^{n+1}, \boldsymbol{\xi_h}\right) = 0 \quad \forall \boldsymbol{\xi_h} \in (W_h)^d.
			\label{pred2}              
		\end{align}
	\end{subequations} 
	
	\item Projection problem: Find $\mathbf{u}_h^{n+1} \in \mathbf{V}_{h,0}$
	and \(\Psi_h^{n+1} \in W_{h,0}\) \st 
	\begin{subequations} \label{eq:Projecproblem-0}
		\begin{align}
			&  \bigg( \frac{1}{\phi} \frac{\mathbf{u}_h^{n+1}-\tilde{\mathbf{u}}_h^{n+1}}{\Delta t}, \mathbf{v}_h\bigg)
			+ \left( \mathbf{L} \left(\mathbf{u}_h^{n+1}-\tilde{\mathbf{u}}_h^{n+1}\right), \mathbf{v}_h\right) - ( \Psi_h^{n+1} - \Psi_h^n,\nabla \cdot \mathbf{v}_h ) = 0
			\quad \forall \mathbf{v}_h \in \mathbf{V}_{h,0}, \label{eq:Projecproblem_1-0} \\
			& \left( \nabla \cdot \mathbf{u}_h^{n+1},w_h\right) = 0 \quad \forall w_h \in W_{h,0}.
			\label{eq:Projecproblem_2-0} 
		\end{align}
	\end{subequations}
	
	\item Pressure gradient update. Find \(\mathbf{q}_h^{n+1}\in (W_h)^d\) \st 
	%
		\begin{gather} 
			\left( \mathbf{q}_h^{n+1}, \bxi_h \right) =  \left( \mathbf{q}_h^n, \bxi_h \right)  
			- \left( \frac{1}{\phi} \frac{\mathbf{u}_h^{n+1}-\tilde{\mathbf{u}}_h^{n+1} }{\Delta t} ,\bxi_h \right)
			- \left(\mathbf{L} \left(\mathbf{u}_h^{n+1} - \tilde{\mathbf{u}}_h^{n+1}\right), \bxi_h \right)
			\quad \forall \bxi_h \in (W_h)^d. \label{eq:up_q-0}
		\end{gather}
		
	\end{itemize}
	
	\subsection{Stability analysis}
	Let $\bV^0 = \{\bv \in \bV: \nabla\cdot\bv = 0\}$ mark the divergence-free subspace of $\bV$, using a similar notation for $\bV_h^0$ and $\bV_{h,0}^0$. The following orthogonality property will be used in the analysis.
	
	\begin{lemma}
		For $\bq_h^{n+1}$ computed in \eqref{eq:up_q-0}, it holds that
		\begin{equation}\label{q-orth}
			(\bq_h^{n+1},\bv_h) = 0 \quad \forall \bv_h \in \bV_{h,0}^0.
		\end{equation}
	\end{lemma}
	
	\begin{proof}
		Using the argument for \eqref{uh-in-Wh}, we have that $\bV_h^0 \subset (W_h)^d$, so
		we can test \eqref{eq:up_q-0} with $\bxi_h = \bv_h \in \bV_{h,0}^0$, which, together with \eqref{eq:Projecproblem_1-0}, implies
		\begin{align*}
			(\bq_h^{n+1},\bv_h) & = (\bq_h^{n},\bv_h) = \ldots = (\bq_h^{0},\bv_h) = (Q_h^d\, \nabla \Psi_0,\bv_h)
			\\
			&
			= (\nabla \Psi_0,\bv_h) = - (\Psi_0,\nabla\cdot\bv_h) + \<\Psi_0,\bv_h\cdot\bn\>_{\Gamma} = 0.
		\end{align*}
	\end{proof}
	
	The following algebraic identity will be used in the analysis,
	\begin{equation}\label{a-bb}
		a(a-b) = \frac12\left(a^2 - b^2 + (a-b)^2\right)
	\end{equation}
	and Young's inequality, for any $\epsilon > 0$,
	\begin{equation}\label{young}
		ab \le \frac{\epsilon}{2}a^2 + \frac{1}{2\epsilon}b^2.
	\end{equation}
	
	\begin{theorem}\label{thm:stab}
		For the method \eqref{eq:Predproblem-0}--\eqref{eq:up_q-0}, there exists a constant $C$ independent of $\Delta t$ and $h$ such that
		\begin{align}
			& \Delta t \sum_{n=0}^{N-1} \nu^{-1}\|\phi^{1/2}\bsi_h^{n+1}\|^2 + \|\phi^{-1/2}\bu_h^N\|^2
			+ \Delta t\sum_{n=0}^{N-1} \|\bL^{1/2}\bu_h^{n+1}\|^2
			+ \Delta t\sum_{n=0}^{N-1} \|\bL^{1/2}\tilde\bu_h^{n+1}\|^2
			\nonumber \\
			& \qquad    
			+ \sum_{n=0}^{N-1} \|\phi^{-1/2}(\bu_h^{n+1} - \tilde\bu_h^{n+1})\|^2 
			+ \Delta t^2\|\phi^{1/2}\bq_h^N\|^2
			\le C\left(\|\phi^{-1/2}\bu_h^0\|^2 + \Delta t^2\|\phi^{1/2}\bq_h^0\|^2\right). \label{eq:stab}
		\end{align}
	\end{theorem}
	
	\begin{proof}
		We take test functions $(\btau_h,\bxi_h) = (\bsi_h^{n+1},\tilde\bu_h^{n+1})$ in \eqref{eq:Predproblem-0}, combine the equations, use \eqref{a-bb}, and multiply by $\Delta t$, to obtain
		\begin{align}
			& \Delta t \, \nu^{-1}\|\phi^{1/2}\bsi_h^{n+1}\|^2
			+ \frac12\|\phi^{-1/2}\tilde\bu_h^{n+1}\|^2 - \frac12\|\phi^{-1/2}\bu_h^{n}\|^2
			+ \frac12\|\phi^{-1/2}(\tilde\bu_h^{n+1} - \bu_h^{n})\|^2 \nonumber \\
			& \qquad
			+ \Delta t\|\bL^{1/2}\tilde\bu_h^{n+1}\|^2
			+ \Delta t \, (\bq_h^n,\tilde\bu_h^{n+1}) = 0. \label{eq:stab-1}
		\end{align}
		Next, we take $\bv_h = \bu_h^{n+1}$ in \eqref{eq:Projecproblem_1-0} and use \eqref{a-bb} and \eqref{div-free}, obtaining
		\begin{align}
			& \frac12\|\phi^{-1/2}\bu_h^{n+1}\|^2 - \frac12\|\phi^{-1/2}\tilde\bu_h^{n+1}\|^2
			+ \frac12\|\phi^{-1/2}(\bu_h^{n+1} - \tilde\bu_h^{n+1})\|^2 \nonumber \\
			& \quad
			+ \frac{\Delta t}{2}\|\bL^{1/2}\bu_h^{n+1}\|^2
			- \frac{\Delta t}{2}\|\bL^{1/2}\tilde\bu_h^{n+1}\|^2
			+ \frac{\Delta t}{2}\|\bL^{1/2}(\bu_h^{n+1} - \tilde\bu_h^{n+1})\|^2 = 0. \label{eq:stab-2}
		\end{align}
		Taking $\bxi_h = \phi\bq_h^{n+1}$ in \eqref{eq:up_q-0}, which is a valid choice, since $\phi$ is assumed to be piecewise constant on the mesh, using \eqref{a-bb} and multiplying by $\Delta t^2$ results in
		\begin{align}
			& \frac{\Delta t^2}{2}\left(\|\phi^{1/2}\bq_h^{n+1}\|^2 - \|\phi^{1/2}\bq_h^{n}\|^2
			+ \|\phi^{1/2}(\bq_h^{n+1} - \bq_h^{n})\|^2\right)
			\nonumber \\
			& \qquad
			= \Delta t(\tilde{\mathbf{u}}_h^{n+1},\bq_h^{n+1})
			- \Delta t^2 \left(\bL(\bu_h^{n+1} - \tilde\bu_h^{n+1}),\phi\bq_h^{n+1}\right),
			\label{eq:stab-3}
		\end{align}
		where we used \eqref{q-orth} for the first term on the r.h.s.. Summing \eqref{eq:stab-1}--\eqref{eq:stab-3}, we obtain
		\begin{align}
			& \Delta t \, \nu^{-1}\|\phi^{1/2}\bsi_h^{n+1}\|^2 + \frac12\left(\|\phi^{-1/2}\bu_h^{n+1}\|^2 - \|\phi^{-1/2}\bu_h^{n}\|^2\right) \nonumber \\
			& \qquad
			+ \frac12\left(\|\phi^{-1/2}(\tilde\bu_h^{n+1} - \bu_h^{n})\|^2
			+ \|\phi^{-1/2}(\bu_h^{n+1} - \tilde\bu_h^{n+1})\|^2\right)
			\nonumber \\
			& \qquad
			+ \frac{\Delta t}{2}\|\bL^{1/2}\bu_h^{n+1}\|^2
			+ \frac{\Delta t}{2}\|\bL^{1/2}\tilde\bu_h^{n+1}\|^2
			+ \frac{\Delta t}{2}\|\bL^{1/2}(\bu_h^{n+1} - \tilde\bu_h^{n+1})\|^2 
			\nonumber \\  
			& \qquad + \frac{\Delta t^2}{2}\left(\|\phi^{1/2}\bq_h^{n+1}\|^2 - \|\phi^{1/2}\bq_h^{n}\|^2
			+ \|\phi^{1/2}(\bq_h^{n+1} - \bq_h^{n})\|^2\right)
			\nonumber \\
			& \quad
			= \Delta t\, (\bq_h^{n+1} - \bq_h^n,\tilde\bu_h^{n+1})
			- \Delta t^2 \left(\bL(\bu_h^{n+1} - \tilde\bu_h^{n+1}),\phi\bq_h^{n+1}\right).
			\label{eq:stab-4}
		\end{align}
		For the first term on the r.h.s. above, using \eqref{q-orth} and \eqref{young}, we write
		\begin{align}
			\Delta t\, (\bq_h^{n+1} - \bq_h^n,\tilde\bu_h^{n+1})
			& = \Delta t\, (\bq_h^{n+1} - \bq_h^n,\tilde\bu_h^{n+1} - \bu_h^n) \nonumber \\
			& \le \frac{\Delta t^2}{2}\|\phi^{1/2}(\bq_h^{n+1} - \bq_h^{n})\|^2
			+ \frac12\|\phi^{-1/2}(\tilde\bu_h^{n+1} - \bu_h^n)\|^2.
		\end{align}\label{eq:stab-5}
		For the second term on the r.h.s. of \eqref{eq:stab-4}, using \eqref{young}, we obtain
		\begin{equation}\label{eq:stab-6}
			- \Delta t^2 \left(\bL(\bu_h^{n+1} - \tilde\bu_h^{n+1}),\phi\bq_h^{n+1}\right)
			\le \frac{\Delta t}{2}\|\bL^{1/2}(\bu_h^{n+1} - \tilde\bu_h^{n+1})\|^2
			+ \frac{\Delta t^3 (\phi\bL)_{\max}}{2}\|\phi^{1/2}\bq_h^{n+1}\|^2,
		\end{equation}
		where $(\phi\bL)_{\max}$ is the spatial supremum of the largest eigenvalue of $\phi\bL$, 
Bound \eqref{eq:stab} follows by combining \eqref{eq:stab-4}--\eqref{eq:stab-6}, summing over $0 \le n \le N-1$, and applying the discrete Gronwall inequality \cite[Lemma~1.4.2]{QV-book} for the last term in \eqref{eq:stab-6}, under the assumption that $\Delta t \le 1/(2(\phi\bL)_{\max})$.
	\end{proof}
	
	\subsection{Stability of the simplified method}
	We next give a stability bound for the simplified method without the permeability correction term. In this case equations \eqref{eq:Projecproblem_1-0} and \eqref{eq:up_q-0} are replaced by, respectively,
	\begin{align}
		&  \bigg( \frac{1}{\phi} \frac{\mathbf{u}_h^{n+1}-\tilde{\mathbf{u}}_h^{n+1}}{\Delta t}, \mathbf{v}_h\bigg)
		- ( \Psi_h^{n+1} - \Psi_h^n,\nabla \cdot \mathbf{v}_h ) = 0
		\quad \forall \mathbf{v}_h \in \mathbf{V}_{h}, \label{eq:Projecproblem_1-simpl} \\
		& \left( \mathbf{q}_h^{n+1}, \bxi_h \right) =  \left( \mathbf{q}_h^n, \bxi_h \right)  
		- \left( \frac{1}{\phi} \frac{\mathbf{u}_h^{n+1}-\tilde{\mathbf{u}}_h^{n+1} }{\Delta t} ,\bxi_h \right)
		\quad \forall \bxi_h \in (W_h)^d. \label{eq:up_q-simpl}
	\end{align}
	
	The proof of the following theorem is similar to the proof of Theorem~\ref{thm:stab}, with a simplified treatment of the Brinkman terms, and it is omitted for sake of space.
	
	\begin{theorem}\label{thm:stab-simpl}
		For the simplified method \eqref{eq:Predproblem-0}, \eqref{eq:Projecproblem_1-simpl}--\eqref{eq:Projecproblem_2-0}, \eqref{eq:up_q-simpl}, it holds that
		\begin{align}
			& 2\Delta t \sum_{n=0}^{N-1} \nu^{-1}\|\phi^{1/2}\bsi_h^{n+1}\|^2 + \|\phi^{-1/2}\bu_h^N\|^2
			+ 2\Delta t\sum_{n=0}^{N-1} \|\bL^{1/2}\tilde\bu_h^{n+1}\|^2
			\nonumber \\
			& \qquad  
			+ \sum_{n=0}^{N-1} \|\phi^{-1/2}(\bu_h^{n+1} - \tilde\bu_h^{n+1})\|^2 
			+ \Delta t^2\|\phi^{1/2}\bq_h^N\|^2
			\le \|\phi^{-1/2}\bu_h^0\|^2 + \Delta t^2\|\phi^{1/2}\bq_h^0\|^2. \label{eq:stab-simpl}
		\end{align}
	\end{theorem}
	
\subsection{Analysis of the time discretization error}
We proceed with deriving a bound on the time discretization error. To this end, the semi-discrete continuous-in-space formulation of the method
\eqref{eq:Predproblem-0}--\eqref{eq:up_q-0} is considered, which is stated below.
	
	\begin{itemize}
		
		\item Initialization step: Let $\bu^0 = \bu_0$, $\Psi^0 = \Psi_0$, $\bq^0 = \nabla \Psi_0$.
		
		For \(0 \le n \le N-1\):
		
		\item  Predictor problem: Find $\bsi^{n+1} \in (\bV)^d$  and \(\tilde{\mathbf{u}}^{n+1} \in (W)^d\) s.t.
		\begin{subequations} \label{eq:varPredproblem}
			\begin{align}
				&\left( \frac{\phi}{\nu} \bsi^{n+1}, \boldsymbol{\tau} \right)
				- \left( \tilde{\mathbf{u}}^{n+1}, \nabla \cdot \boldsymbol{\tau} \right) = 0
				\quad
				\forall \boldsymbol{\tau} \in (\mathbf{V})^d, \label{var1}\\
				&\left( \frac{1}{\phi}\frac{\tilde{\mathbf{u}}^{n+1} - \mathbf{u}^n}{\Delta t}, \boldsymbol{\xi} \right)
				+ \left( \nabla \cdot \bsi^{n+1}, \boldsymbol{\xi}\right)
				+ \left( \mathbf{q}^n, \boldsymbol{\xi}\right)
				+ \left(\bL\tilde\bu^{n+1},\bxi\right)
				= 0 \quad \forall \boldsymbol{\xi} \in (W)^d. \label{var2}
			\end{align}
		\end{subequations}
		
		\item Projection problem: Find $\mathbf{u}^{n+1} \in \mathbf{V}_0$
		and \(\Psi^{n+1} \in W_0\) s.t.
		\begin{subequations} \label{eq:varProjecproblem}
			\begin{align}
				&  \bigg( \frac{1}{\phi} \frac{\mathbf{u}^{n+1}-\tilde{\mathbf{u}}^{n+1}}{\Delta t}, \mathbf{v}\bigg)
				+ \left( \mathbf{L} \left(\mathbf{u}^{n+1}-\tilde{\mathbf{u}}^{n+1}\right), \mathbf{v}\right)
				- ( \Psi^{n+1} - \Psi^n,\nabla \cdot \mathbf{v} ) = 0
				\quad \forall \mathbf{v} \in \mathbf{V}_{0}, \label{eq:varProjecproblem_1} \\
				& \left( \nabla \cdot \mathbf{u}^{n+1},w\right) = 0 \quad \forall w \in W_0.
				\label{eq:varProjecproblem_2} 
			\end{align}
		\end{subequations}
		
		\item Update the pressure gradient: Find \(\mathbf{q}^{n+1}\in (W)^d\) s.t.
		\begin{gather} 
			\left( \mathbf{q}^{n+1}, \bxi \right) =  \left( \mathbf{q}^n, \bxi \right)
			- \left( \frac{1}{\phi}\frac{\mathbf{u}^{n+1}-\tilde{\mathbf{u}}^{n+1} }{\Delta t} ,\bxi \right)
			- \left(\mathbf{L} \left(\mathbf{u}^{n+1} - \tilde{\mathbf{u}}^{n+1}\right), \bxi \right)
			\quad \forall \bxi \in (W)^d. \label{eq:var-up_q}
		\end{gather}
		
	\end{itemize}
	
	As noted in Remark~\ref{rem:div-free}, \eqref{eq:varProjecproblem_2} implies that
	\begin{equation}\label{div-free-0}
		\nabla \cdot \mathbf{u}^{n+1} = 0.
	\end{equation}
	
On the other hand, the solution to the model problem \eqref{eq:governing_Eqq} with boundary condition $\bu = 0$ on $\Gamma$ satisfies, for \(0 \le n \le N-1\),
	\begin{align}
		&\left( \frac{\phi}{\nu} \boldsymbol{\sigma}(t_{n+1}), \boldsymbol{\tau} \right)
		- \left( \mathbf{u}(t_{n+1}), \nabla \cdot \boldsymbol{\tau} \right) = 0
		\quad \forall \boldsymbol{\tau} \in (\mathbf{V})^d, \label{true-1}\\
		& \left( \frac{1}{\phi}\frac{\mathbf{u}(t_{n+1}) - \mathbf{u}(t_n)}{\Delta t}, \boldsymbol{\xi} \right)
		+ \left( \nabla \cdot \boldsymbol{\sigma}(t_{n+1}), \boldsymbol{\xi}\right)
		+ \left( \nabla \Psi(t_{n+1}), \boldsymbol{\xi}\right)
		+ (\bL\bu(t_{n+1}),\bxi) \nonumber \\
		& \qquad = (T_{n+1}(\bu),\bxi) \quad \forall \boldsymbol{\xi} \in (W)^d, \label{true-2}\\
		& \nabla\cdot \bu(t_{n+1}) = 0, \label{true-3}
	\end{align}
	where
	$$
	T_{n+1}(\bu) = \frac{\mathbf{u}(t_{n+1}) - \mathbf{u}(t_n)}{\Delta t} - \frac{\partial \bu}{\partial t}(t_{n+1}).
	$$

	\begin{theorem}\label{thm:time-err}
Assuming that the true solution is sufficiently smooth in time, there exists a constant $C$ independent of $\Delta t$ such that the solution to the semi-discrete method \eqref{eq:varPredproblem}--\eqref{eq:var-up_q} satisfies
\begin{align}
			& \Delta t \sum_{n=0}^{N-1}\nu^{-1}\|\phi^{1/2}(\bsi(t_{n+1}) - \bsi^{n+1})\|^2 + \|\phi^{-1/2}(\bu(t_{N}) - \bu^N)\|^2
			+ \Delta t^2\|\phi^{1/2}(\nabla \Psi(t_N) - \bq^N)\|^2
			\nonumber \\
			& \qquad
			+ \Delta t \sum_{n=0}^{N-1}\left(\|\bL^{1/2}(\bu(t_{n+1}) - \bu^{n+1})\| + \|\bL^{1/2}(\bu(t_{n+1}) - \tilde\bu^{n+1})\|  \right)
			\le C \Delta t^2.\label{err-bound}
		\end{align}
		
	\end{theorem}

	\begin{proof}
		Let $\esi^{n} = \bsi(t_{n}) - \bsi^{n}$, $\etu^n = \bu(t_{n}) - \tilde\bu^n$, $\eu^n = \bu(t_{n}) - \bu^n$, and $\eq^n = \nabla \Psi(t_n) - \bq^n$.
		Subtracting \eqref{var1}--\eqref{var2}, \eqref{div-free-0} from \eqref{true-1}--\eqref{true-3} results in the error equations
		\begin{align}
			&\left( \frac{\phi}{\nu} \esi^{n+1}, \boldsymbol{\tau} \right)
			- \left( \etu^{n+1}, \nabla \cdot \boldsymbol{\tau} \right) = 0
			\quad \forall \boldsymbol{\tau} \in (\mathbf{V})^d, \label{err-1}\\
			& \left( \frac{1}{\phi}\frac{\etu^{n+1} - \eu^n}{\Delta t}, \boldsymbol{\xi} \right)
			+ \left( \nabla \cdot \esi^{n+1}, \boldsymbol{\xi}\right)
			+ \left( \eq^{n}, \boldsymbol{\xi}\right)
			+ (\bL\etu^{n+1},\bxi) \nonumber \\
			& \qquad = (T_{n+1}(\bu),\bxi) - (S_{n+1}(\nabla\Psi),\bxi) \quad \forall \boldsymbol{\xi} \in (W)^d, \label{err-2}\\
			& \nabla\cdot\eu^{n+1} = 0, \label{err-3}
		\end{align}
		where
		$$
		S_{n+1}(\nabla\Psi) = \nabla\Psi(t_{n+1}) - \nabla\Psi(t_{n}).
		$$
		Taking $(\btau,\bxi) = (\esi^{n+1},\etu^{n+1})$ in \eqref{err-1}--\eqref{err-2}, combining the equations, using \eqref{a-bb}, and multiplying by $\Delta t$, we get
		\begin{align}
			& \Delta t \, \nu^{-1}\|\phi^{1/2}\esi^{n+1}\|^2
			+ \frac12\|\phi^{-1/2}\etu^{n+1}\|^2 - \frac12\|\phi^{-1/2}\eu^{n}\|^2
			+ \frac12\|\phi^{-1/2}(\etu^{n+1} - \eu^{n})\|^2
			\nonumber \\
			& \qquad
			+ \Delta t \, (\eq^n,\etu^{n+1})
			+ \|\bL^{1/2}\etu^{n+1}\|^2
			= \Delta t(T_{n+1}(\bu),\etu^{n+1}) - \Delta t(S_{n+1}(\nabla\Psi),\etu^{n+1}). \label{eq:err-1}
		\end{align}
Subtracting and adding $\bu(t_{n+1})$ in \eqref{eq:varProjecproblem_1} and taking $\bv = \eu^{n+1}$, we obtain
		$$
		(\phi^{-1}(\eu^{n+1} - \etu^{n+1}),\eu^{n+1}) + \Delta t\left(\bL(\eu^{n+1} - \etu^{n+1}),\eu^{n+1}\right) = \Delta t (\Psi^{n+1} - \Psi^n,\nabla\cdot\eu^{n+1}),
		$$
		which implies, using \eqref{err-3} and \eqref{a-bb},
		\begin{align}
			& \frac12\|\phi^{-1/2}\eu^{n+1}\|^2 - \frac12\|\phi^{-1/2}\etu^{n+1}\|^2
			+ \frac12\|\phi^{-1/2}(\eu^{n+1} - \etu^{n+1})\|^2 \nonumber \\
			& \qquad
			+ \frac{\Delta t}{2}\|\bL^{1/2}\eu^{n+1}\|^2 - \frac{\Delta t}{2}\|\bL^{1/2}\etu^{n+1}\|^2
			+ \frac{\Delta t}{2}\|\bL^{1/2}(\eu^{n+1} - \etu^{n+1})\|^2\label{eq:err-2}
			= 0.
		\end{align}
We subtract and add $\nabla \Psi(t_{n+1})$, $\nabla \Psi(t_{n})$, and $\bu(t_{n+1})$
in \eqref{eq:var-up_q}, multiply by $\Delta t^2$, and take $\bxi = \phi\eq^{n+1}$, obtaining
		\begin{align}
			& \Delta t^2(\eq^{n+1} - \eq^n,\phi\eq^{n+1})
			+ \Delta t \left(\eu^{n+1} - \etu^{n+1},\eq^{n+1}\right)
			+ \Delta t^2\left(\bL(\eu^{n+1}-\etu^{n+1}),\phi\eq^{n+1}\right) \nonumber \\
			& \qquad = \Delta t^2(S_{n+1}(\nabla\Psi),\phi\eq^{n+1}). \label{eq:err-3a}
		\end{align}
We note that the argument for \eqref{q-orth} implies
		$$
		(\bq^{n},\bv) = 0 \quad \forall \bv \in \bV_{0}^0.
		$$
Combining the above equation with
		\begin{equation}\label{grad-orth}
			(\nabla \Psi,\bv) = -(\Psi,\nabla\cdot\bv) + \<\Psi,\bv\cdot\bn\>_{\Gamma} = 0 \quad \forall \bv \in \bV_{0}^0,
		\end{equation}
		implies
		\begin{equation}\label{q-orth-err}
			(\eq^{n},\bv) = 0 \quad \forall \bv \in \bV_{0}^0.
		\end{equation}
		Therefore, using that $\eu^{n+1} \in \bV_{0}^0$, cf. \eqref{err-3}, as well as \eqref{a-bb}, \eqref{eq:err-3a} results in
		\begin{align}
			&  \frac{\Delta t^2}{2}\left(\|\phi^{1/2}\eq^{n+1}\|^2 - \|\phi^{1/2}\eq^{n}\|^2
			+ \|\phi^{1/2}(\eq^{n+1} - \eq^n)\|^2\right)
			- \Delta t (\etu^{n+1},\eq^{n+1})
			\nonumber \\
			& \qquad
			+ \Delta t^2\left(\bL(\eu^{n+1}-\etu^{n+1}),\phi\eq^{n+1}\right)
			= \Delta t^2(S_{n+1}(\nabla\Psi),\phi\eq^{n+1})\label{eq:err-3}.
		\end{align}
		The next step is to sum \eqref{eq:err-1}, \eqref{eq:err-2}, and \eqref{eq:err-3}. For the sum of the next-to-last terms on l.h.s. of \eqref{eq:err-1} and \eqref{eq:err-3}, using \eqref{q-orth-err}, we write
		\begin{equation}\label{eq:err-4}
			-(\etu^{n+1},\eq^{n+1} - \eq^n) = (\eu^{n} - \etu^{n+1},\eq^{n+1} - \eq^n).
		\end{equation}
		Summing \eqref{eq:err-1}, \eqref{eq:err-2}, and \eqref{eq:err-3} and using \eqref{eq:err-4}, we obtain
		\begin{align}
			& \Delta t \, \nu^{-1}\|\phi^{1/2}\esi^{n+1}\|^2 + \frac12\Big(\|\phi^{-1/2}\eu^{n+1}\|^2
			- \|\phi^{-1/2}\eu^{n}\|^2
			+ \|\phi^{-1/2}(\etu^{n+1} - \eu^{n})\|^2
			\nonumber \\
			& \qquad
			+ \|\phi^{-1/2}(\eu^{n+1} - \etu^{n+1})\|^2 \Big)
			+ \frac{\Delta t}{2}\|\bL^{1/2}\eu^{n+1}\|^2
			+ \frac{\Delta t}{2}\|\bL^{1/2}\etu^{n+1}\|^2
			+ \frac{\Delta t}{2}\|\bL^{1/2}(\eu^{n+1} - \etu^{n+1})\|^2
			\nonumber \\
			& \qquad + \frac{\Delta t^2}{2}\left(\|\phi^{1/2}\eq^{n+1}\|^2 - \|\phi^{1/2}\eq^{n}\|^2
			+ \|\phi^{1/2}(\eq^{n+1} - \eq^n)\|^2\right)
			\nonumber \\
			&\quad = - \Delta t (\eu^{n} - \etu^{n+1},\eq^{n+1} - \eq^n)
			- \Delta t^2\left(\bL(\eu^{n+1}-\etu^{n+1}),\phi\eq^{n+1}\right)
			+ \Delta t(T_{n+1}(\bu),\etu^{n+1})
			\nonumber \\
			&\qquad \,
			- \Delta t(S_{n+1}(\nabla\Psi),\etu^{n+1})
			+ \Delta t^2(S_{n+1}(\nabla\Psi),\phi\eq^{n+1})
			=: I_1 + I_2 + I_3 + I_4 + I_5. \label{eq:err-5}
		\end{align}
We next bound the four terms on the r.h.s.. For $I_1$, using \eqref{young}, we have
		\begin{equation}\label{I1}
			|I_1| \le \frac12\|\phi^{-1/2}(\eu^{n} - \etu^{n+1})\|^2
			+ \frac{\Delta t^2}{2}\|\phi^{1/2}(\eq^{n+1} - \eq^n)\|^2.
		\end{equation}
		Using \eqref{young} for $I_2$ gives
		\begin{equation}\label{I2}
			|I_2| \le \frac{\Delta t}{2}\|\bL^{1/2}(\eu^{n+1} - \etu^{n+1})\|^2
			+ \frac{\Delta t^3 (\phi\bL)_{\max}}{2}\|\phi^{1/2}\eq^{n+1}\|^2.
		\end{equation}
		For $I_3$, using \eqref{young}, we write
		\begin{align}
			|I_3| & = |\Delta t(T_{n+1}(\bu),\etu^{n+1} - \eu^{n+1}) + \Delta t(T_{n+1}(\bu),\eu^{n+1})| \nonumber \\
			& \le \Delta t^2\|\phi^{1/2}T_{n+1}(\bu)\|^2
			+ \frac14\|\phi^{-1/2}(\etu^{n+1} - \eu^{n+1})\|^2
			+ \Delta t\|\phi^{1/2}T_{n+1}(\bu)\|^2
			+ \frac{\Delta t}{4}\|\phi^{-1/2}\eu^{n+1}\|^2.\label{I3}
		\end{align}
		For $I_4$, using \eqref{grad-orth} and \eqref{err-3}, we have that $(S_{n+1}(\nabla\Psi),\eu^{n+1}) = 0$. Then, using \eqref{young}, we obtain
		\begin{equation}\label{I4}
			|I_4| = |\Delta t(S_{n+1}(\nabla\Psi),\etu^{n+1} - \eu^{n+1})| \le \Delta t^2 \|\phi^{1/2}S_{n+1}(\nabla\Psi)\|^2 + \frac14 \|\phi^{-1/2}(\etu^{n+1} - \eu^{n+1})\|^2.
		\end{equation}
		For $I_5$, the use of \eqref{young}  gives
		\begin{equation}\label{I5}
			|I_5| \le \Delta t \|\phi^{1/2}S_{n+1}(\nabla\Psi)\|^2 + \frac{\Delta t^3}{4}\|\phi^{1/2}\eq^{n+1}\|^2.
		\end{equation}
		Combining \eqref{eq:err-5}--\eqref{I4} gives
		\begin{align}
			& \Delta t \, \nu^{-1}\|\phi^{1/2}\esi^{n+1}\|^2
			+ \frac12\left(\|\phi^{-1/2}\eu^{n+1}\|^2 - \|\phi^{-1/2}\eu^{n}\|^2\right)
			+ \frac{\Delta t}{2}\|\bL^{1/2}\eu^{n+1}\|^2
			+ \frac{\Delta t}{2}\|\bL^{1/2}\etu^{n+1}\|^2 \nonumber \\
			& \qquad
			+ \frac{\Delta t^2}{2}\left(\|\phi^{1/2}\eq^{n+1}\|^2 - \|\phi^{1/2}\eq^{n}\|^2\right)
			\nonumber \\
			&\quad \le \left(\Delta t^2 + \Delta t\right)\|\phi^{1/2}T_{n+1}(\bu)\|^2
			+ \left(\Delta t^2 + \Delta t \right)\|\phi^{1/2}S_{n+1}(\nabla\Psi)\|^2 \nonumber \\
			& \qquad
			+ \frac{\Delta t}{4}\|\phi^{-1/2}\eu^{n+1}\|^2
			+ \frac{\Delta t^3 (\phi\bL)_{\max}}{2}\|\phi^{1/2}\eq^{n+1}\|^2
			+ \frac{\Delta t^3}{4}\|\phi^{1/2}\eq^{n+1}\|^2. \label{eq:err-6}
		\end{align}
Under the assumption that the solution is sufficiently smooth in time, one can easily prove for the time discretization and splitting errors,  that $\|T_{n+1}(\bu)\| \le C\Delta t$ and $\|S_{n+1}(\nabla\Psi)\| \le C \Delta t$. Then, summing \eqref{eq:err-6} over $n$ from 0 to $N-1$, using that $\eu^0 = 0$ and $\eq^0 = 0$, and, under the assumption $\Delta t \le \min\{1,1/(4(\phi\bL)_{\max})\}$, applying the discrete Gronwall inequality for the last three terms in \eqref{eq:err-6}, we obtain \eqref{err-bound}.
	\end{proof}
	
\section{Implementation of a second order multipoint flux MFE method} \label{implem}
In this section we present the details of the implementation of a second order version of the MFE projection method introduced in \cref{ge,numalg}. We consider \(\mathbf{V}_h \times W_h \) on \(T_h\) (see \cref{discr}) to be the Raviart--Thomas pair of spaces $RT_1$ \cite{R-T} on triangular or tetrahedral grids. We apply the MFMFE methodology originally proposed in \cite{W-Y} as a first order scheme, and further extended in \cite{Radu} to the second order by using the $RT_1$ spaces.
Using a quadrature rule whose nodes are associated with the degrees of freedom (\DOFs) of the \(RT_1\) space $\bV_h$, we get mass lumping in both the \PP\ and \PjP.
	In the \PP, the viscous stress $\bsi_h^{n+1}$  is locally eliminated and a \spd\ system is solved for each component of the velocity $\tilde{\mathbf{u}}_h^{n+1}$. In the \PjP, the corrected velocity \(\bu_h^{n+1}\) is locally eliminated and a \spd\ system is solved for the pressure \(\Psi_h^{n+1}\). A local post-processing easily allows to recover the viscous stress and corrected velocity. \par
	
	The proposed method results in a very efficient computational algorithm. As we show in this section, the solution of \(\left(d+1\right)\) \spd\ linear systems is required at each time step: \(d\) systems for the calculation of the predicted velocity components in the \PP\ and one system for the pressure in the \PjP. Each system involves \(\left(d+1\right)\) unknowns per simplex, with a significant reduction in the number of unknowns compared to the original saddle point systems. The velocity obtained at the end of each time iteration is second order accurate, $H(\mbox{div})$-conforming, pointwise divergence-free, and linear on each simplex of the grid. \par

	\subsection{The $RT_1$ mixed finite element spaces}
	Let \(\hat{E}\) be the reference simplex with vertices $\hat\br_1 = (0,0)^T$, $\hat\br_2 = (1,0)^T$, $\hat\br_3 = (0,1)^T$ if $d=2$ and $\hat\br_1 = (0,0,0)^T$, $\hat\br_2 = (1,0,0)^T$, $\hat\br_3 = (0,1,0)^T$, $\hat\br_4 = (0,0,1)^T$ if $d=3$. For any physical simplex \(E\in T_h\) with vertices \(\mathbf{r}_i\), $i = 1,\ldots,d+1$, such that if \(d = 2\), \(\mathbf{r}_i\) are anti-clockwise oriented, or if \(d = 3\), \(\mathbf{r}_i\), $i = 1,2,3$ are anti-clockwise oriented with respect to \(\mathbf{r}_4\), see \cref{basis_functs}. A multilinear mapping \(F_E : \hat{E}\rightarrow E \) exists such that
	\begin{equation} \label{mapping}
		F_E =  \mathbf{r}_1  \left(1 - \hat{x} - \hat{y} \right) + \mathbf{r}_2  \hat{x} + \mathbf{r}_3  \hat{y}  \textnormal{ if }  d=2, \quad F_E =  \mathbf{r}_1  \left(1 - \hat{x} - \hat{y} - \hat{z} \right) + \mathbf{r}_2  \hat{x} + \mathbf{r}_3  \hat{y} + \mathbf{r}_4  \hat{z}  \textnormal{ if }  d=3.
	\end{equation}  
	The \(RT_1\) spaces are defined on the reference simplex \(\hat{E}\) as
	\begin{equation*}
		\mathbf{V}(\hat{E}) = \big(P_1(\hat{E})\big)^d + \mathbf{x}P_1(\hat{E}), \quad
		W(\hat{E}) = P_1(\hat{E}),
	\end{equation*}
	where \(P_k(\hat{E})\) is the space of multivariate polynomials of degree \(\le k\) on \(\hat{E}\). The dimension of $\mathbf{V}(\hat E)$ is 8 if \(d = 2\) and 15 if \(d = 3\). Let \(\hat{\mathbf{r}}_{d+2}\) be the center of mass of \(\hat{E}\). We consider \(d\) degrees of freedom (\DOFs) associated with each face \(\hat e\), which are the normal components at the vertices of \(\hat e\), 
	and \(d\) \DOFs\ associated with the center of mass.
	Let \(\hat{\mathbf{n}}_{k,l}\), with \(k = 1, \dots, \left(d+1\right)\) and \(l = 1, \dots, d\), be the unit outward normal vectors to the \(d\) faces sharing vertex \(\hat{\mathbf{r}}_k\). The \(d\) vectors  \(\hat{\mathbf{n}}_{d+2,l}\) associated with the midpoint of \(\hat{E}\) are \(\hat{\mathbf{n}}_{d+2,1} = \left(1 \ 0 \right)^T\) and \(\hat{\mathbf{n}}_{d+2,2} = \left(0 \ 1 \right)^T\), if \(d=2\), or \(\hat{\mathbf{n}}_{d+2,1} = \left(1 \ 0 \ 0\right)^T\), \(\hat{\mathbf{n}}_{d+2,2} = \left(0 \ 1 \ 0\right)^T\) and \(\hat{\mathbf{n}}_{d+2,3} = \left(0 \ 0 \ 1\right)^T\), if \(d=3\).
	For each point \(\hat{\mathbf{r}}_i\), \(i = 1, \dots, \left(d+2\right)\), the \(d\) associated \(RT_1\) basis functions \(\hat{\mathbf{v}}_{i,j}\), \(j = 1, \dots, d\), are defined as 
	\begin{equation*} 
		\hat{\mathbf{v}}_{i,j}(\hat\br_k) \cdot \hat{\mathbf{n}}_{k,l} = \delta_{ik} \delta_{jl} , \qquad i,k = 1,\dots,d+2, \ \ j,l = 1,\dots,d,
	\end{equation*} 
	where \(\delta_{rs} = 1\) if \(r = s\), \(\delta_{rs} = 0\) if \(r \neq s\). See \cite[Fig. 1]{ARICO2025117616} for the \DOFs\ of \(\bV(\hat E)\) and the the basis functions in the case \(d = 2\). For \(d = 3\), the \DOFs\ of \(\bV(\hat E)\) are shown in \cref{basis_functs} and the \(x\), \(y\) and \(z\) components of the basis functions are listed in \cref{RT1_basis_fcts_3D} (see also \cite{Radu}). The \DOFs\ of the pressure space $W(\hat E)$
	are the values at any \(d+1\) points within $\hat E$, see \cref{basis_functs} for the case \(d=3\).
	\begin{figure}
		\centering
		\includegraphics[width=0.3\textwidth]{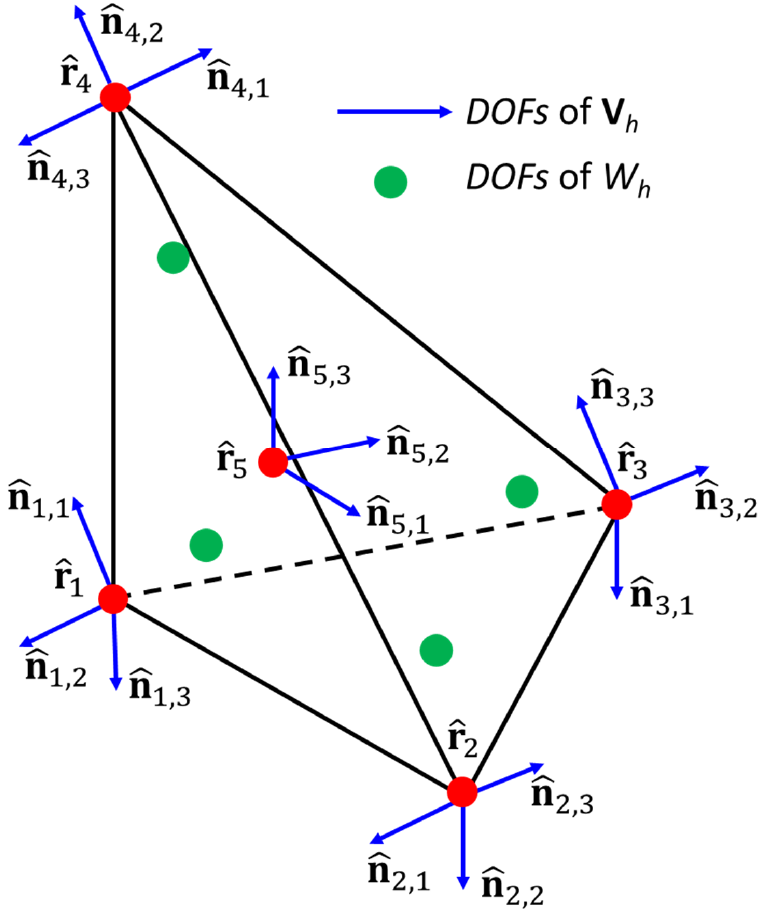}
		\caption{\DOFs\ of \(\bV(\hat{E})\) and \(W(\hat E\) for the case \(d=3\).}
		\label{basis_functs}
	\end{figure}
	We use the Piola transformation \cite{BBF} to define the \(RT_1\) velocity space in each simplex (\eg \cite{W-Y}), as well as the standard change of variables for the pressure space:
	\begin{subequations} \label{eq:transf}
		\begin{gather}
			\mathbf{v} \longleftrightarrow  \mathbf{\hat{v}} : \mathbf{v} = \frac{1}{\left| \mathbf{J}_E\right|} \mathbf{J}_E \mathbf{\hat{v}}, \quad \mathbf{J}_E = \left[\mathbf{r}_{21}, \mathbf{r}_{31} \right] \ \textnormal{if} \ d=2, \quad \mathbf{J}_E = \left[\mathbf{r}_{21}, \mathbf{r}_{31},\mathbf{r}_{41}\right] \ \textnormal{if} \ d=3,  \label{eq:Piola} \\
			w \longleftrightarrow \hat{w} : w = \hat{w},
		\end{gather}
	\end{subequations}
	where \(\mathbf{J}_E\) is the Jacobian matrix associated with the mapping \eqref{mapping}, and \(\mathbf{r}_{ij} = \mathbf{r}_i - \mathbf{r}_j\). The determinant of \(\mathbf{J}_E\) is \(\left| \mathbf{J}_E\right|= 2 \left|E\right|\) with \(\left|E\right|\) the area of simplex \(E\), if \(d=2\), or \(\left| \mathbf{J}_E\right|= 6 \left|E\right|\) with  \(\left|E\right|\) the the volume of \(E\), if \(d=3\). Two important properties of the Piola transformation are \cite{BBF}
	\begin{equation*} 
		\nabla \cdot \mathbf{v} = \frac{1}{\left| \mathbf{J}_E\right|} \nabla_{\hat{\mathbf{x}}} \cdot \mathbf{\hat{v}},
		\quad \bv\cdot\bn_e = \frac{\left| \hat{e}\right|}{\left| e\right|}\hat\bv\cdot\hat\bn_e,
	\end{equation*}
	where $F_E: \hat e \to e$, $\bn_e$ and $\hat\bn_e$ are the unit normal vectors on the faces $e$ and $\hat e$, respectively, and $|e|$ is the length (if \(d=2\)) or area (if \(d=3\)) of $e$. Thanks to the property of the Piola transformation to preserve the vectors normal components, we get the continuity of \(\mathbf{v} \cdot \mathbf{n}\) across any face \(e\) \cite{BBF}.\par 
	
	The \(RT_1\) spaces on $T_h$ are defined as
	\begin{gather*} 
		\mathbf{V}_h= \left\{ \mathbf{v} \in \mathbf{V} : \mathbf{v}|_E \longleftrightarrow \hat{\mathbf{v}}, \ \hat{\mathbf{v}} \in \hat{\mathbf{V}} (\hat{E}) \ \forall E\in T_h \right\},  \\
		W_h= \left\{ w \in W : w|_E \longleftrightarrow \hat{w}, \ \hat{w} \in \hat{W}(\hat{E}) \ \forall E\in T_h  \right\},
	\end{gather*}
	where \eqref{eq:transf} defines the the transformations \(\mathbf{v} \longleftrightarrow \hat{\mathbf{v}}\) and \(w \longleftrightarrow \hat{w}\).

	\begin{table} 
		\caption{The \(RT_1\) basis functions for the reference tetrahedron (\(d=3\).) (see \cite{Radu})}
		{\scriptsize	
			\centering
			\begin{tabular}{c | c c c }
				\hline
				\(\hat{\bv}_{1,1}\) & \(3\hat{x}+\hat{y}+\hat{z} -1 -2\hat{x}^2 -\hat{y}\hat{x}-\hat{z}\hat{x}\) &	\( -2 \hat{x}\hat{y} -\hat{y}^2 +\hat{y} -\hat{z}\hat{y} \) &	\( -2 \hat{x}\hat{z} -\hat{z}\hat{y} -\hat{z}^2 +\hat{z} \)  \\
				\(\hat{\bv}_{1,2}\)& \( -\hat{x}^2 +\hat{x} -2\hat{y}\hat{x}-\hat{z}\hat{x}\) & 	\(\hat{x}+3\hat{y}+\hat{z} -1 - \hat{x}\hat{y} -2\hat{y}^2 -\hat{z}\hat{y} \) &	\( -\hat{x}\hat{z} -2\hat{z}\hat{y} -\hat{z}^2 + \hat{z}  \)  \\
				\(\hat{\bv}_{1,3}\) & \( -\hat{x}^2 +\hat{x} -\hat{y}\hat{x}-2\hat{z}\hat{x}\) & 	\( - \hat{x}\hat{y} -\hat{y}^2 + \hat{y} -2\hat{z}\hat{y} \) &	\( \hat{x}+\hat{y}+3\hat{z} -1-\hat{x}\hat{z} -\hat{z}\hat{y} -2\hat{z}^2 \)  \\
				\(\hat{\bv}_{2,1}\) & \(\hat{x}^2 -\hat{y}\hat{x} \) &\( -\hat{x}  +\hat{y}\hat{x} - \hat{y}^2 +\hat{y}\) &\( \hat{z}\hat{x} - \hat{y}\hat{z}\)  \\
				\( \hat{\bv}_{2,2} \)& \(\hat{x}^2 -\hat{z}\hat{x} \) &\( \hat{y}\hat{x} - \hat{z}\hat{y}\) &\(-\hat{x} + \hat{z}\hat{x} - \hat{z}^2+\hat{z}\)  \\
				\(\hat{\bv}_{2,3}\) & \(\sqrt{3} \left(2\hat{x}^2 -\hat{x} +\hat{y}\hat{x}+\hat{z}\hat{x}\right)\) &	\(\sqrt{3} \left(2 \hat{x}\hat{y} +\hat{y}^2 -\hat{y} +\hat{z}\hat{y}\right) \) &	\(\sqrt{3} \left(2 \hat{x}\hat{z} +\hat{z}\hat{y} +\hat{z}^2 -\hat{z}\right) \)  \\
				\(\hat{\bv}_{3,1}\) & \(\hat{x}\hat{y} - \hat{x}\hat{z}\) &	\(\hat{y}^2 - \hat{z}\hat{y}\) &\(- \hat{y} + \hat{y}\hat{z} -\hat{z}^2 +\hat{z} \) \\
				\(\hat{\bv}_{3,2}\)& \(\sqrt{3}\left(\hat{x}^2 -\hat{x} +2\hat{x}\hat{y} +\hat{x}\hat{z}\right) \) & \( \sqrt{3} \left(\hat{x}\hat{y} + 2\hat{y}^2 -\hat{y} + \hat{y}\hat{z}\right)\) & \( \sqrt{3} \left(\hat{x}\hat{z} + \hat{y}\hat{z} + \hat{z}^2 -\hat{z}\right) \)  \\
				\(\hat{\bv}_{3,3}\) & \(-\hat{y} -\hat{x}^2 +\hat{x} +\hat{y}\hat{x}\) & \( \hat{y} - \hat{y}\hat{x} + \hat{y}^2 - \hat{y} \) &	\(-\hat{x}\hat{z} + \hat{y}\hat{z}\)\\
				\(\hat{\bv}_{4,1}\) & \( \sqrt{3} \left(\hat{x}^2 -\hat{x} + \hat{y}\hat{x} + 2 \hat{z}\hat{x}\right)\) & \( \sqrt{3} \left(\hat{y}\hat{x} + \hat{y}^2 -\hat{y} + 2\hat{z}\hat{y}\right) \) & \( \sqrt{3} \left(\hat{z}\hat{x} + \hat{z}\hat{y} + 2\hat{z}^2 -\hat{z} \right)\) \\
				\(\hat{\bv}_{4,2}\)& \(-\hat{z} -\hat{x}^2 +\hat{x} + \hat{z}\hat{x}\) & \( -\hat{y}\hat{x} + \hat{z}\hat{y}\)	 &	\(-\hat{z}\hat{x} + \hat{z}^2 \)  \\
				\(\hat{\bv}_{4,3}\) & \(-\hat{y}\hat{z} + \hat{z}\hat{x}\) & \(-\hat{z} -\hat{y}^2 +\hat{y} + \hat{y}\hat{z} \) &	\(- \hat{z}\hat{y} +\hat{z}^2 \) \\
				\(\hat{\bv}_{5,1}\) & \(-8 \left(\hat{x}^2 -\hat{x}\right) - 4 \left(\hat{x}\hat{y}\right) - 4\left(\hat{x}\hat{z}\right)\)  &	\(-8\left(\hat{x}\hat{y}\right) - 4\left(\hat{y}^2 -\hat{y}\right) - 4\left(\hat{y}\hat{z}\right)\) & \(-8\left(\hat{x}\hat{z}\right) - 4\left(\hat{z}\hat{y}\right) -4\left(\hat{z}^2 -\hat{z}\right)\)	  \\
				\(\hat{\bv}_{5,2}\)& \(-4 \left(\hat{x}^2 -\hat{x}\right) - 8 \left(\hat{x}\hat{y}\right) - 4\left(\hat{x}\hat{z}\right)\)  &	\(-4\left(\hat{x}\hat{y}\right) - 8\left(\hat{y}^2 -\hat{y}\right) - 4\left(\hat{y}\hat{z}\right)\) & \(-4\left(\hat{x}\hat{z}\right) - 8\left(\hat{z}\hat{y}\right) -4\left(\hat{z}^2 -\hat{z}\right)\)  \\
				\(\hat{\bv}_{5,3}\) & \(-4 \left(\hat{x}^2 -\hat{x}\right) - 4 \left(\hat{x}\hat{y}\right) - 8\left(\hat{x}\hat{z}\right)\)  &	\(-4\left(\hat{x}\hat{y}\right) - 4\left(\hat{y}^2 -\hat{y}\right) - 8\left(\hat{y}\hat{z}\right)\) & \(-4\left(\hat{x}\hat{z}\right) - 4\left(\hat{z}\hat{y}\right) -8\left(\hat{z}^2 -\hat{z}\right)\) \\
				\hline		
			\end{tabular}
		}
		\label{RT1_basis_fcts_3D}
	\end{table}

	\subsection{The quadrature rule}\label{MPMFEM} 
	Let \(\mathbf{s}, \mathbf{v}\) be any pair of continuous vector functions. The following second order accurate quadrature rule is defined in \cite{Radu}:
	\begin{equation} \label{eq:qr_elem}
		\left(\mathbf{s},\mathbf{v}\right)_Q = \sum_{E \in T_h} \left(\mathbf{s},\mathbf{v}\right)_{Q,E}
		= \sum_{E \in T_h} |E| \sum_{i=1}^{d+2} \omega_i \, \mathbf{s}\left(\mathbf{r}_i\right) \cdot \mathbf{v}\left(\mathbf{r}_i\right),
	\end{equation}
	where the weights are $\omega_i = \frac{1}{\left(k+1\right) \times \left(k+2\right)}$, $i = 1,\dots, d+1$, \(\omega_{d+2} = \frac{d+1}{d+2}\). If $\bv \in \bV_h$, the vertex vector \(\mathbf{v}(\mathbf{r}_i)\), $i = 1,\dots,d+1$, can be uniquely obtained from the \DOFs\ $\bv(\br_i)\cdot\bn_{i,1}, \dots, \bv(\br_i)\cdot\bn_{i,d}$, i.e., its normal components to the \(d\) faces sharing vertex \(\mathbf{r}_i\). In a similar way, the vector at the center of mass $\br_{d+2}$ can be constructed from the \DOFs\ $\bv(\br_{d+2})\cdot\bn_{d+2,1}, \dots,\bv(\br_{d+2})\cdot\bn_{d+2,d}$.
	
	The quadrature rule \eqref{eq:qr_elem} is applied to the bilinear forms \((\frac{\phi}{\nu} \bsi_h^{n+1}, \boldsymbol{\tau}_h)\) in \eqref{eq:PP1} in the \PP\ and \(\left(\frac{\mathbf{u}_h^{n+1}-\tilde{\mathbf{u}}_h^{n+1}}{\Delta t}, \mathbf{v}_h\right) + \Bigl(\mathbf{L}\left(\mathbf{u}_h^{n+1}-\tilde{\mathbf{u}}_h^{n+1}\right), \mathbf{v}_h\Bigr) \) in \eqref{eq:Projecproblem_1} in the \PjP. 
	The quadrature rule couples only the \(d\) basis functions associated with the quadrature vertex $\br_i$ \cite{W-Y,Radu}.
	This implies that the viscous stress \(\boldsymbol{\sigma}^{n+1}_{x\left(y,z\right),h}\) and the velocity $\bu_h^{n+1}$ can be locally eliminated, resulting in \spd\ systems for $\tilde\bu_{x\left(y,z\right),h}^{n+1}$ in \eqref{eq:Predproblem} and $\Psi_h^{n+1}$ in \eqref{eq:Projecproblem}, respectively.
	
In the following sections, we describe our numerical methodology for a three dimensional computational domain \(\Omega \subset \mathbf{R}^d \) with \(d = 3\).

	\subsection{Predictor Problem} \label{PP}
	The \PP\ \eqref{eq:Predproblem} is solved separately for the \(x\), \(y\) and \(z\) stress and velocity components. We denote by $\bsi_{x,h}^{n+1}$, $\bsi_{y,h}^{n+1}$ and $\bsi_{z,h}^{n+1}$ the first, second and third rows of tensor $\bsi_{h}^{n+1}$, respectively, and by $\tilde u_{x,h}^{n+1}$, $\tilde u_{y,h}^{n+1}$ and $\tilde u_{z,h}^{n+1}$ the three components of $\tilde\bu_h^{n+1}$. Similarly, $q_{x,h}^n$, $q_{y,h}^n$ and $q_{z,h}^n$ are the components of \(\mathbf{q}_h^n\), while for the BCs we set $\boldsymbol{\Sigma}_b = (\Sigma_{x,b},\Sigma_{y,b}, \Sigma_{z,b})^T$ and $\bu_b = (u_{x,b},u_{y,b},u_{z,b})^T$. In the following, our method is presented only for the \(x\)-component, with similar formulations along the other directions. \par
	
	Applying the quadrature rule \eqref{eq:qr_elem}, the $x$-component of the \PP\  \eqref{eq:Predproblem} becomes: find \(\bsi_{x,h}^{n+1} \in \mathbf{V}_h  : \bsi_{x,h}^{n+1} \cdot \mathbf{n} = Q_h^\Gamma(\Sigma_{x,b}^{n+1} - \Psi_b^n \ n_x)\) on \(\Gamma_n\) and \(\tilde{u}_{x,h}^{n+1} \in W_h\), such that
	\begin{subequations} \label{eq:Predproblem_x} 
		\begin{align}
			& \left( \frac{\phi}{\nu} \bsi_{x,h}^{n+1}, \boldsymbol{\tau}_h \right)_Q - \left( \tilde{u}_{x,h}^{n+1}, \nabla \cdot \boldsymbol{\tau}_h \right) =- \langle u_{x,b}^{n+1}, \boldsymbol{\tau}_h \cdot \mathbf{n} \rangle _{\Gamma_d} \quad \forall \boldsymbol{\tau}_h \in \mathbf{V}_{h,0,\Gamma_n}, \label{eq:Predproblem_x_1} \\
			& \left(\frac{1}{\phi} \frac{\tilde{u}_{x,h}^{n+1} - u_{x,h}^n}{\Delta t}, \xi_h\right) + \left( \nabla \cdot \bsi_{x,h}^{n+1}, \xi_h\right) + \left( q_{x,h}^n, \xi_h\right) \nonumber \\
			& \quad + \left(\left(L_{11} \tilde{u}_{x,h}^{n+1} + L_{12} \tilde u_{y,h}^n + L_{13} \tilde u_{z,h}^n\right), \xi_h\right) = 0 \quad \forall \xi_h \in W_h. \label{eq:Predproblem_x_2}
		\end{align}
	\end{subequations}
	where \(L_{ij}\), \(i,j = 1,2,3\), are the coefficients of tensor \(\mathbf{L}\) defined in \eqref{tensor_L}. We note that the $y$- and $z$-components $\tilde u_{y,h}$ and $\tilde u_{z,h}$ in the last term are evaluated at the previous time $t^n$, which allows for complete decoupling of the problems for each component.
Next, the local elimination of the viscous stress $\bsi_{x,h}^{n+1}$ is described both for the case of any internal or boundary vertex and the center of mass midpoint of any simplex.\par 
	
	\medskip
	\noindent
	\textbf{The case of any vertex.} Let \(\mathbf{r}\) be any internal vertex of \(T_h\) shared by \(\mathcal{I}\) interfaces and \(\mathcal{I}\) simplices, denoted by \(e_i\) and \(E_i\), respectively. Let \(\boldsymbol{\tau}_{i} \in \mathbf{V}_h\) be the basis functions on the faces $e_i$, and $\sigma_{x,i}^{n+1}$ be the associated \DOFs\ of $\bsi^{n+1}_{x,h}$. Thanks to the property of the quadrature rule to localize the interaction of the face \DOFs, setting 
	\(\boldsymbol{\tau}_h = \boldsymbol{\tau}_1, \ldots, \boldsymbol{\tau}_\mathcal{I}\) in \eqref{eq:Predproblem_x_1} results in 
	a \( \left(\mathcal{I} \times \mathcal{I}\right) \) local linear system for the \(\sigma^{n+1}_{{x,i}}\) unknowns, \(i=1, \dots, \mathcal{I}\):
	\begin{equation} \label{eq:discr_x_0}
		\sum_j \sigma^{n+1}_{x,j} \left(\left(\frac{\phi}{\nu}\boldsymbol{\tau}_j,\boldsymbol{\tau}_i\right)_{Q,E_a}
		+ \left(\frac{\phi}{\nu}\boldsymbol{\tau}_j,\boldsymbol{\tau}_i\right)_{Q,E_b} \right)  
		= (\tilde{u}^{n+1}_{x,h},\nabla \cdot \boldsymbol{\tau}_i)_{E_a}
		+ (\tilde{u}^{n+1}_{x,h},\nabla \cdot \boldsymbol{\tau}_i)_{E_b},
		\ i = 1, \dots, \mathcal{I},
	\end{equation}
	where \(E_a\) and \(E_b\) are the two simplices sharing face \(e_i\) and the summation on $j$ on the l.h.s. is over the faces \(e_j \left(\in E_a \cup E_b\right)\) sharing with $e_i$ the vertex associated with $\boldsymbol{\tau}_i$. The integrals
	\((\tilde{u}^{n+1}_{x,h},\nabla \cdot \boldsymbol{\tau}_h)_{E_{a(b)}}\) on the
	r.h.s. by computed by a \(\left(d+1\right)\)-point Gaussian integration rule. This quadrature rule is exact for quadratic functions, and since both \(\tilde{u}^{n+1}_{x,h} \in P_1\left(E\right)\) and \(\nabla \cdot \boldsymbol{\tau}_h \in P_1\left(E\right)\), it is exact for \((\tilde{u}^{n+1}_{x,h},\nabla \cdot \boldsymbol{\tau}_h)_{E_{a(b)}}\).
	
	The matrix-vector form of the local linear system \eqref{eq:discr_x_0} is 
	\begin{equation} \label{eq:vect_matr}
		\mathbf{A} \mathbf{\Sigma}_x = \mathbf{B}^T \mathbf{U}_x,
	\end{equation} 
	where %
	\begin{itemize}
		\item \(\mathbf{A}\) is a \(\left(\mathcal{I} \times \mathcal{I}\right)\) matrix with coefficients 
		\(A_{ij} = \left(\frac{\phi}{\nu}\boldsymbol{\tau}_j,\boldsymbol{\tau}_i\right)_Q\),
		$i,j = 1,\ldots,\mathcal{I}$, 
		
		\item \(\mathbf{\Sigma}_x\) is a \( \left(\mathcal{I} \times 1 \right) \) vector whose coefficients are the \DOFs\ \(\sigma^{n+1}_{x,i}\), \(i=1, \dots , \mathcal{I}\),
		
		\item \(\mathbf{B}^T\) is a \(\left(\mathcal{I} \times \left(\left(d+1\right) \times \mathcal{I}\right) \right)\) matrix with coefficients coming from \((\tilde{u}^{n+1}_{x,h},\nabla \cdot \boldsymbol{\tau}_i)_{E_a}\) + \((\tilde{u}^{n+1}_{x,h},\nabla \cdot \boldsymbol{\tau}_i)_{E_b}\), $i = 1,\ldots,\mathcal{I}$,
		
		\item \(\mathbf{U}_x\) is a \(\left(\left(\left(d+1\right) \times \mathcal{I}\right) \times 1\right)\) vector, whose coefficients are the \(\left(d+1\right)\) \DOFs\ of $\tilde{u}^{n+1}_{x,h}$ within each simplex $E_i$, $i = 1,\ldots,\mathcal{I}$.
	\end{itemize}
	
	The matrix \(\mathbf{A}\) is \spd\, as proved in \cite{W-Y}, which implies that the local system \eqref{eq:vect_matr} is solvable and one can  
	express the $\mathcal{I}$ \DOFs\ \(\sigma^{n+1}_{x,i}\),  \(i=1, \dots, \mathcal{I}\), sharing vertex $\br$ in terms of the $\left(\left(d+1\right) \times \mathcal{I}\right)$ \DOFs\ of \(\tilde{u}^{n+1}_{x,h}\) on the $\mathcal{I}$ simplices sharing $\br$.
	
	Let \(\mathbf{r}\) be a boundary vertex and denote by \(\mathcal{I}\) and \(\mathcal{H}\) the number of interfaces and simplices sharing \(\mathbf{r}\), respectively, where generally \(\mathcal{I} \ne \mathcal{H}\). In this case the size of matrix \(\mathbf{B}^T\) is \(\left(\mathcal{I} \times \left(\left(d+1\right) \times (\mathcal{H}\right)) \right)\) and the size of vector \(\breve{\mathbf{U}}_x\) is \(\left(\left(d+1\right) \times (\mathcal{H}) \times 1\right)\). If \(\br \in \overline\Gamma_d\), thanks to \eqref{eq:Predproblem_x_1}, we also account for the  contribution \(- \langle u^{n+1}_{x,b}, \boldsymbol{\tau}_h \cdot \mathbf{n} \rangle _{\Gamma_d}\) in the local system \eqref{eq:vect_matr}, computed by numerical integration over the face(s) \(e_i \in \Gamma_d\) sharing $\br$. The system \eqref{eq:Predproblem_x} changes as 
	\begin{equation} \label{eq:vect_matr_gamma_d}
		\mathbf{A} \mathbf{\Sigma}_x = \mathbf{B}^T \mathbf{U}_x + \mathbf{G}_{d,x} ,
	\end{equation} 
	where the size of vector \(\mathbf{G}_{d,x}\) is \(\left(\mathcal{I} \times 1 \right)\) and its nonzero coefficients appear only in the rows associated with the boundary faces \(e_i \in \Gamma_d\) sharing \(\mathbf{r}\). If \(\mathbf{r} \in \overline\Gamma_n\), the local system \eqref{eq:Predproblem_x} incorporates the essential stress boundary condition: 
	\begin{equation} \label{eq:vect_matr_gamma_n}
		\tilde{\mathbf{A}} \mathbf{\Sigma}_x = \tilde{\mathbf{B}}^T \mathbf{U}_x + \mathbf{G}_{n,x},
	\end{equation} 
	where \(\mathbf{G}_{n,x}\) is a \(\left(\mathcal{I} \times 1 \right)\) vector. \par 
	
	\medskip
	
	\noindent
	\textbf{The case of any center of mass.} Let \(\mathbf{r}\) be the center of mass of any simplex \(E\) and let \(\boldsymbol{\tau}_{i} \in \mathbf{V}_h\), \(i=1,\dots,d\), be the associated stress basis functions. Let \(\sigma^{n+1}_{x,i}\) be the \(d\) \DOFs\ of $\bsi^{n+1}_{x,h}$ associated with $\br$. Since the quadrature rule localizes the interaction of \DOFs, 
	\(\sigma^{n+1}_{x,i}\), \(i=1,\dots, d\) are coupled only with each other, resulting in a \(d \times d\) linear system with \spd\ matrix written in the same form as in \eqref{eq:vect_matr} for \(\sigma^{n+1}_{x,i}\), \(i=1,\dots, d\).
	
	\medskip
	\noindent
	\textbf{The \PP\ solved for \(\tilde{\bu}^{n+1}\) velocity component}. With the local elimination of the viscous stress, as previously described, the MFMFE scheme of the \PP\ becomes a system for the \(\tilde u_{x,h}^{n+1}\) with \(\left(d+1\right)\) \DOFs\ per simplex, and the associated algebraic system arising from \eqref{eq:Predproblem_x} becomes 
	\begin{equation} \label{system_global}
		\begin{pmatrix}
			\mathrm{A} \quad -\mathrm{B}^T\\
			\mathrm{B} \ \qquad \mathrm{D}
		\end{pmatrix} \begin{pmatrix}
			\mathrm{\Sigma}_x \\ \mathrm{U}_x
		\end{pmatrix} = \begin{pmatrix}
			\mathrm{G}_x \\ \mathrm{Q}_x + \mathrm{F}_x + \mathrm{L}_y + \mathrm{L}_z
		\end{pmatrix},
	\end{equation}
	where 
	\begin{itemize}
		
		\item \(\mathrm{A}\) is a \(\left(\mathfrak{I} \times \mathfrak{I}\right)\) block diagonal matrix, with \(\mathfrak{I} = d \times \mathfrak{S}_T + d \times N_T\), (with \(\mathfrak{S}_T\) and $N_T$  the total number of the faces and simplices, respectively, as specified in \cref{discr}). Matrix \(\mathrm{A}\) is given by assembling the (local) block matrices \(\mathbf{A}\), cf. \eqref{eq:vect_matr},
		
		\item \(\mathrm{\Sigma}_x\) is a \(\left(\mathfrak{I} \times 1\right)\) vector whose coefficients are the \DOFs\ of $\bsi_{x,h}^{n+1}$, obtained by assembling the (local) vectors \(\mathbf{\Sigma}_x\), cf. \eqref{eq:vect_matr}, 
		
		\item \(\mathrm{B}^T\) is a \(\left(\mathfrak{I} \times \left(d \times N_T\right) \right)\) matrix, obtained by assembling the block (local) matrices \(\mathbf{B}^T \), cf. \cref{eq:vect_matr},
		
		\item \(\mathrm{D}\) is a \(\left(\left(d \times N_T\right) \times\left( d \times N_T\right)\right)\) block diagonal matrix, whose blocks are associated with \linebreak $\left(\left(\frac{1}{\phi} \frac{1}{\Delta t} + L_{{11}}\right) \tilde{u}_{x,h}^{n+1},\xi_h\right)$ in \eqref{eq:Predproblem_x},
		
		\item \(\mathrm{U}_x\) is a \(\left(\left(d \times N_T\right) \times 1\right)\) vector, with coefficients the \DOFs\ of $\tilde{u}_{x,h}^{n+1}$,
		
		\item \(\mathrm{G}_x\) is a \(\left(\mathfrak{I} \times 1\right)\) vector obtained by the assembly of the local vectors \(\mathbf{G}_{d,x}\) and \(\mathbf{G}_{n,x}\), 
		
		\item \(\mathrm{Q}_x\), \(\mathrm{F}_x\), \(\mathrm{L}_y\) and \(\mathrm{L}_z\) are \(\left(\left(d \times N_T\right) \times 1\right)\) vectors corresponding to $-\left( q_{x,h}^n, \xi_h\right)$, $\left(\frac{1}{\phi}\frac{1}{\Delta t} u_{x,h}^{n},\xi_h\right)$, \( \left( L_{{12}} u_{y,h}^n, \xi_h \right) \) and \( \left( L_{{13}} u_{z,h}^n, \xi_h \right) \) in \eqref{eq:Predproblem_x}, respectively. 
	\end{itemize}
	Similar to $\mathrm{B}$, the integrals in $\mathrm{D}$, \(\mathrm{Q}_x\), \(\mathrm{F}_x\), \(\mathrm{L}_y\) and \(\mathrm{L}_z\) are computed by a \(\left(d+1\right)\)-point Gaussian quadrature rule
	
	By inverting the matrix $\mathrm{A}$, the  unknown stress vector \(\mathrm{\Sigma}_x\) can be eliminated in \cref{system_global}, resulting in a \(\left(\left(\left(d+1\right) \times N_T\right) \times \left(\left(d+1\right) \times N_T\right) \right)\) system for \(\mathrm{U}_x\):
	\begin{equation} \label{eq:sys_velox_PP}
		\left(\mathrm{D} + \mathrm{B} \ \mathrm{A}^{-1} \ \mathrm{B}^T \right)\mathrm{U}_x = \mathrm{Q}_x + \mathrm{F}_x + \mathrm{L}_y + \mathrm{L}_z + \mathrm{B} \mathrm{A}^{-1} \mathrm{G}_x,
	\end{equation}
	where the \(\left(d+1\right)\) \DOFs\ of \(\tilde{u}^{n+1}_{x,h}\) within each simplex \(E\) are coupled with the \(\left(d+1\right)\) \DOFs\ of all simplices sharing a vertex with $E$, see \cite[Figure~3]{ARICO2025117616} for a 2D view.
	
	The matrix in \eqref{eq:sys_velox_PP} is \spd\ \cite{W-Y} and the system is solved by a preconditioned conjugate gradient method with incomplete Cholesky factorization \cite{Dongarra}, which results in a very fast and efficient procedure. A lot of computational effort is saved because the factorization of the matrix of system in \eqref{eq:sys_velox_PP} occurs only once, before the the time loop, since the matrix coefficients depend only on geometric quantities, the kinematic viscosity \(\nu\), the time step size \(\Delta t\), as well as the physical properties of the porous medium, like the porosity \(\phi\) and the coefficients of the permeability tensor. \par
	
	\begin{remark}\label{sym-system}
		With the ratio \(\frac{\nu}{\phi}\) inside the divergence operator in \eqref{eq:momentum}, we obtain a symmetric matrix in \eqref{eq:sys_velox_PP}. It can be easily seen that this matrix would be non-symmetric if \(\frac{\nu}{\phi}\) was outside the divergence operator.
	\end{remark}

We apply the same procedure along the \(y\) and \(z\) directions, with the system having the same matrix as in \eqref{eq:sys_velox_PP}. At the end of the \PP \ we obtain an intermediate velocity $\tilde{\bu}^{n+1}_{h}|_E \in (P_1(E))^2$ and discontinuous, i.e., the continuity of normal flux at the simplex interfaces is not imposed in this step.

\subsection{Projection problem and pressure gradient update} \label{PjP}
Applying the quadrature rule in \eqref{eq:qr_elem}, the \PjP\ in \eqref{eq:Projecproblem} is: find \(\mathbf{u}_h^{n+1} \in \mathbf{V}_h : \mathbf{u}_h^{n+1} \cdot \mathbf{n} = Q_h^\Gamma\left(\mathbf{u}_b \cdot \mathbf{n}\right)\) on \(\Gamma_d\) and \(\Psi_h^{n+1} \in W_h\) such that
	\begin{subequations} \label{eq:Projecblem}
		\begin{align}
			&  \hskip - .1cm \bigg( \frac{1}{\phi} \frac{\mathbf{u}_h^{n+1}-\tilde{\mathbf{u}}_h^{n+1}}{\Delta t} + \mathbf{L} \left(\mathbf{u}_h^{n+1}-\tilde{\mathbf{u}}_h^{n+1}\right), \mathbf{v}_h\bigg)_Q 
			- ( \Psi_h^{n+1} - \Psi_h^n,\nabla \cdot \mathbf{v}_h ) = \notag \\
			& \quad - \langle \Psi_b^{n+1} - \Psi_b^n, \mathbf{v}_h \cdot \mathbf{n} \rangle_{\Gamma_n}
			\ \forall \mathbf{v}_h \in \mathbf{V}_{h,0,\Gamma_d}, \label{eq:Projecblem_1} \\
			& \hskip -.1cm \left( \nabla \cdot \mathbf{u}_h^{n+1},w_h\right) = 0 \quad \forall w_h \in W_h.
			\label{eq:Projecblem_2} 
		\end{align}
	\end{subequations}
	The method is the same as the one applied for the solution of the \PP\ for $\bsi_{x,h}^{n+1} \in \mathbf{V}_h$ and $\tilde{u}_{x,h}^{n+1} \in W_h$ in \eqref{eq:Predproblem_x}. The corrected velocity $\mathbf{u}_h^{n+1}$ is locally eliminated according to the same procedure as described in \cref{PP}, and we solve a \spd\ system for $\Psi_h^{n+1}$ of type \eqref{eq:sys_velox_PP}, where the \(\left(d+1\right)\) \DOFs\ of $\Psi_h^{n+1}$ within each simplex \(E\) are coupled with the \(\left(d+1\right)\) \DOFs\  of those simplices sharing a vertex with \(E\). The corrected velocity $\mathbf{u}_h^{n+1}$ can be easily recovered after the solution of the system by a local post-processing.
	
	The term \(\langle \Psi_b^{n+1} - \Psi_b^n, \mathbf{v}_h \cdot \mathbf{n} \rangle_{\Gamma_n}\) in \eqref{eq:Projecblem_1} is computed via numerical integration using the boundary values of \(\Psi_b^{n+1} \) on \(\Gamma_n\) given by 
	\begin{equation*}
		\Psi_b^{n+1}|_e = \left(\boldsymbol{\Sigma}_b^{n+1}|_e + \left.\left(\frac{\nu}{\phi} \nabla \tilde{\mathbf{u}}_h^{n+1}\right)\right|_e \, \mathbf{n}\right) \cdot \mathbf{n},
	\end{equation*}
	where the tensor $\nabla \tilde{\mathbf{u}}_h^{n+1}$ is computed on the simplex $E$ with face $e$ and it has a constant value, since $\tilde{\mathbf{u}}_h^{n+1} \in (P_1(E))^2$.\par 
	Before the next time iteration, we update the pressure gradient according to \eqref{eq:up_q}. This is performed by updating the \DOFs\ of $\mathbf{q}_h^{n+1}$, since both $\mathbf{q}_h^{n+1}$ and the test function $\bxi_h$ live in the same discretized space $(W_h)^d$. For any degree of freedom point $\br$, we have
	$$
	\mathbf{q}_h^{n+1}(\br) = \mathbf{q}_h^n(\br) - \left( \frac{1}{\Delta t} \frac{1}{\phi(\br)} \mathbf{I} + \mathbf{L}(\br) \right) \left(\mathbf{u}_h^{n+1}(\br) - \tilde{\mathbf{u}}_h^{n+1}(\br)\right).
	$$

	\section{Numerical Tests} \label{tests}
	In this section we present several numerical tests to illustrate the performance of the numerical method. We call ``method 1'' and ``method 2'' the standard method in \eqref{eq:Predproblem}--\eqref{eq:up_q} and the simplified method without the permeability correction term described in \eqref{eq:Projecproblem_1-simpl}--\eqref{eq:up_q-simpl}, respectively.
	
	We present four numerical tests. Test 1 is devoted to the numerical investigation of the convergence order in space and time, both for \(d=2\) and \(d=3\). In Test 2 we study the flow of a free fluid around a porous obstacle under different working conditions and we also compare the results from methods 1 and 2. In Test 3 we consider the interaction of a free fluid with a porous region with a steps-like interface. For Tests 2 and 3, we compare the results of our method to numerical solutions available in the literature. Finally, Test 4 is a show-case application where we investigate the effect of a porous medium filling the aneurysmatic sac of a real-case intracranial aneurysm (ICA). We emphasize that tests 2-4 involve two regions with parameters in either the Stokes or Darcy regimes. The results illustrate the robustness of the proposed method in both regimes.
	
	The computational grids of the applications proposed in this Section have been generated by the open-source software Netgen \cite{Schberl1997NETGENAA}, the numerical scheme has been implemented in an in-house Fortran 90-95 code, and the the open-source software Paraview \cite{Paraview} has been used for the post-processing and visualization of the results.\par
	
	\subsection{Test 1: study of the convergence order in space and time} \label{test1}
	The strategy we adopt to investigate the convergence order in space and time is 
	\begin{itemize}
		\item consider a computational heterogeneous porous domain \(\Omega \subset \mathbf{R}^d\), \(d = 2, 3\), with assigned spatial distribution of porosity and permeability tensor coefficients,
		\item assign an analytical time-dependent solution for the velocity components \(u_x, u_y, u_z\) and kinematic pressure \(\Psi\), and the corresponding ICs and BCs,
		\item discretize \(\Omega\) by a coarse simplicial grid \(T_h\) and progressively operate some refinements of \(T_h\),
		
	\end{itemize}
	
	To study the convergence order in space, we adopt a small enough time step size \(\Delta t\) and compute the \(L_2\)-norms of the errors of \(u_x, u_y, u_z\) and \(\Psi\) as
	\begin{equation} \label{eq:L2_error}
		L_2^q = \max_{1\le n\le N}\|q(t_n) - q_h^n\|, \textnormal{ with } q=u_x, \, u_y, \, u_z, \, \Psi,
	\end{equation}
	where \(N\) is the number of intervals in the time discretization, as specified in \cref{discr}. \par
	
	If \(h_l\) marks the grid size associated with the \(l\)-th refinement level and \(e_l^q \sim h_l^{r_h^q}\) is the error of the variable \(q\) in this level, we compute the the associated convergence rate \(r_h^q\) by comparing the errors associated with two consecutive grid refinement levels with size \(h_{l-1}\) and \(h_l\):
	\begin{equation} \label{eq:rc}
		r_h^q = \frac{\log \left(\frac{e_{l-1}^q}{e_l^q}\right)}{\log\left(\frac{h_{l-1}}{h_l}\right)}.    
	\end{equation}
	To investigate the convergence order in time, we run simulations over a refined enough grid, changing the time step size in a given range, and compute the \(L_2\)-norms of the errors as in \eqref{eq:L2_error}. The rate of convergence in time is obtained as 
	\begin{equation} \label{eq:rc_time}
		r_{\Delta t}^q = \frac{\log \left(\frac{e_{l-1}^q}{e_l^q}\right)}{\log\left(\frac{\Delta t_{l-1}}{\Delta t_l}\right)},
	\end{equation}
	where \(\Delta t_{l-1}\) and \(\Delta t_l\) are two consecutive time step sizes in the given range. \par
	In the following 2D and 3D applications, we assume that the kinematic fluid viscosity is \(\nu = 1\).

	\subsubsection{2D study} \label{2D_analyt}
	We consider the circular domain \(\Omega \subset \mathbf{R}^2 \) shown in \cref{test_1_sett_a} (top), where we have three layers (zones) with outer (dimensionless) radii 0.3, 0.5 and 1, respectively. Zones 1 and 2 are the inner and outer bulk porous regions, with constant in space values of the porosity \(\phi\) and the permeability tensor $\mathbf{K}$ given as 
	\begin{equation} \label{perm_bulk}
		\mathbf{K} = \mathbf{M} \mathbf{C} \mathbf{M}^{-1}, \qquad \mathbf{M} = \begin{pmatrix} \cos \alpha & -\sin \alpha \\ \sin \alpha & \cos \alpha  \end{pmatrix}, \qquad \mathbf{C} = \begin{pmatrix} \frac{k}{\beta} & 0 \\ 0 & k \end{pmatrix},
	\end{equation}
	where \(k\) and \(\beta\) are positive scalar parameters and \(\alpha\) is the anisotropy angle.  The porosity and permeability parameters are listed in \cref{test1_bulk_par}. A transition layer (or zone), is located in between the two bulk regions (zone 3 in \cref{test_1_sett_a} (top)). Here the values of \(\phi\) and \(\mathfrak{K}_{i,j}\) undergo continuous changes, from the values in zone 1 to the values in zone 2, according to 
	\begin{subequations} \label{profiles}
		\begin{gather} 
			\phi\left(d\right)=\frac{1}{2} \left(\phi_{max} - \phi_{min}\right) \mbox{tanh}\left(d\, \theta_{\psi}\right) + \frac{1}{2} \left(\phi_{max} + \phi_{min}\right), \label{prof_poro}\\
			\mathfrak{K}_{i,j}\left(d\right)=\mathfrak{K}^0_{i,j}\frac{1}{2} \left(1-\mbox{tanh}\left(d \, \theta_{\mathfrak{K}}\right)\right), \label{prof_invpermb}		
		\end{gather}
	\end{subequations}
	which are similar to \cite[Eq. (39)]{Arico-ODA}. In \eqref{profiles}, \(d\) is the distance from any point to the mid-line of the transition zone, assumed to be positive or negative depending on whether the point is in the side of zone 1 or 2, respectively. In addition,
	\(\phi_{max}\) and \(\phi_{min}\), as well as \(\mathfrak{K}^0_{i,j}\), \(i,j =1,\dots, d\), are positive real scalar parameters such that the porosity and permeability values in the transition layer and in the bulk porous regions match at the interfaces. The parameters \(\theta_{\psi}\) and \(\theta_{\mathfrak{K}}\) adjust the slope of the profiles. The values of all parameters are listed in \cref{test1_transzone_par}.
	
	We assign the time-dependent (dimensionless) analytical solution
	\begin{equation} \label{an_sol_2D}
		\Psi= \left(a_1 + b_1 x + c_1 x^2 + a_2 + b_2 y + c_2 y^2\right) \cos \omega t,
		\quad \bu = - \mathbf{K} \nabla \Psi,
	\end{equation}
	where \(a_i, \ b_i,\ c_i\) and \(\omega\) are scalar real variables listed in \cref{test1_par_ansol}. We add source terms in the model equations \eqref{eq:governing_Eqq} according to the assigned analytical solution. The maximum value of the Reynolds number is  \(Re = \frac{\| \bu\|_{max}\ 2\ R_{out}}{\nu} \simeq 1.4\e{-04}\), where \(\|\bu \|_{max}\) is the maximum value of the velocity magnitude and \(R_{out}\) is the outer radius of the domain.\par 
	We discretize the domain by either a coarse regular grid or a distorted grid shown in  \cref{test_1_sett_a} (top) and denoted as \(T_{h,R}\) and \(T_{h,D}\), respectively. The number of triangles and vertices of \(T_{h,R}\) and \(T_{h,D}\) is 638 and 292, and 340 and 163, respectively. For any triangle \(E \in T_{h,R} \textnormal{ or } E \in T_{h,D}\), we compute the dimensionless aspect ratio \(A_r\), 
	\begin{equation}
		A_r=\frac{2 \ h_{n}}{\sqrt{3} \ L_{x} },
	\end{equation}
	which represents the ratio between the aspect ratio of \(E\), where \(h_n\) and \(L_x\) are the minimum and the maximum values of the  heights and length sides of \(E\), respectively, and the ideal aspect ratio of the equilateral triangle, \(\frac{\sqrt{3}}{2}\). The minimum values of \(A_r\) are approximately 0.53 and 0.33 for \(T_{h,R}\) and \(T_{h,D}\), respectively. \par
	
	We perform five refinements of both \(T_{h,R}\) and \(T_{h,D}\), by halving each side. We consider two scenarios of the assigned BCs, depending on the distribution of the boundary portions \(\Gamma_d\) and \(\Gamma_n\), as in \cref{test_1_sett_a} (bottom), denoted as ``case 1'' and ``case 2'', respectively. The BCs in case 2 is a way we implement full Dirichlet boundary condition for the velocity, assigning \(\Psi\) only at one edge in \(\Gamma_n\) for the solution to be unique. \par 
	The ICs are obtained by \cref{an_sol_2D} by setting \(t = 0\). In \cref{fig_test1_sol} we plot the computed solution over the 2nd refinement grid level and case 1 at the (dimensionless) simulation time \(\mathcal{T} =2\). \par 
	To investigate the convergence order in space, we use a (dimensionless) time step size \(\Delta t = 1\e{-04}\). The final (dimensionless) simulation time is $\mathcal{T} = 6$. Tables \ref{error_RG_1}--\ref{error_DG_2}
	list the \(L_2\)-norms of the errors for the velocity components and pressure, as specified in \eqref{eq:L2_error}, as well as the associated convergence order \(r_h^q\), given by \eqref{eq:rc}, for the two BCs cases and two grid choices. The convergence rate \(r_h\) is approximately 2 in all tables. The values of \(L_2^{u_{x(y)}}\) computed for the second BCs case are generally smaller than in the first BCs scenario, while the opposite occurs for \(L_2^{\Psi}\). The errors computed over the sets of the regular grids are approximately a magnitude order smaller than the errors obtained over the sets of the distorted grids.   \par
	
	We investigate the convergence order in time running our simulations over the \(3^{rd}\) refined grid level and using a (dimensionless) time step size \(\Delta t\) in the range \(\left[1\e{-04}, 1\e{-01}\right]\).
	Tables \ref{error_time_RG_1}--\ref{error_time_DG_2}
	list the computed \(L_2\)-norms of the errors and the associated convergence order \(r_{\Delta t}^q\). The rate \(r_{\Delta t}^q\) is approximately 1, in line with the theoretically proved first order in \cref{thm:time-err}.\par
	
	We have also studied the convergence order in space and time for method 2. For sake of space, we present only the errors and rates over the set of distorted grids for the second BCs case in \cref{error_DG_2_model2,error_time_DG_2_model2}, observing again second order in space and first order in time. We note that the errors for the two methods are similar in this test with smooth analytical solution. However, in Test 2 we illustrate that method 2 may produce less accurate results that method 1 for more challenging problems.
	
	\begin{figure}
		\centering
		\begin{subfigure} {0.55\textwidth}
			\includegraphics[width=\linewidth]{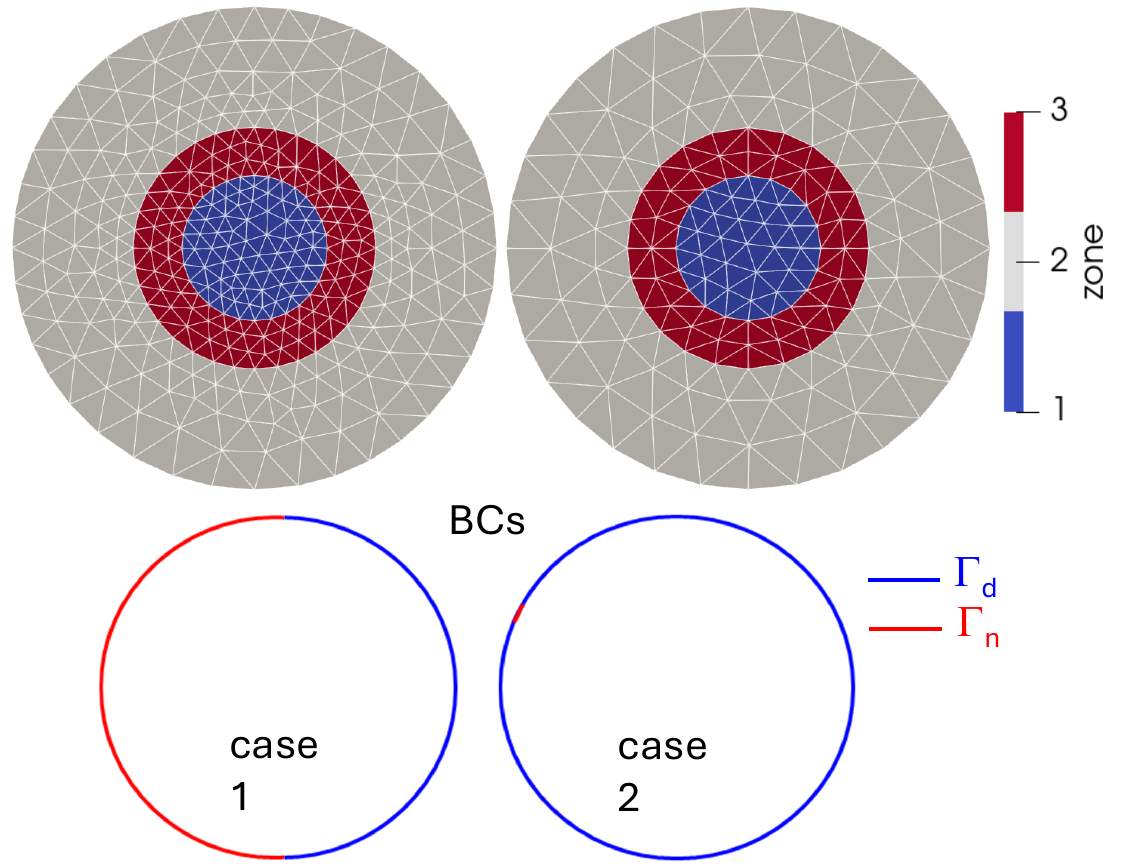}
			\caption{Test 1, 2D case. Top: coarse grids \(T_{h,R}\) (left) and \(T_{h,D}\) (right) and porous medium zones (the color bar indicates the zone number). Bottom: assigned BCs.}
			\label{test_1_sett_a}
		\end{subfigure}
		\begin{subfigure}  {0.3\textwidth}
			\includegraphics[width=\linewidth]{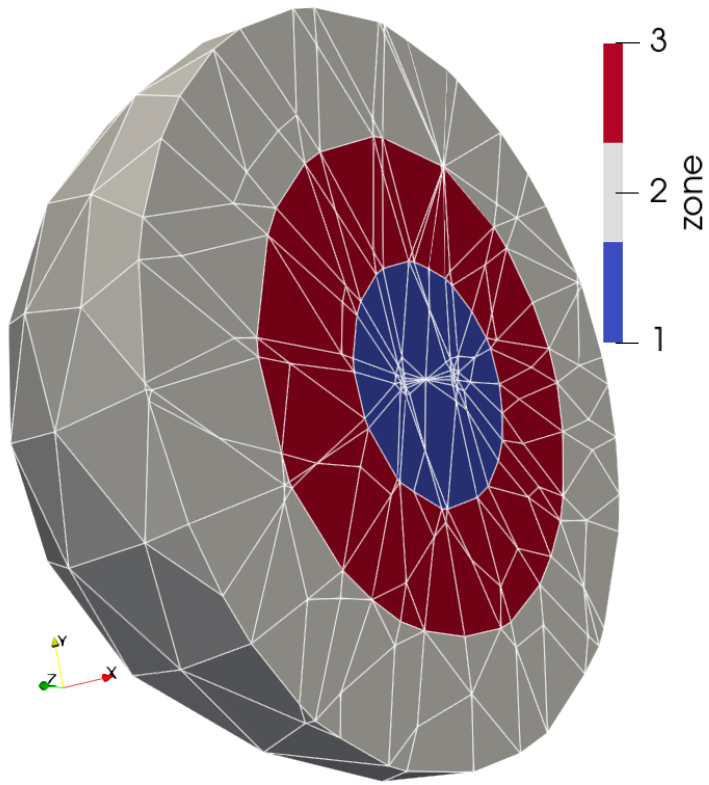}
			\caption{Test 1, 3D case. Porous medium zones (the color bar indicates the zone number).}
			\label{test_1_sett_b}
		\end{subfigure}
		\label{test_1_sett}
		\caption{Test 1, 2D and 3D cases.}
	\end{figure}
	
	\begin{figure}
		\centering
		\includegraphics[width=0.65\textwidth]{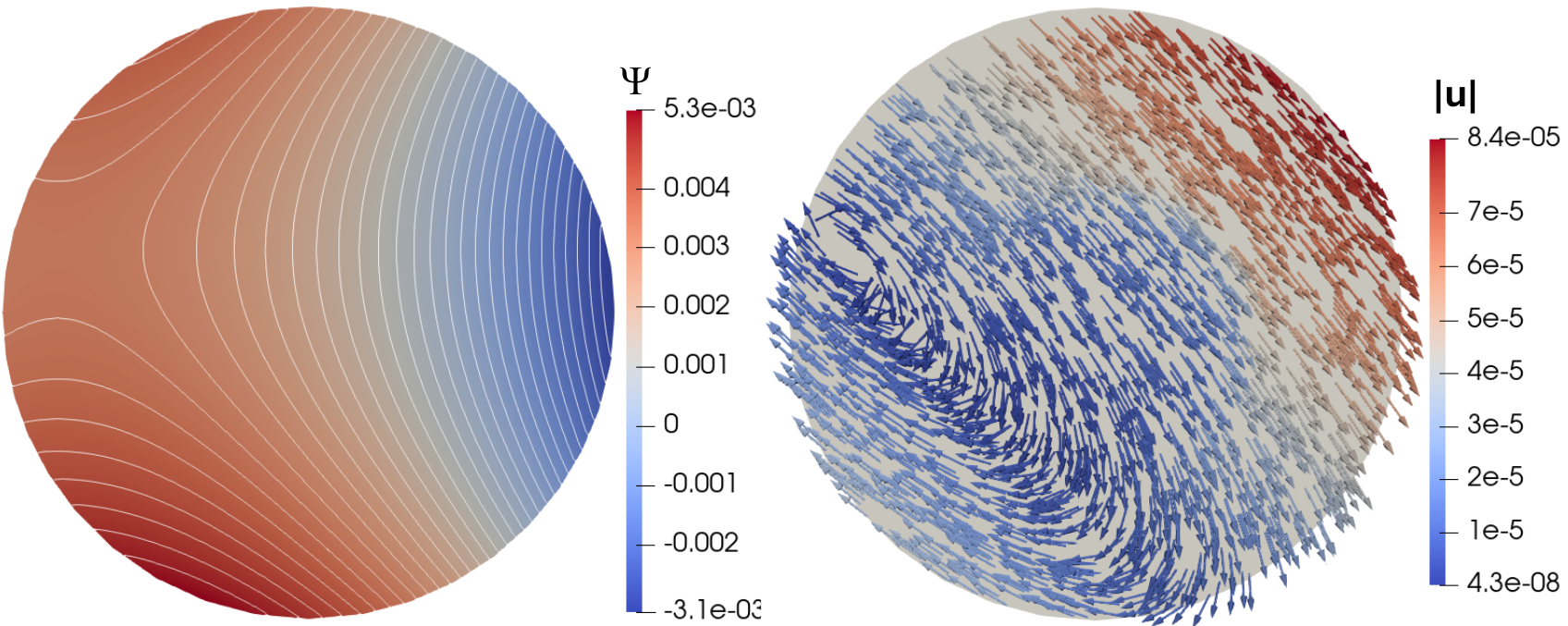}
		\caption{Test 1. Computed solution 1st refinement grid level, case 1. \(\Psi\) (left), \(\mathbf{u}\) (right)}
		\label{fig_test1_sol}
	\end{figure}
	
	\begin{table} 
		\caption{Test 1, 2D case. (Dimensionless) parameters of the bulk porous regions.}
		\centering
		\begin{tabular}{c c c c c c c c}
			zone & \(\phi\) & \(\alpha\) & $\beta$ & \(k\) & \(\mathfrak{K}_{1,1}\) & \(\mathfrak{K}_{2,2}\) & \(\mathfrak{K}_{1,2} = \mathfrak{K}_{2,1}\) \\
			\hline
			1 & 0.597	&\(-\pi / 4\) &	10 & \(1.37\e{02}\)& 403 & 403 & -329 \\	
			2 & 0.7037	&\(-\pi / 4\) &	10 & \(3.72\e{02}\)& 147.9 & 147.9 & -121  \\
			\hline		
		\end{tabular}
		\label{test1_bulk_par}
	\end{table}
	
	\begin{table} 
		\caption{Test 1, 2D case. (Dimensionless) parameters of the transition porous zone.}
		\centering
		\begin{tabular}{c c c c c c c c}
			zone & \(\phi_{max}\) & \(\phi_{min}\) & \(\theta_{\psi}\) & \(\theta_{\mathfrak{K}_{i,j}}\) & \(\mathfrak{K}^0_{1,1}\) & \(\mathfrak{K}^0_{2,2}\) & \(\mathfrak{K}^0_{1,2} = \mathfrak{K}^0_{2,1}\) \\
			\hline
			3 & 0.766	& 0.5344 &	5 & 5 & 550 & 550 & -450 \\	
			\hline		
		\end{tabular}
		\label{test1_transzone_par}
	\end{table}
	
	\begin{table} 
		\caption{Test 1, 2D case. (Dimensionless) parameters of the assigned analytical solution in \eqref{profiles}.}
		\centering
		\begin{tabular}{c c c c c c c}
			\(a_1\) & \(b_1\) & \(c_1\) & \(a_2\) & \(b_2\) & \(c_2\) & \(\omega\)\\
			\hline
			\(5\e{-04}\) & \(1.5\e{-03}\)	& \(-3.2\e{-03}\) &	\(-8\e{-04}\) & \(-1.95\e{-03}\) & \(1.95\e{-03}\) & \(\pi\) \\
			\hline		
		\end{tabular}
		\label{test1_par_ansol}
	\end{table}
	
	\begin{table} 
		\small
		\caption{Test 1, 2D case, method 1. \(L_2\)-norms of errors and rates in space: \(T_{h,R}\), BCs case 1.}
		\centering
		\begin{tabular}{c c c c c c c}
			ref. level \(l\) & $L_2^{u_x}$ & $L_2^{u_y}$ & \(L_2^{\Psi}\)&  $r_h^{u_x}$ & $r_h^{u_y}$ & \(r_h^{\Psi}\)   \\
			\hline
			0 & \(5.51\e{-04}\) &	\(5.51\e{-04}\) &	\(9.24\e{-02}\) & & & \\
			1& \(1.33\e{-04}\) &	\(1.47\e{-04}\) &	\(1.72\e{-02}\) &	2.05  & 1.9	 &	2.42 \\
			2 & \(3.3\e{-05}\) &	\(3.45\e{-05}\) &	\(4.15\e{-03}\) &	2.01 & 2.1	&	2.05  \\
			3 & \(7.9\e{-06}\) &	\(8.16\e{-06}\) &	\(1.4\e{-03}\) &	2.06 &	2.08 &	2 \\
			4 & \(1.88\e{-06}\) &	\(1.98\e{-06}\) &	\(2.57\e{-04}\) &	2.07 &	2.05 &	2.01  \\
			5 &\(4.69\e{-07}\) &	\(4.81\e{-07}\) &	\(6.4\e{-05}\) & 2	 &	2.04 &	2.01 \\
			\hline		
		\end{tabular}
		\label{error_RG_1}
	\end{table}
	
	\begin{table} 
		\small
		\caption{Test 1, 2D case, method 1. \(L_2\)-norms of errors and rates in space: \(T_{h,R}\), BCs case 2.}
		\centering
		\begin{tabular}{c c c c c c c}
			ref. level \(l\) & $L_2^{u_x}$ & $L_2^{u_y}$ & \(L_2^{\Psi}\)&  $r_h^{u_x}$ & $r_h^{u_y}$ & \(r_h^{\Psi}\)   \\
			\hline
			0 & \(8.39\e{-04}\) &	\(8.05\e{-04}\) &	\(1.38\e{-02}\) & & & \\
			1& \(1.18\e{-04}\) &	\(1.16\e{-04}\) &	\(3.11\e{-02}\) &	2.83  & 2.79	 &	2.15 \\
			2 & \(2.02\e{-05}\) &	\(2.13\e{-05}\) &	\(7.7\e{-03}\) &	2.55 & 2.45	&	2.01  \\
			3 & \(4.32\e{-06}\) &	\(4.76\e{-06}\) &	\(1.91\e{-03}\) &	2.22 &	2.16 &	2.01 \\
			4 & \(1\e{-06}\) &	\(1.1\e{-06}\) &	\(4.78\e{-04}\) &	2.11 &	2.11 &	2  \\
			5 &\(2.43\e{-07}\) &	\(2.59\e{-07}\) &	\(1.21\e{-04}\) & 2.04	 &	2.09 &	1.99 \\
			\hline		
		\end{tabular}
		\label{error_RG_2}
	\end{table}
	
	\begin{table} 
		\small
		\caption{Test 1, 2D case, method 1. \(L_2\)-norms of errors and rates in space: \(T_{h,D}\), BCs case 1.}
		\centering
		\begin{tabular}{c c c c c c c}
			ref. level \(l\) & $L_2^{u_x}$ & $L_2^{u_y}$ & \(L_2^{\Psi}\)&  $r_h^{u_x}$ & $r_h^{u_y}$ & \(r_h^{\Psi}\)   \\
			\hline
			0 & \(3.06\e{-03}\) &	\(2.64\e{-03}\) &	\(3.4\e{-01}\) & & & \\
			1& \(4.39\e{-04}\) &	\(3.81\e{-04}\) &	\(7.7\e{-02}\) &	2.8  & 2.8	 &	2.14 \\
			2 & \(1.02\e{-04}\) &	\(8.79\e{-05}\) &	\(1.88\e{-02}\) &	2.1 & 2.12	&	2.04  \\
			3 & \(2.4\e{-05}\) &	\(2.04\e{-05}\) &	\(4.5\e{-03}\) &	2.09 &	2.11 &	2.06 \\
			4 & \(5.6\e{-06}\) &	\(4.99\e{-06}\) &	\(1.12\e{-03}\) &	2.1 &	2.03 &	2.01 \\
			5 &\(1.35\e{-06}\) &	\(1.24\e{-06}\) &	\(2.8\e{-04}\) & 2.05	 &	2.01 &	2 \\
			\hline		
		\end{tabular}
		\label{error_DG_1}
	\end{table}
	
	\begin{table} 
		\small
		\caption{Test 1, 2D case, method 1. \(L_2\)-norms of errors and rates in space: \(T_{h,D}\), BCs case 2.}
		\centering
		\begin{tabular}{c c c c c c c}
			ref. level \(l\) & $L_2^{u_x}$ & $L_2^{u_y}$ & \(L_2^{\Psi}\)&  $r_h^{u_x}$ & $r_h^{u_y}$ & \(r_h^{\Psi}\)   \\
			\hline
			0 & \(6.07\e{-03}\) &	\(7\e{-03}\) &	\(7.1\e{-01}\) & & & \\
			1& \(8.01\e{-04}\) &	\(9.09\e{-04}\) &	\(1.6\e{-01}\) &	2.92  & 2.95	 &	2.15 \\
			2 & \(1.08\e{-04}\) &	\(1.2\e{-04}\) &	\(4\e{-02}\) &	2.89 & 2.92	&	2  \\
			3 & \(2.06\e{-05}\) &	\(2.08\e{-05}\) &	\(9.98\e{-03}\) &	2.39 &	2.53 &	2 \\
			4 & \(4.79\e{-06}\) &	\(4.9\e{-06}\) &	\(2.48\e{-03}\) &	2.1 &	2.09 &	2.01 \\
			5 &\(1.19\e{-06}\) &	\(1.17\e{-06}\) &	\(6.24\e{-04}\) & 2.01	 &	2.06 &	1.99 \\
			\hline		
		\end{tabular}
		\label{error_DG_2}
	\end{table}
	
	\begin{table} 
		\small
		\caption{Test 1, 2D case, method 1. \(L_2\)-norms of errors and rates in time: \(T_{h,R}\), BCs case 1.}
		\centering
		\begin{tabular}{c c c c c c c}
			\(\Delta t\) & $L_2^{u_x}$ & $L_2^{u_y}$ & \(L_2^{\Psi}\)&  $r_{\Delta t}^{u_x}$ & $r_{\Delta t}^{u_y}$ & \(r_{\Delta t}^{\Psi}\)   \\
			\hline
			\(1\e{01}\) & \(8.2\e{-03}\) &	\(8.4\e{-03}\) &	\(1.1\e{00}\) & & & \\
			\(1\e{02}\)& \(8.1\e{-04}\) &	\(8.3\e{-04}\) &	\(1.4\e{-01}\) &	1.01  & 1.01	 &	1.02 \\
			\(1\e{03}\) & \(7.99\e{-05}\) &	\(8.2\e{-05}\) &	\(1.04\e{-02}\) &	1.01 & 1.01	&	1  \\
			\(1\e{04}\) & \(7.9\e{-06}\) &	\(8.16\e{-06}\) &	\(1.04\e{-03}\) &1 &	1 &	1 \\
			\hline		
		\end{tabular}
		\label{error_time_RG_1}
	\end{table}
	
	\begin{table} 
		\small
		\caption{Test 1, 2D case, method 1. \(L_2\)-norms of errors and rates in time: \(T_{h,R}\), BCs case 2.}
		\centering
		\begin{tabular}{c c c c c c c}
			\(\Delta t\) & $L_2^{u_x}$ & $L_2^{u_y}$ & \(L_2^{\Psi}\)&  $r_{\Delta t}^{u_x}$ & $r_{\Delta t}^{u_y}$ & \(r_{\Delta t}^{\Psi}\)   \\
			\hline
			\(1\e{01}\) & \(4.41\e{-03}\) &	\(4.93\e{-03}\) &	\(1.93\e{00}\) & & & \\
			\(1\e{02}\)& \(4.35\e{-04}\) &	\(4.87\e{-04}\) &	\(1.9\e{-01}\) &	1.01  & 1.01	 &	1.01 \\
			\(1\e{03}\) & \(4.31\e{-05}\) &	\(4.87\e{-05}\) &	\(1.89\e{-02}\) &	1 & 1	&	1  \\
			\(1\e{04}\) & \(4.32\e{-06}\) &	\(4.76\e{-06}\) &	\(1.91\e{-03}\) &0.99 &	1.01 &	0.996 \\
			\hline		
		\end{tabular}
		\label{error_time_RG_2}
	\end{table}
	
	\begin{table} 
		\small
		\caption{Test 1, 2D case, method 1. \(L_2\)-norms of errors and rates in time: \(T_{h,D}\), BCs case 1.}
		\centering
		\begin{tabular}{c c c c c c c}
			\(\Delta t\) & $L_2^{u_x}$ & $L_2^{u_y}$ & \(L_2^{\Psi}\)&  $r_{\Delta t}^{u_x}$ & $r_{\Delta t}^{u_y}$ & \(r_{\Delta t}^{\Psi}\)   \\
			\hline
			\(1\e{01}\) & \(2.4\e{-02}\) &	\(2.1\e{-02}\) &	\(4.71\e{00}\) & & & \\
			\(1\e{02}\)& \(2.39\e{-03}\) &	\(2.04\e{-03}\) &	\(4.58\e{-01}\) &	1  & 1.01	 &	1.01 \\
			\(1\e{03}\) & \(2.43\e{-04}\) &	\(2.07\e{-04}\) &	\(4.6\e{-02}\) &	0.994 & 0.994	&	0.998  \\
			\(1\e{04}\) & \(2.4\e{-05}\) &	\(2.04\e{-05}\) &	\(4.5\e{-03}\) &1.01 &	1.01 &	1.01 \\
			\hline		
		\end{tabular}
		\label{error_time_DG_1}
	\end{table}
	
	\begin{table} 
		\small
		\caption{Test 1, 2D case, method 1. \(L_2\)-norms of errors and rates in time: \(T_{h,D}\), BCs case 2.}
		\centering
		\begin{tabular}{c c c c c c c}
			\(\Delta t\) & $L_2^{u_x}$ & $L_2^{u_y}$ & \(L_2^{\Psi}\)&  $r_{\Delta t}^{u_x}$ & $r_{\Delta t}^{u_y}$ & \(r_{\Delta t}^{\Psi}\)   \\
			\hline
			\(1\e{01}\) & \(2.23\e{-02}\) &	\(2.21\e{-02}\) &	\(9.97\e{00}\) & & & \\
			\(1\e{02}\)& \(2.14\e{-03}\) &	\(2.13\e{-03}\) &	\(1\e{00}\) &	1.02  & 1.03  &	0.998 \\
			\(1\e{03}\) & \(2.15\e{-04}\) &	\(2.15\e{-04}\) &	\(1.01\e{-01}\) &	0.998 & 0.997	&	0.997  \\
			\(1\e{04}\) & \(2.06\e{-05}\) &	\(2.08\e{-05}\) &	\(9.98\e{-03}\) &1.02 &	1.01 &	1.01 \\
			\hline		
		\end{tabular}
		\label{error_time_DG_2}
	\end{table}
	
	\begin{table} 
		\small
		\caption{Test 1, 2D case, method 2. \(L_2\)-norms of errors and rates in space: \(T_{h,D}\), BCs case 2.}
		\centering
		\begin{tabular}{c c c c c c c}
			ref. level \(l\) & $L_2^{u_x}$ & $L_2^{u_y}$ & \(L_2^{\Psi}\)&  $r_h^{u_x}$ & $r_h^{u_y}$ & \(r_h^{\Psi}\)   \\
			\hline
			0 & \(6.27\e{-03}\) &	\(7.9\e{-03}\) &	\(8\e{-01}\) & & & \\
			1& \(8.32\e{-04}\) &	\(1.1\e{-03}\) &	\(1.82\e{-01}\) &	2.91  & 2.84	 &	2.14 \\
			2 & \(1.35\e{-04}\) &	\(1.5\e{-04}\) &	\(4.53\e{-02}\) &	2.62 & 2.87	&	2.01  \\
			3 & \(2.48\e{-05}\) &	\(2.6\e{-05}\) &	\(1.1\e{-02}\) &	2.44 &	2.53 &	2.04 \\
			4 & \(5.02\e{-06}\) &	\(5.78\e{-06}\) &	\(2.63\e{-03}\) &	2.3 &	2.17 &	2.06 \\
			5 &\(1.23\e{-06}\) &	\(1.39\e{-06}\) &	\(6.45\e{-04}\) & 2.03	 &	2.06 &	2.03 \\
			\hline		
		\end{tabular}
		\label{error_DG_2_model2}
	\end{table}
	
	\begin{table} 
		\small
		\caption{Test 1, 2D case, method 2. \(L_2\)-norms of errors and rates in time: \(T_{h,D}\), BCs case 2.}
		\centering
		\begin{tabular}{c c c c c c c}
			\(\Delta t\) & $L_2^{u_x}$ & $L_2^{u_y}$ & \(L_2^{\Psi}\)&  $r_{\Delta t}^{u_x}$ & $r_{\Delta t}^{u_y}$ & \(r_{\Delta t}^{\Psi}\)   \\
			\hline
			\(1\e{01}\) & \(2.57\e{-02}\) &	\(2.615\e{-02}\) &	\(11.9\e{00}\) & & & \\
			\(1\e{02}\)& \(2.505\e{-03}\) &	\(2.59\e{-03}\) &	\(1.1\e{00}\) &	1.01  & 1  & 1.03 \\
			\(1\e{03}\) & \(2.49\e{-04}\) &	\(2.589\e{-04}\) &	\(1.107\e{-01}\) &	1 & 1	&	1  \\
			\(1\e{04}\) & \(2.48\e{-05}\) &	\(2.6\e{-05}\) &	\(1.1\e{-02}\) &1 &	0.998 &	1 \\
			\hline		
		\end{tabular}
		\label{error_time_DG_2_model2}
	\end{table}

	\subsubsection{3D study}
	We present a 3D version of the study presented in \cref{2D_analyt}, where \(\Omega\) is a spherical heterogeneous porous medium, and a generic section is shown in \cref{test_1_sett_b}, with the two bulk porous regions 1 and 2, whose physical properties have the values listed in \cref{test1_bulk_par_3D}, and a transition layer in between (zone 3) where the physical properties of the porous medium change as in \eqref{profiles}, according to the values listed in \cref{test1_transzone_par_3D}. The (dimensionless) radii of zones 1, 3, and 2 are 0.2, 0.4, and 0.6, respectively.
	
	The assigned analytical solution is 
	\begin{equation} \label{an_sol_3D}
		\Psi= \left(a_1 + b_1 x + c_1 x^2 + a_2 + b_2 y + c_2 y^2 + a_3 + b_3 z + c_3 z^2\right) \cos \omega t, \quad \bu = - \mathbf{K} \nabla \Psi,
	\end{equation}
	with coefficients \(a_i, b_i, c_i\) listed in \cref{test1_par_ansol_3D}. The maximum value of the Reynolds number is \(Re = \frac{\| \bu\|_{max}\ 2\ R_{out}}{\nu} \simeq 3.75\e{-04}\).
	
	The domain \(\Omega\) is discretized by a coarse unstructured tetrahedral grid with 1181 simplices and 278 vertices. The maximum aspect ratio, given for each simplex \(E \in T_h\) by the ratio between the maximum edge length and the minimum height, is 5.35, where the ideal value of a regular tetrahedron is \(\simeq 1.225\). Again, two scenarios of BCs have been considered: the first where the boundary faces on \(\Gamma_n\) and \(\Gamma_d\) lie to the left and right sides of the plane \(x=0\), respectively, and the second, where only one face is on \(\Gamma_n\). Due to computational limitations, we perform only two grid refinements. \par 
	
	As in the 2D case, we study the convergence order in space for both BCs cases running simulations using a (dimensionless) time step size \(\Delta t = 1\e{-04}\). The ICs are obtained as in \cref{2D_analyt}. In \cref{error_RG_1_3D,error_RG_2_3D} we list the \(L_2\)-norms of the errors of the velocity components and the pressure as well as the associated convergence order \(r_h^q\). The convergence order is higher than 2, but we expect that it will approach 2 if more refinements were considered. For the convergence order in time, we run simulations over the \(2^{nd}\) refined grid level, using a (dimensionless) time step size \(\Delta t\) in the range \(\left[1\e{-04}, 1\e{-01}\right]\). In \cref{error_RG_1_3D_time,error_RG_2_3D_time} we list the computed \(L_2\)-norms of the errors and the associated rates \(r_{\Delta t}^q\). As in the 2D case, the convergence order in time is approximately equal to 1. \par
	
	Generally, all results listed in \crefrange{error_RG_1}{error_time_DG_2_model2} and \crefrange{error_RG_1_3D}{error_RG_2_3D_time} illustrate the unconditional stability of the method in time and space, in line with \cref{thm:stab,thm:stab-simpl}.

	\begin{table} 
		\caption{Test 1, 3D case. (Dimensionless) parameters of the bulk porous regions}.
		\centering
		\begin{tabular}{c c c c c c c c}
			zone & \(\phi\) & \(\mathfrak{K}_{1,1}\) & \(\mathfrak{K}_{2,2}\) & \(\mathfrak{K}_{3,3}\) & \(\mathfrak{K}_{1,2} = \mathfrak{K}_{2,1} \) & \(\mathfrak{K}_{1,3} = \mathfrak{K}_{3,1}\) & \(\mathfrak{K}_{2,3} = \mathfrak{K}_{3,2}\) \\
			\hline
			1 & 0.597	& 716.94 &	251.70 & 566.57 & 82.24 & -48.177 & 280.8 \\	
			2 & 0.7037	& 263.75 &	92.6 & 208.43 & 30.25 & -17.72 & 103.3  \\
			\hline		
		\end{tabular}
		\label{test1_bulk_par_3D}
	\end{table}
	
	\begin{table} 
		\caption{Test 1, 3D case. (Dimensionless) parameters of the transition porous zone.}
		\centering
		\begin{tabular}{c c c c c c c c c c c}
			zone & \(\phi_{max}\) & \(\phi_{min}\) & \(\theta_{\psi}\) & \(\theta_{\mathfrak{K}_{i,j}}\) & \(\mathfrak{K}^0_{1,1}\) & \(\mathfrak{K}^0_{2,2}\) & \(\mathfrak{K}^0_{3,3}\) & \(\mathfrak{K}^0_{1,2} = \mathfrak{K}^0_{2,1}\) & \(\mathfrak{K}^0_{1,3} = \mathfrak{K}^0_{3,1}\) & \(\mathfrak{K}^0_{2,3} = \mathfrak{K}^0_{3,2}\) \\
			\hline
			3 & 0.766	& 0.5344 &	5 & 5 & 980.7 & 344.3 & 775 & 112.5 & -65.9 & 384.1 \\	
			\hline		
		\end{tabular}
		\label{test1_transzone_par_3D}
	\end{table}
	
	\begin{table}  \footnotesize 
		\caption{Test 1, 3D case. (Dimensionless) parameters of the assigned analytical solution in \eqref{profiles}.}
		\centering
		\begin{tabular}{@{\,}c @{\,}c @{\,}c @{\,}c @{\,}c @{\,}c @{\,}c @{\,}c @{\,}c @{\,}c } 
			\(a_1\) & \(b_1\) & \(c_1\) & \(a_2\) & \(b_2\) & \(c_2\) & \(a_3\) & \(b_3\) & \(c_3 \) & \(\omega\)\\
			\hline
			\(5\e{-04}\) & \(1.5\e{-03}\) & \(1.3\e{-03}\)	& \(-3.2\e{-03}\) &	\(-8\e{-04}\) & \(1.5\e{-04}\) & \(-1.95\e{-03}\) & \(1.95\e{-03}\) & \(2.45\e{-03}\) & \(\pi\) \\
			\hline		
		\end{tabular}
		\label{test1_par_ansol_3D}
	\end{table}
	
	\begin{table} 
		\small
		\caption{Test 1, 3D case, method 1. \(L_2\)-norms of errors and rates in space for BCs case 1.}
		\centering
		\begin{tabular}{c c c c c c c c c }
			ref. level \(l\) & $L_2^{u_x}$ & $L_2^{u_y}$ & $L_2^{u_z}$ & \(L_2^{\Psi}\)&  $r_h^{u_x}$ & $r_h^{u_y}$ & $r_h^{u_z}$ & \(r_h^{\Psi}\)   \\
			\hline
			0 & \(6.48\e{-03}\) &	\(1.68\e{-02}\) & \(1.44\e{-02}\) &	\(2.42\e{00}\) & & & & \\
			1& \(7.4\e{-04}\) &	\(7.9\e{-04}\) & \(8.4\e{-04}\) & 	\(4.97\e{-01}\) &	3.13  & 4.41 &	4.1 & 2.29 \\
			2 & \(9.19\e{-05}\) &	\(1.15\e{-04}\) & \(1.7\e{-04}\) &	\(8.29\e{-02}\) &	3.01 & 2.78	& 2.3 &	2.58  \\
			\hline		
		\end{tabular}
		\label{error_RG_1_3D}
	\end{table}
	
	\begin{table} 
		\small
		\caption{Test 1, 3D case, method 1. \(L_2\)-norms of errors and rates in space for BCs case 2.}
		\centering
		\begin{tabular}{c c c c c c c c c }
			ref. level \(l\) & $L_2^{u_x}$ & $L_2^{u_y}$ & $L_2^{u_z}$ & \(L_2^{\Psi}\)&  $r_h^{u_x}$ & $r_h^{u_y}$ & $r_h^{u_z}$ & \(r_h^{\Psi}\)   \\
			\hline
			0 & \(6.85\e{-03}\) &	\(1.41\e{-02}\) & \(1.31\e{-02}\) &	\(2.45\e{00}\) & & & & \\
			1& \(7.4\e{-04}\) &	\(7.9\e{-04}\) & \(8.4\e{-04}\) & 	\(4.97\e{-01}\) &	3.21  & 4.15 &	3.97 & 2.3 \\
			2 & \(9.28\e{-05}\) &	\(1.1\e{-04}\) & \(1.7\e{-04}\) &	\(8.29\e{-02}\) &	3 & 2.84	& 2.3 &	2.58  \\
			\hline		
		\end{tabular}
		\label{error_RG_2_3D}
	\end{table}
	
	\begin{table} 
		\small
		\caption{Test 1, 3D case, method 1. \(L_2\)-norms of errors and rates in time for BCs case 1.}
		\centering
		\begin{tabular}{c c c c c c c c c }
			\(\Delta t\) & $L_2^{u_x}$ & $L_2^{u_y}$ & $L_2^{u_z}$ & \(L_2^{\Psi}\)&  $r_{\Delta t}^{u_x}$ & $r_{\Delta t}^{u_y}$ & $r_{\Delta t}^{u_z}$ & \(r_{\Delta t}^{\Psi}\)   \\
			\hline
			\(1\e{01}\) & \(1.1\e{-01}\) &	\(1.18\e{-01}\) & \(1.78\e{-01}\) &	\(7.8\e{01}\) & & & & \\
			\(1\e{02}\)& \(9.96\e{-03}\) &	\(1.18\e{-02}\) & \(1.74\e{-02}\) & 	\(7.97\e{00}\) &	1.04  & 1 & 1.01 & 0.991 \\
			\(1\e{03}\) & \(9.4\e{-04}\) &	\(1.16\e{-03}\) & \(1.72\e{-03}\) &	\(8.15\e{-01}\) &	1.03 & 1.01	& 1.01 &	0.99  \\
			\(1\e{04}\)& \(9.19\e{-05}\) &	\(1.15\e{-04}\) & \(1.7\e{-04}\) &	\(8.29\e{-02}\) &	1.01 & 1	& 1.01 &	0.993 \\
			\hline		
		\end{tabular}
		\label{error_RG_1_3D_time}
	\end{table}
	
	\begin{table} 
		\small
		\caption{Test 1, 3D case, method 1. \(L_2\)-norms of errors and rates in time for BCs case 2.}
		\centering
		\begin{tabular}{c c c c c c c c c }
			\(\Delta t\) & $L_2^{u_x}$ & $L_2^{u_y}$ & $L_2^{u_z}$ & \(L_2^{\Psi}\)&  $r_{\Delta t}^{u_x}$ & $r_{\Delta t}^{u_y}$ & $r_{\Delta t}^{u_z}$ & \(r_{\Delta t}^{\Psi}\)   \\
			\hline
			\(1\e{01}\) & \(9.64\e{-02}\) &	\(1.2\e{-01}\) & \(1.69\e{-01}\) &	\(8.27\e{01}\) & & & & \\
			\(1\e{02}\)& \(9.61\e{-03}\) &	\(1.17\e{-02}\) & \(1.68\e{-02}\) & 	\(8.3\e{00}\) &	1  & 1.01 & 1 & 0.998 \\
			\(1\e{03}\) & \(9.35\e{-04}\) &	\(1.15\e{-03}\) & \(1.69\e{-03}\) &	\(8.3\e{-01}\) &	1.01 & 1.01	& 0.996 &	1  \\
			\(1\e{04}\)& \(9.28\e{-05}\) &	\(1.1\e{-04}\) & \(1.7\e{-04}\) &	\(8.29\e{-02}\) &	1 & 1.02	& 0.999 &	1 \\
			\hline		
		\end{tabular}
		\label{error_RG_2_3D_time}
	\end{table}
	

	\subsection{Test 2: free fluid flow around a porous obstacle under different working conditions} \label{test2}
	The setting of this numerical experiment is shown in \cref{fig_test2_sett}, where a fluid hits a porous obstacle acting as a filter. This test has been proposed in \cite{Iliev2004} and further studied in \cite{SCHNEIDER2020109012, Arico-ODA}. The flow is driven by a left-to-right pressure drop \(\left(\Psi_1 - \Psi_2\right)=1\e{-06} \textnormal{m}^2/\textnormal{s}^2\). The kinematic fluid viscosity is \(\nu = 1.5 \e{-05}\) \nuu. In the free fluid region we set $\phi = 1$ and $\mathbf{K} = \infty$, i.e. $\mathfrak{K} = 0$, and test different values of $\phi$ and $\mathbf{K}$ in the porous region, with $\mathbf{K}$ computed as in \eqref{perm_bulk}. There is no transition region between the free fluid and porous regions.
	
	\begin{figure}
		\centering
		\includegraphics[width=0.44\linewidth]{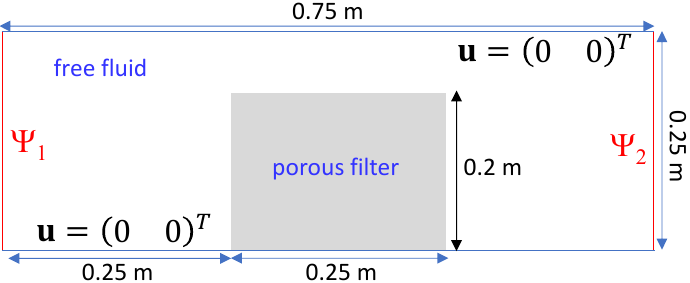} 
		\includegraphics[width=0.265\linewidth]{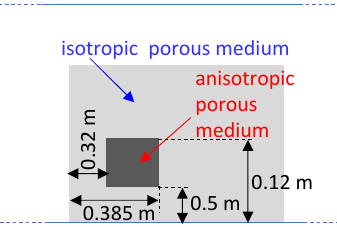}
		\includegraphics[width=0.28\linewidth]{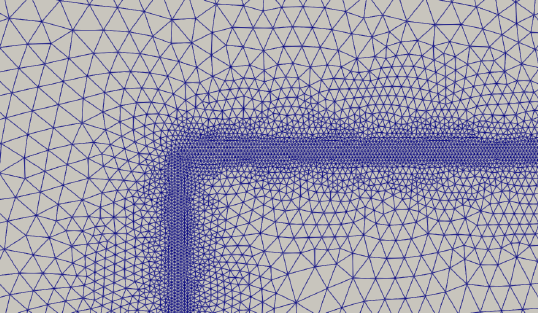}
		\caption{Test 2, settings of the numerical experiment. Homogeneous porous medium (left), heterogeneous porous medium (center), and zoom of the computational grid near the interface (right).}
		\label{fig_test2_sett}
	\end{figure}
	

A channelized flow is established above the porous filter and the maximum value of the Reynolds number \(Re\), computed according the to maximum flow velocity in the fluid region, the depth of the fluid channel above the obstacle, and the kinematic fluid viscosity, is \(Re \in \left[0.2, 0.22\right]\) and 2, 3 magnitude orders smaller in the porous block. Generally, in all the scenarios presented for this test case, the fluid flows primarily around the porous obstacle, being upward oriented before entering the obstacle and downward oriented after exiting it. \par

The computational grid \(T_h\) has 161105 triangles and 80991 vertices, and the grid size ranges from \(\simeq 2.25\e{-03}\) m, in the bulk fluid and porous regions, to \(\simeq 2.82 \e{-04}\) m, close to the interface. The time step size \(\Delta t\) is 4 s. The ICs are zero velocity and pressure in the domain. In the following we show the results at the simulation time \(\mathcal{T} = 80\) s, when the steady state conditions are established. \par

In \cref{homog} we consider homogeneous porous filter (see \cref{fig_test2_sett}, left) and different porous medium anisotropy conditions, while in \cref{hetero} we investigate how the method simulates a heterogeneous filter (see \cref{fig_test2_sett}, right) with strong contrast of porosity and permeability values.

\subsubsection{Homogeneous porous medium} \label{homog}
The porosity in the porous filter is $\phi = 0.4$, while the permeability $\mathbf{K}$ is computed as in \eqref{perm_bulk}. \cref{fig:test_2_2} shows the velocity and pressure fields computed for method 1, setting \(k=1 \e{-06}\) \msq, \(\beta=100\), and \(\alpha = \pm \pi/4\). Due to the anisotropy effects in the porous filter, the flow, oriented according to the principal direction of the permeability tensor, exits or enters at the top horizontal side, if \(\alpha = -\pi/4\) or \(\alpha = \pi/4\), respectively. The pressure profiles are also aligned with the principal direction of the permeability tensor. The pressure exhibits very slight discontinuity across the top horizontal interface. In the case of \(\alpha = -\pi/4\), the flow inside the filter is upward oriented. The opposite flow directions inside and outside the filter along the right interface create small vortices at the bottom of that interface. In the case of \(\alpha = \pi/4\), the flow inside the filter is downward oriented. The opposite flow directions outside and intside the filter along the left interface 
create small vortices at the bottom of that interface. We obtained very similar results using smaller time step sizes \(\Delta t\), not shown here for brevity.
A very good agreement is observed with the results provided in \cite[Fig. 5]{SCHNEIDER2020109012} and \cite[Fig. 18]{Arico-ODA}. In \cite{SCHNEIDER2020109012} the authors apply a two-domain approach (TDA), coupling the Navier--Stokes equations (for compressible fluid) in the fluid region with the Darcy equation in the porous domain via the Beavers--Joseph--Saffman slip condition \cite{Beavers_Joseph_1967}, as well as conservation of mass and momentum, at the interface. They adopt a staggered-grid finite volume method for the Navier--Stokes equations and MPFA for Darcy flow, using uniform square grids grid with size \(5 \e{-04}\) m.
In \cite{Arico-ODA}, a one-domain approach (ODA) with a transition layer is used to solve the Navier--Stokes--Brinkman equations for an incompressible fluid using a three-step predictor-corrector method. A finite volume method is applied for the predictor step, accounting for the inertial terms of the momentum equation, and two corrections are then performed to update the viscous and gradient pressure terms using a mass-lumped mixed hybrid finite element scheme with \(RT_0\) basis functions. An unstructured triangular grid is used with mesh size ranging from \(5 \e{-03}\) m in the bulk fluid and porous regions to \(1 \e{-04}\) m in the transition layer. More details can be found in \cite{SCHNEIDER2020109012,Arico-ODA}. \par

\begin{figure}
	\centering
	\includegraphics[width=0.47\linewidth]{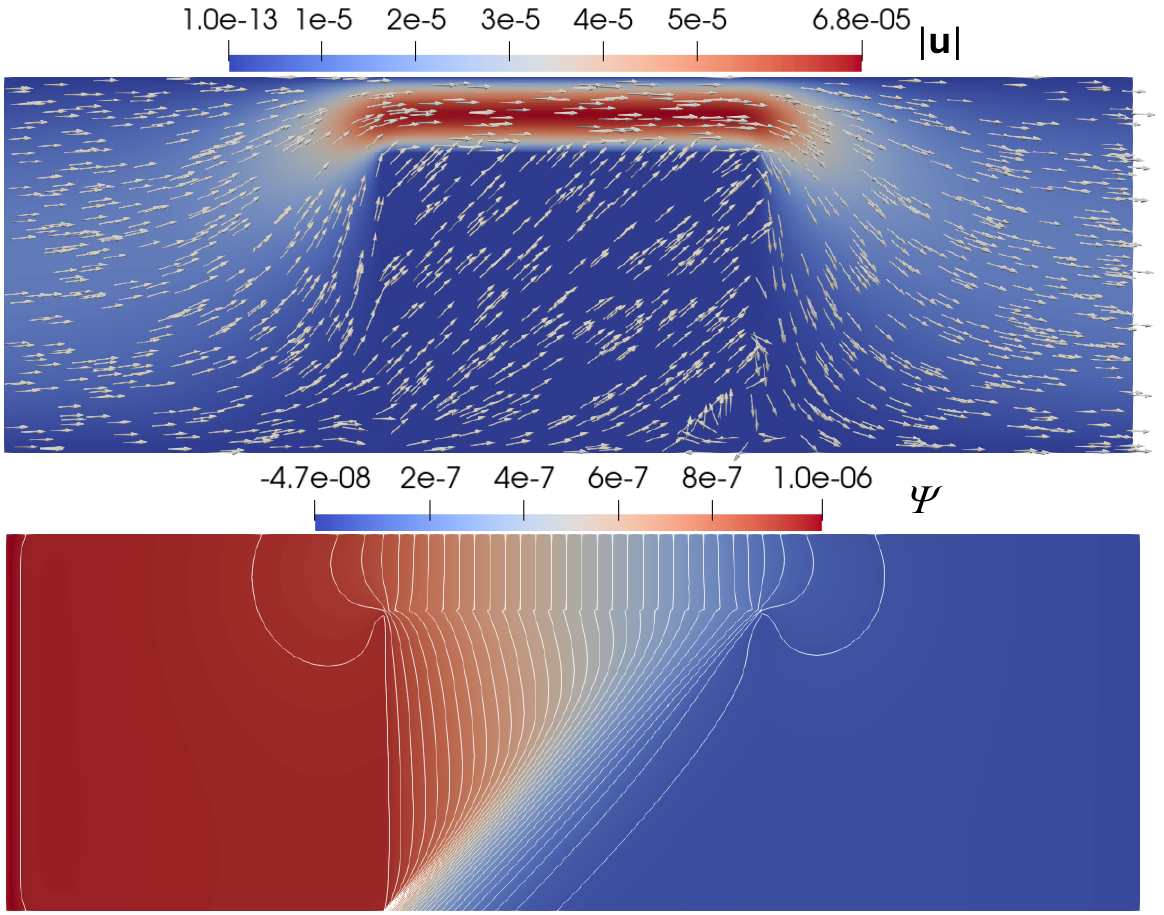}  \includegraphics[width=0.46\linewidth]{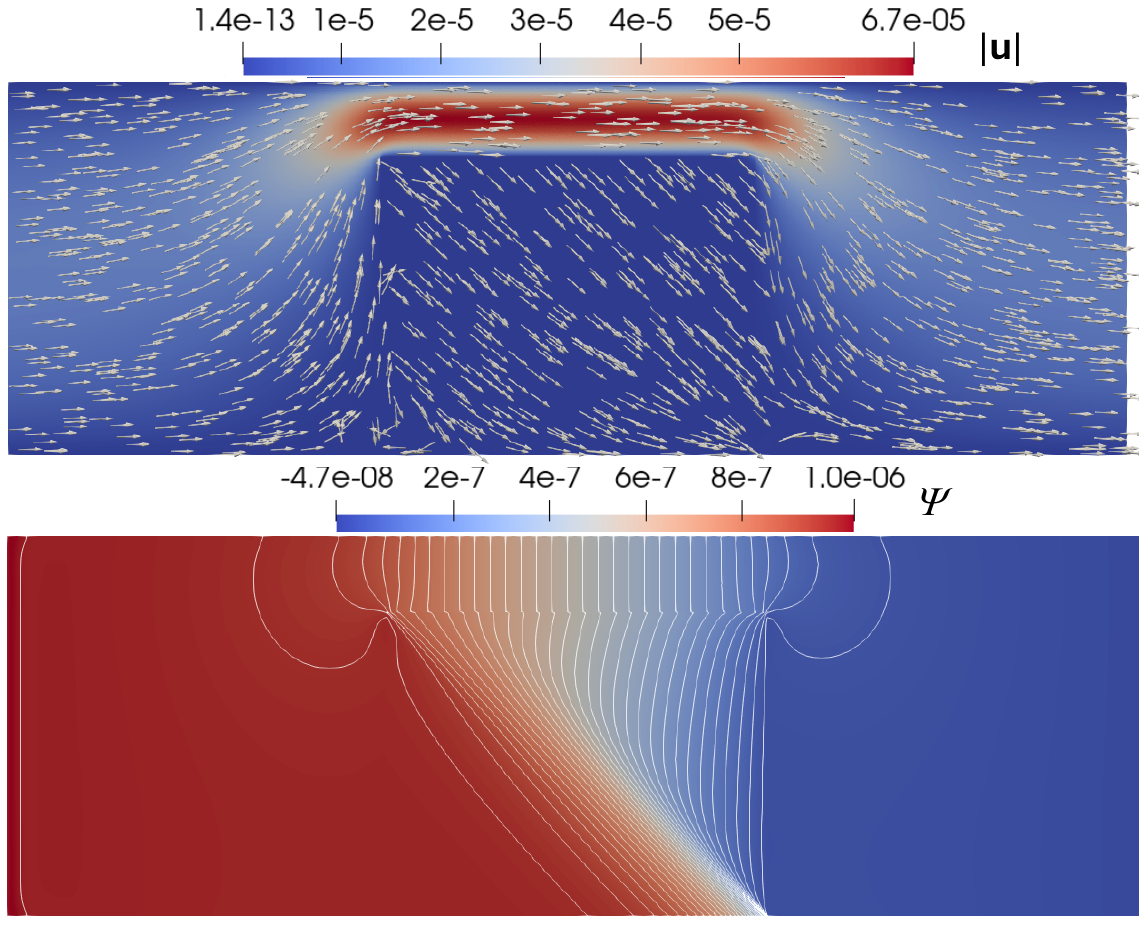}
	\caption{Test 2, homogeneous porous medium, \(k=1\e{-06}\), \(\beta = 100\).
		Computed velocity (top) and pressure (bottom) with method 1. Left: \(\alpha= -\pi / 4 \), right: \(\alpha= \pi / 4 \).}
	\label{fig:test_2_2}
\end{figure}

In \cref{fig:test_2_3} we compare the velocity fields computed with method 1 (top) and method 2 (bottom). We observe irregular and non-physical vortical structures with method 2. We also ran a simulation using method 2 with $\Delta t = 0.2$ s, obtaining results very similar to these on the top row of \cref{fig:test_2_3} produced with method 1 and $\Delta t = 4$ s. These results illustrate that, while method 2 is convergent, it may exhibit non-physical behavior for large time steps and requires smaller time steps than method 1 to produce an accurate solution. This is consistent with the fact that method 2 does not include the permeability correction term in the \PjP\ and the pressure gradient update, cf. \eqref{eq:Projecproblem_1-simpl}--\eqref{eq:up_q-simpl}, which results in additional time-splitting error.

\begin{figure}
	\centering
	\includegraphics[width=0.475\linewidth]{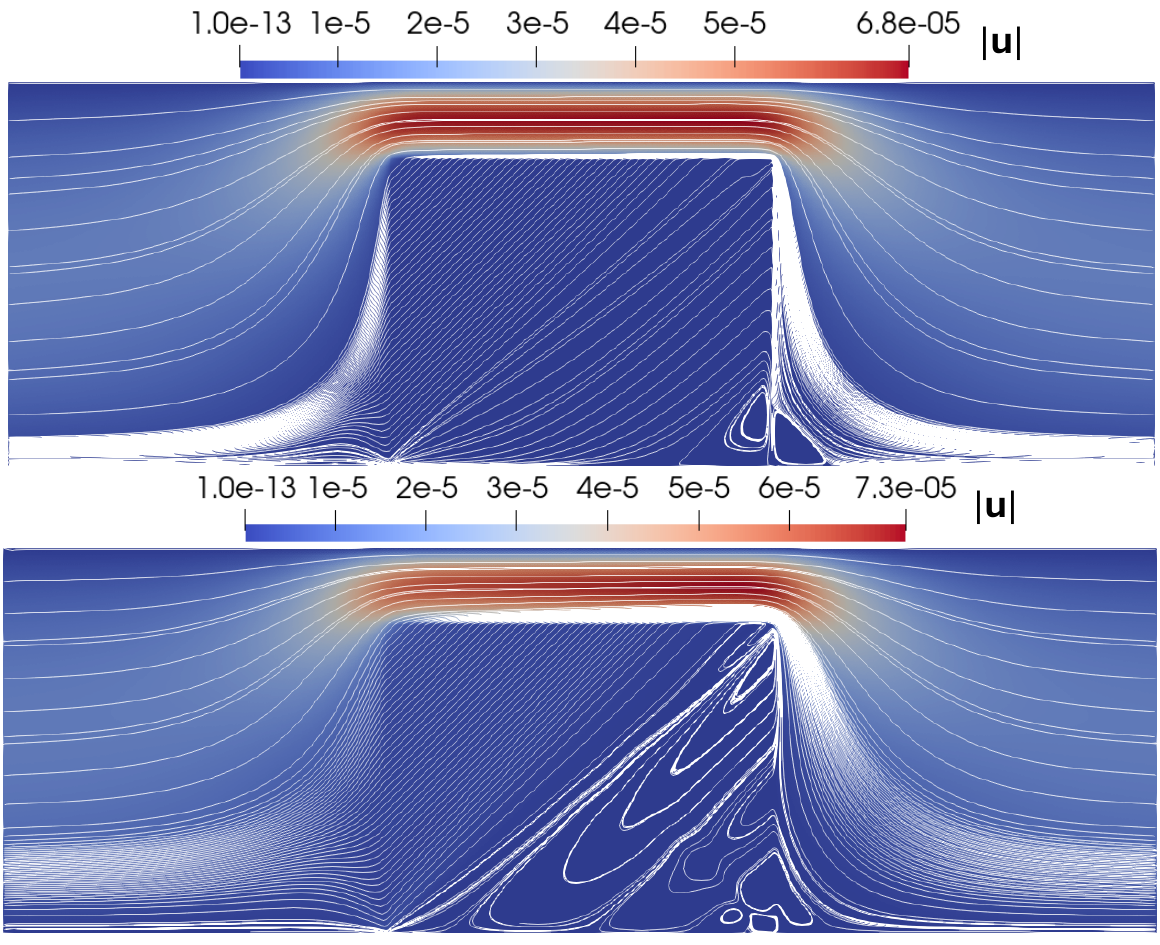}  \includegraphics[width=0.475\linewidth]{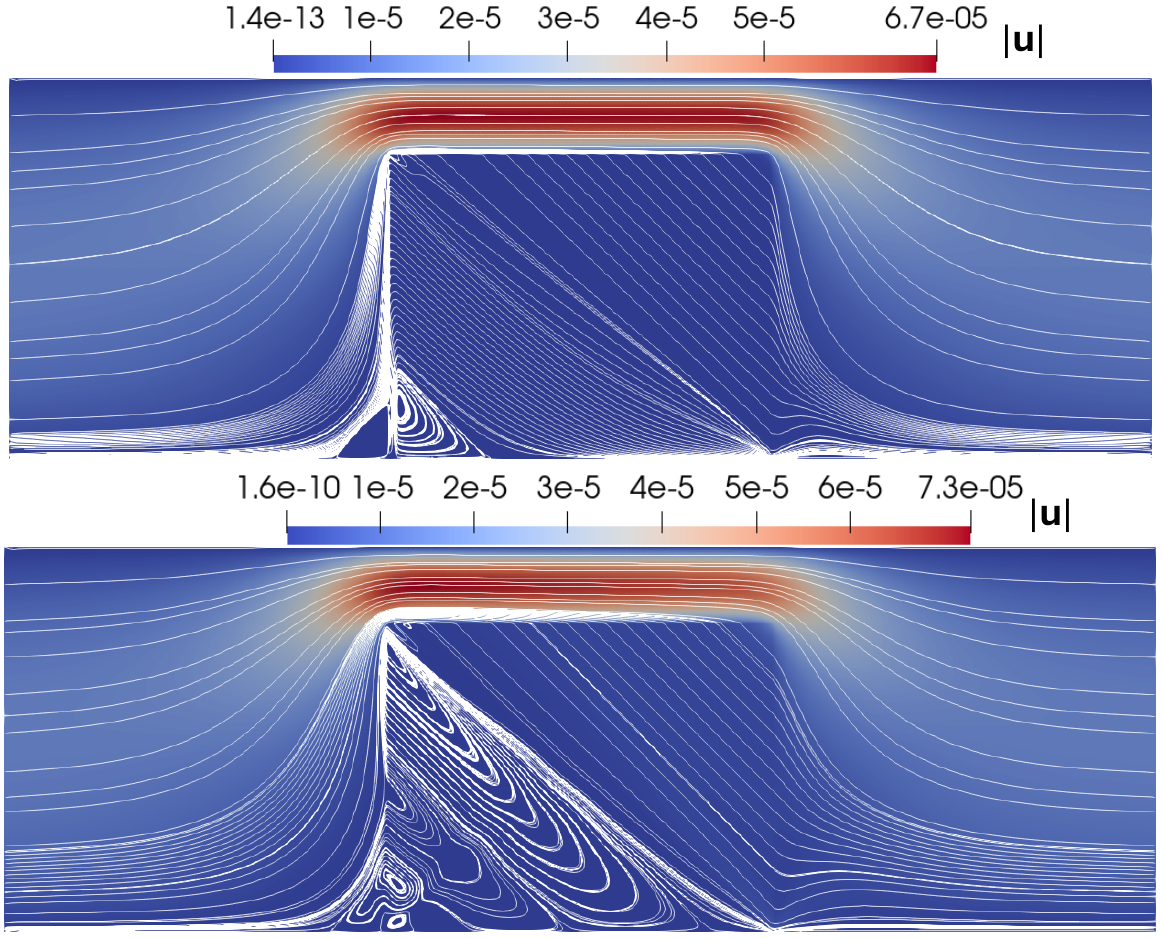}
	\caption{Test 2, homogeneous porous medium, \(k=1\e{-06}\), \(\beta = 100\).
		Computed velocity with method 1 (top) and method 2 (bottom). Left: \(\alpha= -\pi / 4 \), right: \(\alpha= \pi / 4 \).}
	\label{fig:test_2_3}
\end{figure}

In \cref{fig:test_2_4} we show the results obtained by method 1 in the case of a strong vertical anisotropy, with \(\alpha = 0\), \(k =1\e{-06}\) \msq\ and \(\beta=1000\), in which case the horizontal permeability is three orders of magnitude smaller than the vertical permeability. The fluid enters the filter primarily from the left side, mostly vertically upward oriented, concentrated in very thin stripe along the porous block boundary, see \cref{fig:test_2_4} (top-left, top-right, and bottom-tight). This generates there a pressure gradient downward oriented, which explains the local minimum in the pressure field close to the left upper corner, see \cref{fig:test_2_4} (bottom-left and top-right). Due to the high vertical permeability and the local pressure minimum, part of the channelized flow enters the porous medium from the top, downward oriented, close to the left upper corner, with velocity magnitude 2-3 magnitude orders smaller than in the region along to the left boundary, see \cref{fig:test_2_4} (bottom-right). The opposite occurs close to the right boundary of the filter. A small amount of the flow exits at the top near the upper right corner, upward oriented, while most of the flow exits on the right side, downward oriented, concentrated in a thin stripe along the right boundary. The flow generates a strong pressure gradient upward oriented, which explains the local maximum in the pressure field close to the right upper corner. In the interior of the filter the flow is mostly horizontal, but its magnitude is much smaller compared to the flow along the vertical boundaries. The sharp change of the velocity field in the horizontal direction results in a pressure jump across the vertical boundaries, which is clearly evident in \cref{fig:test_2_4} (top-right). This behavior is consistent with the balance of force interface condition in the two-domain Stokes--Darcy model. The pressure contour lines are vertically oriented and aligned with the anisotropy direction. 
We obtained very similar results using smaller time step sizes \(\Delta t\), not shown here for brevity. This example illustrates the ability of the method to produce a physically accurate solution in a challenging setting with strong permeability anisotropy and complex flow behavior.

\begin{figure}
	\centering
	\includegraphics[width=.55\textwidth]{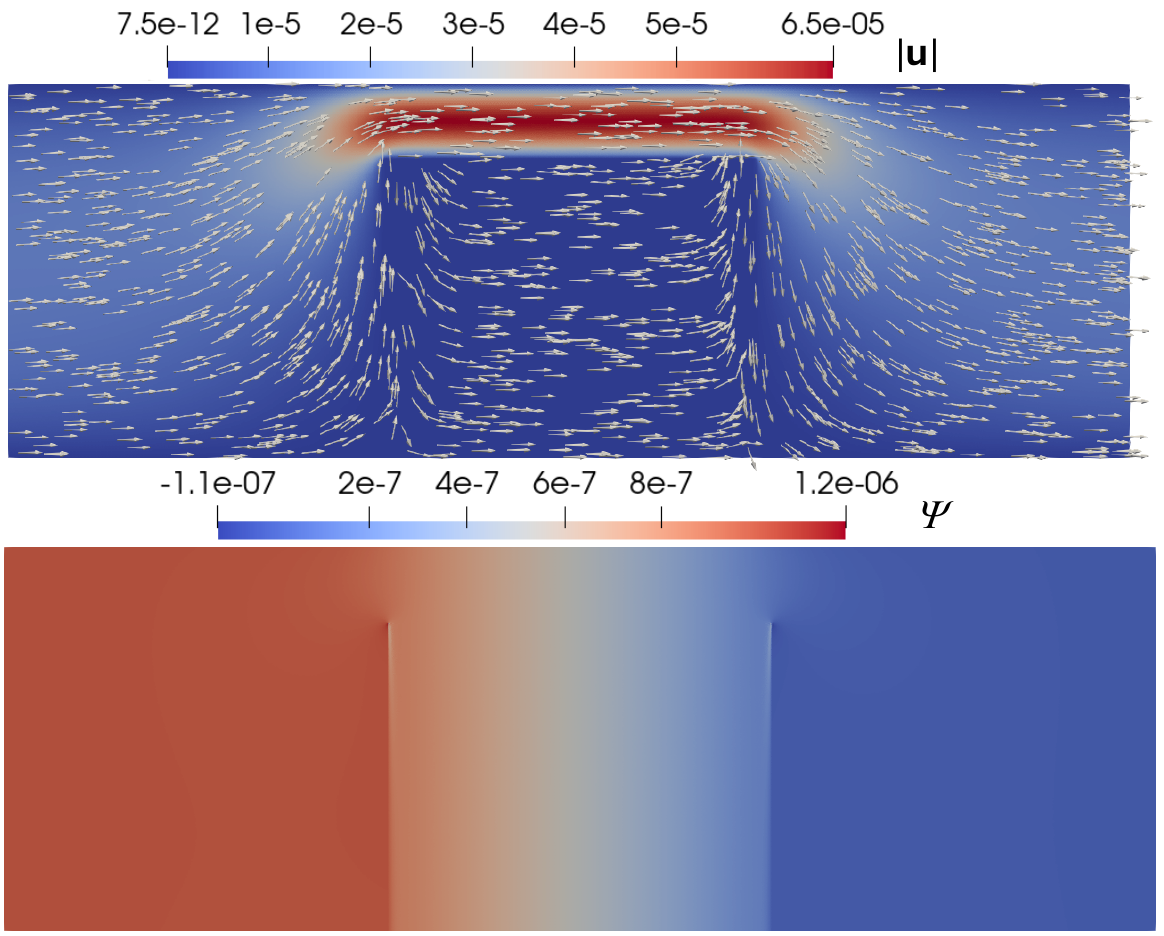}
	\includegraphics[width=.36\textwidth]{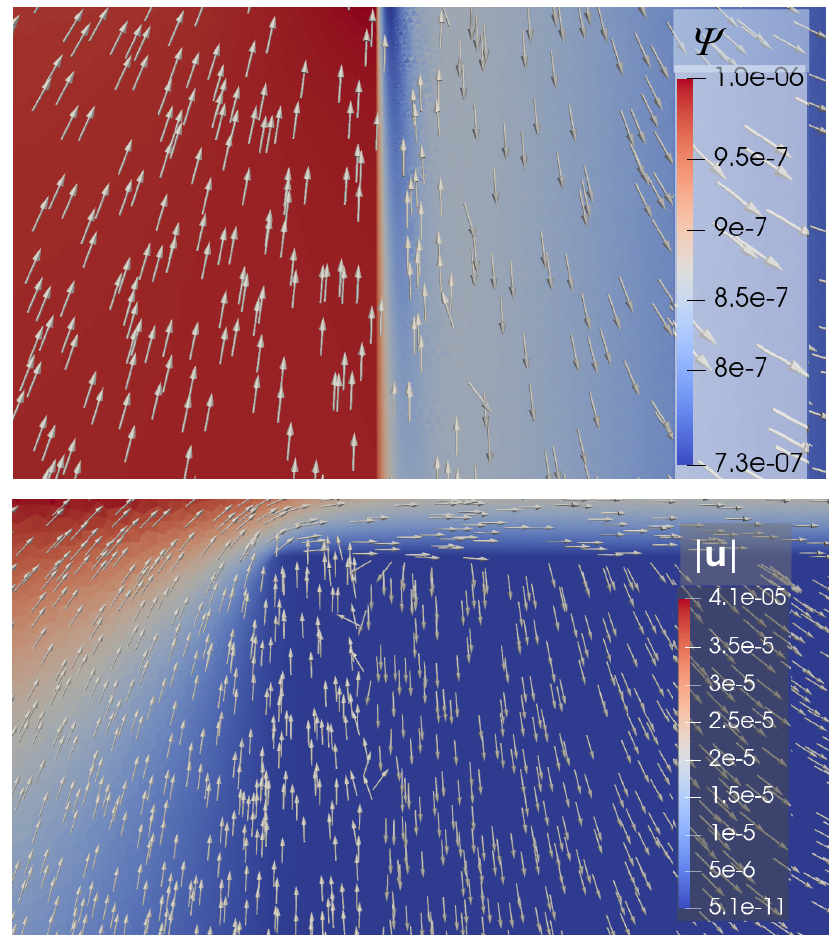}
	\caption{Test 2, homogeneous porous medium, \(k=1\e{-06}\), \(\alpha= 0 \), \(\beta = 1000\). Computed velocity (top-left) and pressure (bottom-left) with method 1. Zooms of the velocity vectors and pressure field close to the left side of the filter (top-right) and close to the upper left corner of the filter (bottom-right).}
	\label{fig:test_2_4}
\end{figure}

In \cref{fig:test_2_5} we compare the velocity fields in the case of the strong vertical anisotropy, (\(k=1\e{-06}\) \msq, \(\alpha=0\), \(\beta = 1000\)),
computed by method 1 (top, left) and method 2 (bottom, left), using \(\Delta t\) = 2 s in method 2. While the general flow behavior is similar with both methods, we observe non-physical circular patterns in the velocity field computed with method 2. The zoom in of the pressure gradient and velocity vectors in \cref{fig:test_2_5} (right) show that, while the pressure gradients in method 2 are aligned with the permeability anisotropy, the velocity vectors are not. We note that the solution obtained with method 2 using smaller time step of size \(\Delta t = 0.01\) s resembles the one obtained with method 1 using \(\Delta t = 4\) s. 

\begin{figure}
	\centering
	\includegraphics[width=0.55\textwidth]{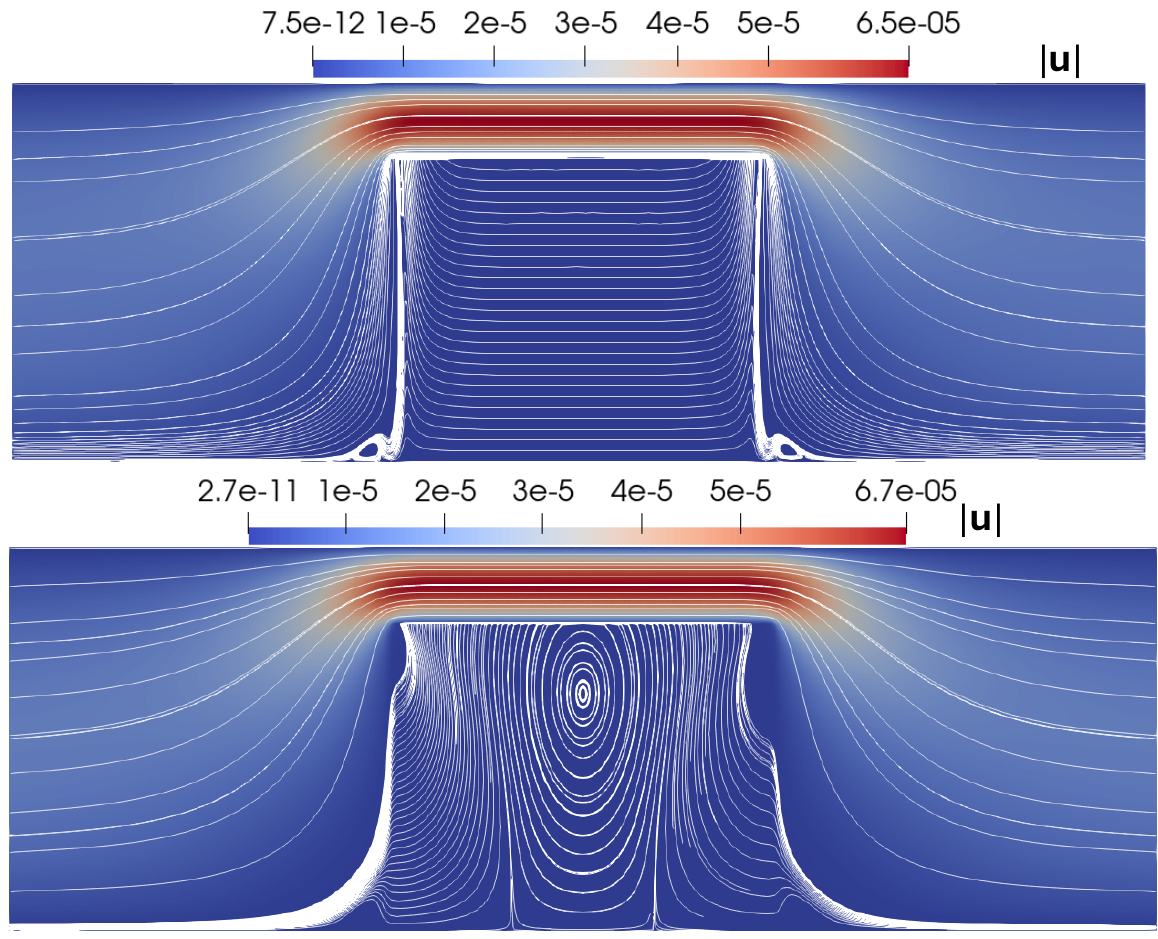} \ \includegraphics[width=0.368\textwidth]{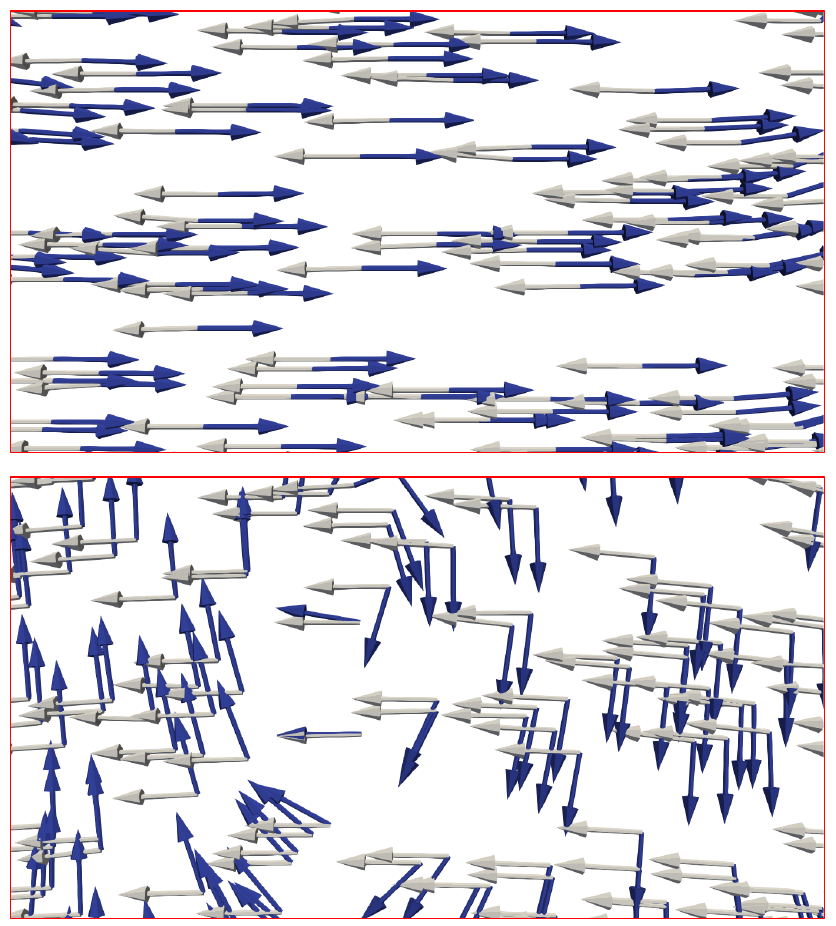}
	\caption{Test 2, homogeneous porous medium, \(k=1\e{-06}\), \(\alpha= 0 \), \(\beta = 1000\). Left: computed velocity field with method 1 (top) and method 2 (bottom). Right: zoom in of the pressure gradient (gray) and velocity (blue) vectors computed with method 1 (top) and method 2 (bottom). }
	\label{fig:test_2_5}
\end{figure}

\subsubsection{Heterogeneous porous medium} \label{hetero}
We investigate the capabilities of the proposed scheme to simulate flows in porous medium with strong contrast of the porosity and permeability tensor coefficients. We consider a porous filter composed by an outer isotropic region (\(\phi=0.7\), \(\alpha = 0\), \(\beta = 1\), \(k=1\e{-06}\)), with an inner anisotropic area \(\Omega_{in}= \left(0.32, 0.385\right) \times \left(0.5, 0.12\right)\), see \cref{fig_test2_sett} (right). In \(\Omega_{in}\) we set
\(\phi=0.4\), \(\alpha = \pi/4\), \(\beta = 100\) and \(k=1\e{-08}\) \msq,
resulting in a jump in the values of the porosity and permeability tensor coefficients across its boundary. In \cref{fig:test_2_6} we show the velocity and pressure fields computed by method 1. The solver properly simulates the physical process: the velocity streamlines and the pressure contour lines deviate where the flow encounters the inner anisotropic region. We also compare in \cref{fig:test_2_7} the velocity fields computed with method 1 (top) and method 2 (bottom). While method 2 produces a similar velocity field in the isotropic porous region, in the anisotropic region we observe some non-physical streamlines deviations and vorticities.     

\begin{figure}
	\centering
	\includegraphics[width=0.47\textwidth]{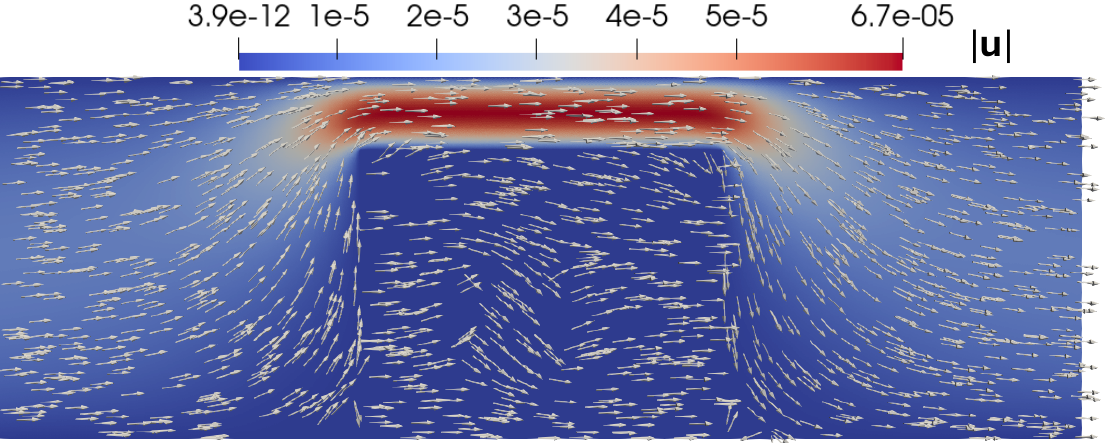}
	\includegraphics[width=0.47\textwidth]{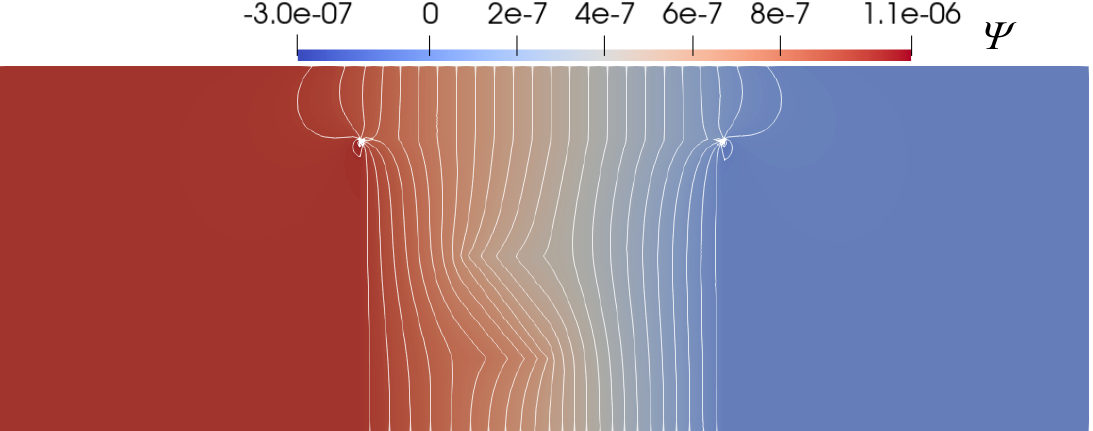}
	\caption{Test 2, heterogeneous porous medium. Computed velocity (left) and pressure (right) with method 1.}
	\label{fig:test_2_6}
\end{figure}

\begin{figure}
	\centering
	\includegraphics[width=0.55\textwidth]{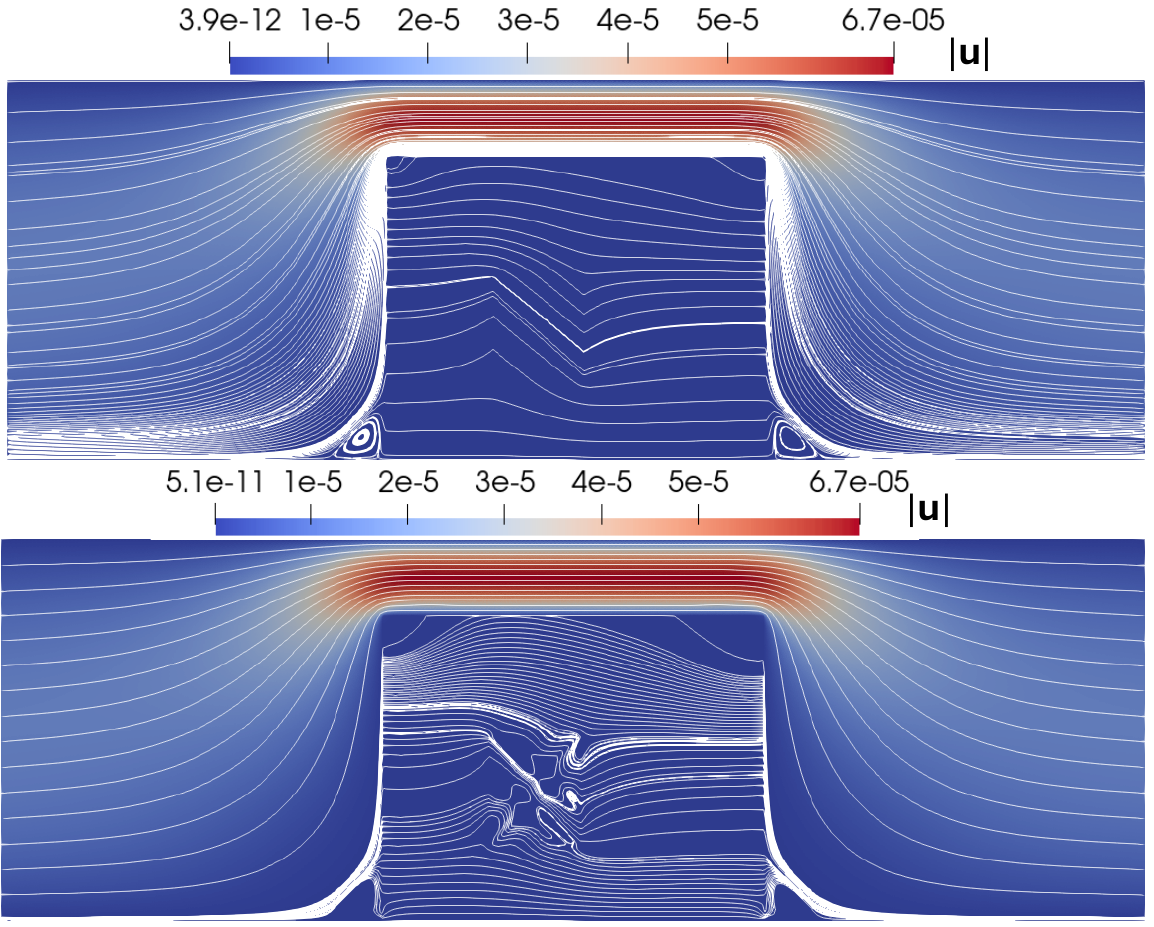} \ \includegraphics[width=0.368\textwidth]{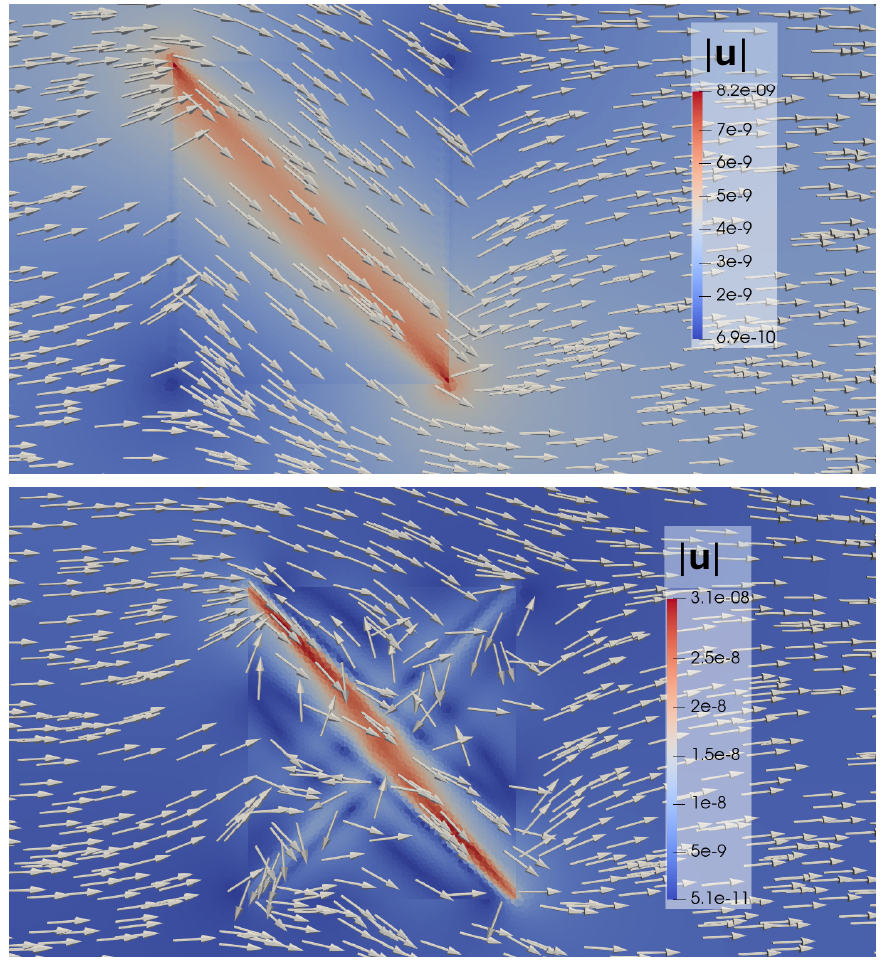}
	\caption{Test 2, heterogeneous  porous medium. Velocity field computed with method 1 (top) and method 2 (bottom).}
	\label{fig:test_2_7}
\end{figure}

\subsection{Test 3: interaction of a free fluid with a porous region with a steps-like interface}  \label{test3}

The setting of this test, presented in \cite{CHEN2014650}, is shown in \cref{fig_test3_sett} (left). A polygonal interface with three steps divides the (dimensionless) computational domain \( \Omega = (0, 2) \times (0, 2)\) in a fluid and a porous region. The Dirichlet boundary $\Gamma_d$ consists of the top and right sides of \(\Omega\), with no-slip BCs on the top side and (dimensionless) Poiseuille velocity profile \(\mathbf{u}=\left(\frac{y}{2}\left(y-2\right),0\right)^T\) on the right side. On the rest of the boundary we set \(\boldsymbol{\Sigma}_b=0\).
The ICs are zero velocity and pressure. The (dimensionless) kinematic viscosity is \(\nu=1\). The medium is assumed to be isotropic with porosity and inverse permeability defined in \eqref{profiles}. The porosity parameters are $\phi_{max} = 1$ and \(\phi_{min}=0.4\). Two (dimensionless) values of the inverse of the permeability coefficient are considered, \(\mathfrak{K}^0_{i,i} = 100 \) and \(\mathfrak{K}^0_{i,i} = 10000 \), \(i=1,2\). According to \cite{CHEN2014650}, we set \(\theta_{\psi}=\theta_{\mathfrak{K}_{i,j}}=100\) in the transition zone between the fluid and porous region. The Reynolds number in the fluid region, based on the maximum incoming velocity from the right side and the domain length side, is \(Re = 1\). In the porous region the maximum value of the Reynolds number is \(Re \simeq 0.24\) if \(\mathfrak{K}_{i,i} = 100 \) and \(Re \simeq 0.0132\) if \(\mathfrak{K}_{i,i} = 10000 \). We adopt a relatively coarse grid with (dimensionless) size ranging from 0.065 in the bulk fluid and porous regions to 0.02 in the transition region. The grid has 6421 triangles and 3281 vertices, see \cref{fig_test3_sett} (right), similar to the one used in \cite{CHEN2014650}. \par

\begin{figure}
	\centering
	\includegraphics[width=0.65\textwidth]{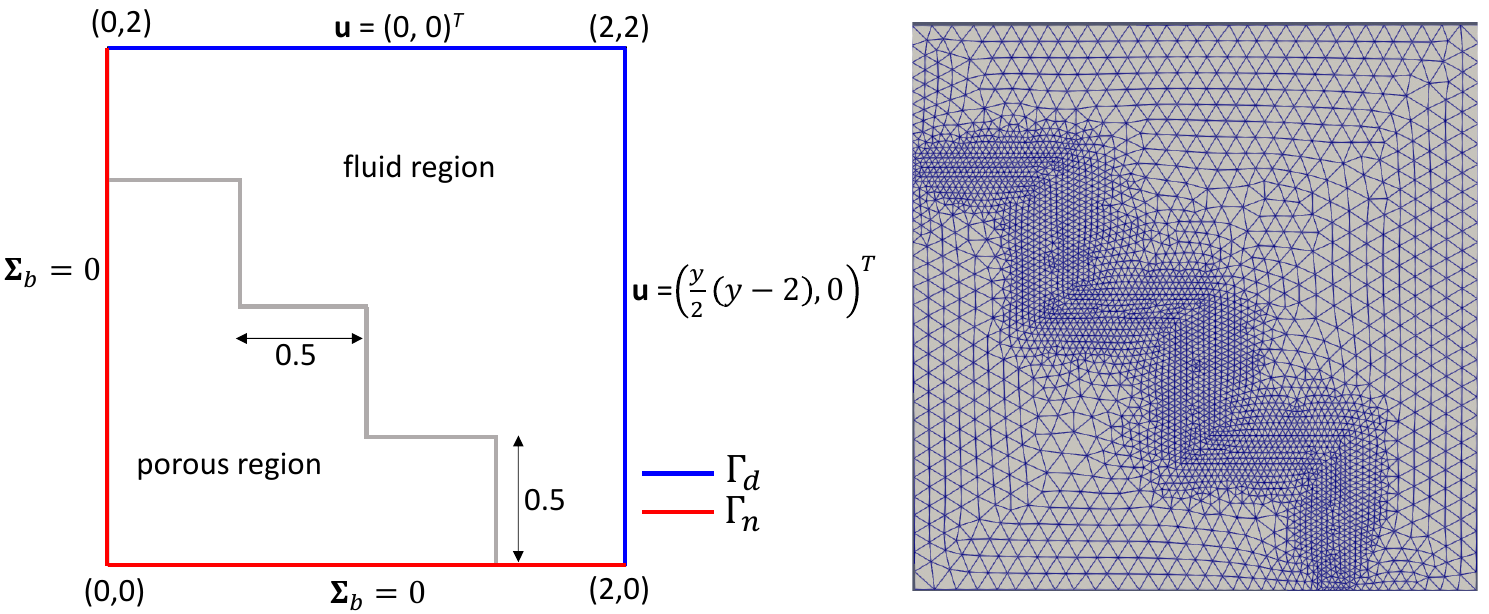}
	\caption{Test 3, settings of the numerical experiments (left) and computational grid (right).}
	\label{fig_test3_sett}
\end{figure}

In \cref{test3_velo_comps_k100,test3_velo_comps_k10000} we plot the computed (dimensionless) velocity components and the (dimensionless) kinematic pressure fields with the velocity vectors. The flow enters the domain from the right vertical side. Due to the polygonal porous medium, the flow is partly channelized in the upper fluid region and partly leaves from the bottom side near the right corner, oriented downward. The velocity components are in very good agreement with the results in \cite{CHEN2014650} (see Figures 7 and 9 of the referred paper), where the authors solve the Brinkman equations in the framework of a ODA using Taylor–Hood finite elements. In the case of smaller permeability of the porous medium ($\mathfrak{K}_{i,i} = 10000$), the flux leaving the domain downward oriented close to the right bottom step increases, see \cref{test3_velo_comps_k10000} (center), which explains the local increment of the pressure seen in \cref{test3_velo_comps_k10000} (right). We also observe higher deviations of the velocity vectors close to the interface with some vorticities. In \cref{test3_u_x_k10000} we compare the profiles of \(u_x\) at \(x=0.5\) in the case of \(\mathfrak{K}_{i,i} = 10000 \) computed by the present solver over the coarse grid (as detailed before), a much finer one (size ranging from 0.03 to 0.0075, with 30126 triangles and 15227 vertices) and the output obtained in \cite{CHEN2014650} (see Figure 11 in the referred paper). No evidence of any significant grid size effect is detected in the outputs of the present solver. The \(u_x\) profile provided in \cite{CHEN2014650} is slightly shifted down, with a small underestimation of the peak value.      

\begin{figure}
	\centering
	\includegraphics[width=0.325\textwidth]{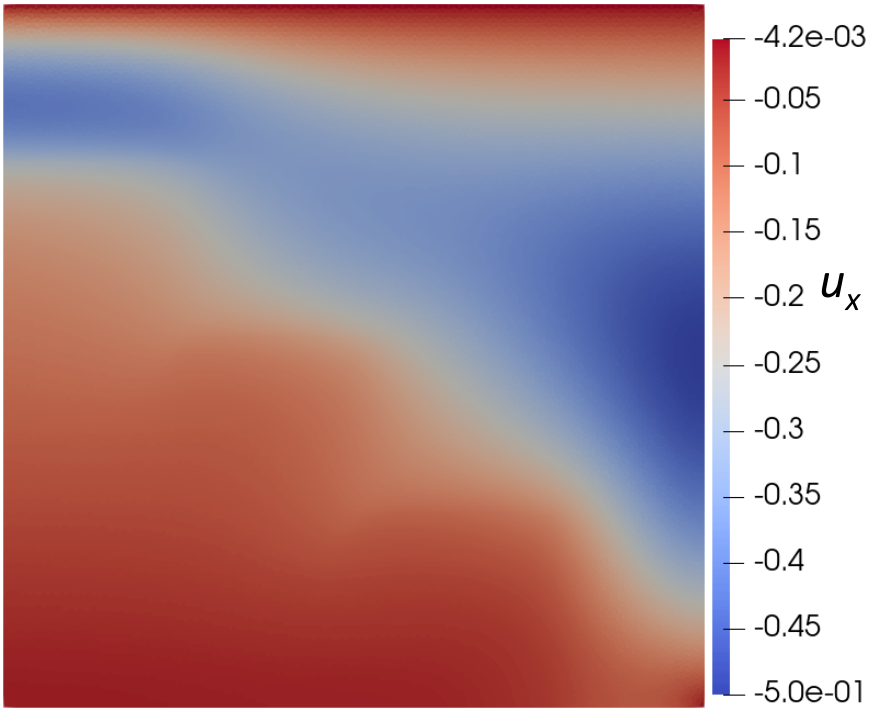}
	\includegraphics[width=0.325\textwidth]{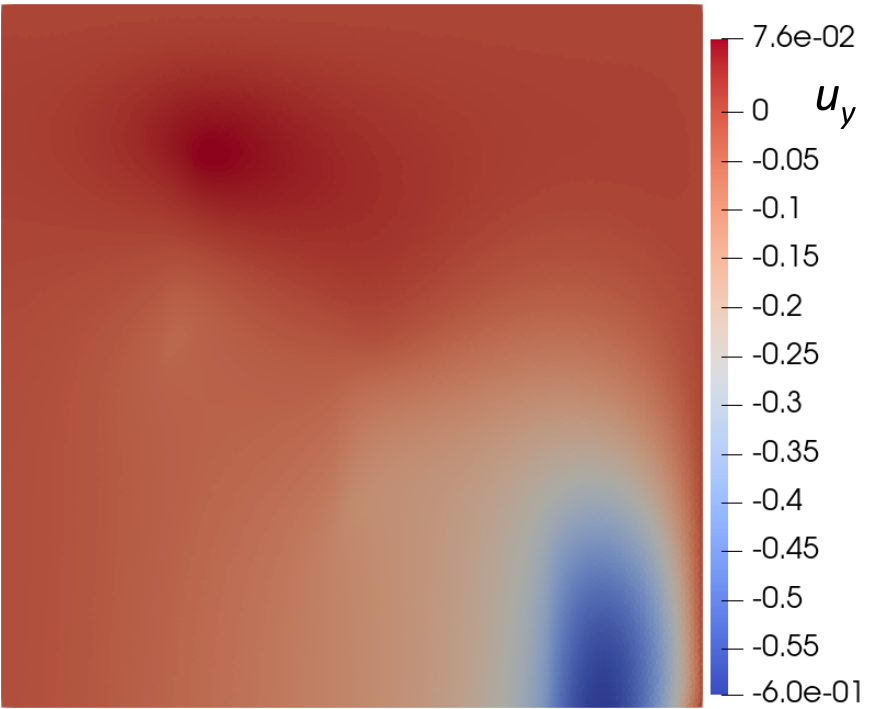}
	\includegraphics[width=0.325\textwidth]{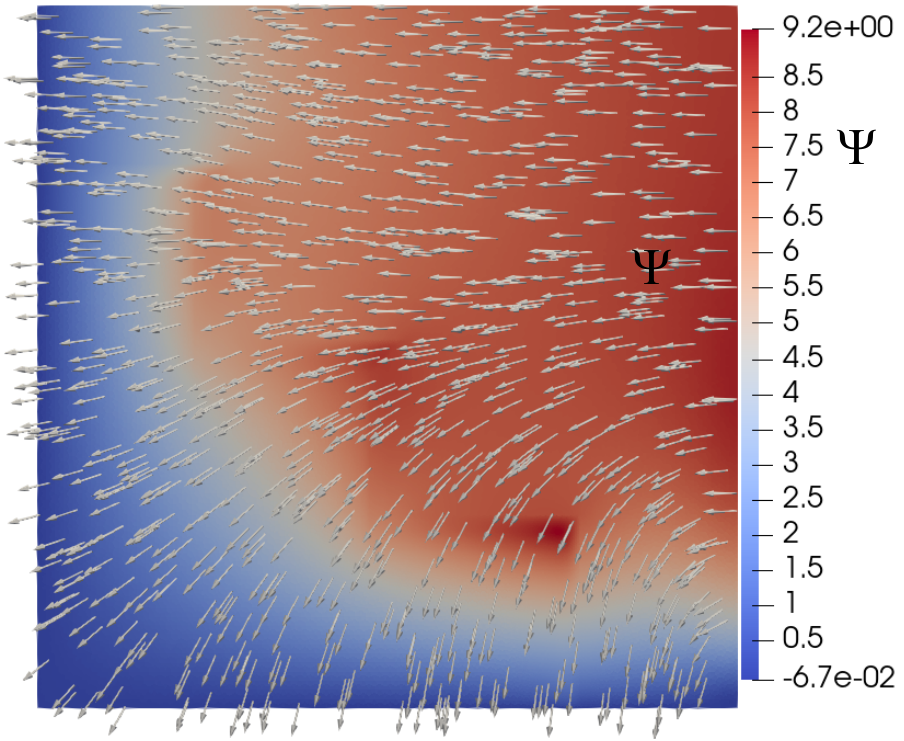}
	\caption{Test 3, computed \(u_x\) (left), \(u_y\) (middle), and \(\Psi\) and velocity vectors (right) with \(\mathfrak{K}_{i,i} = 100\).}
	\label{test3_velo_comps_k100}
\end{figure}

\begin{figure}
	\centering
	\includegraphics[width=0.325\textwidth]{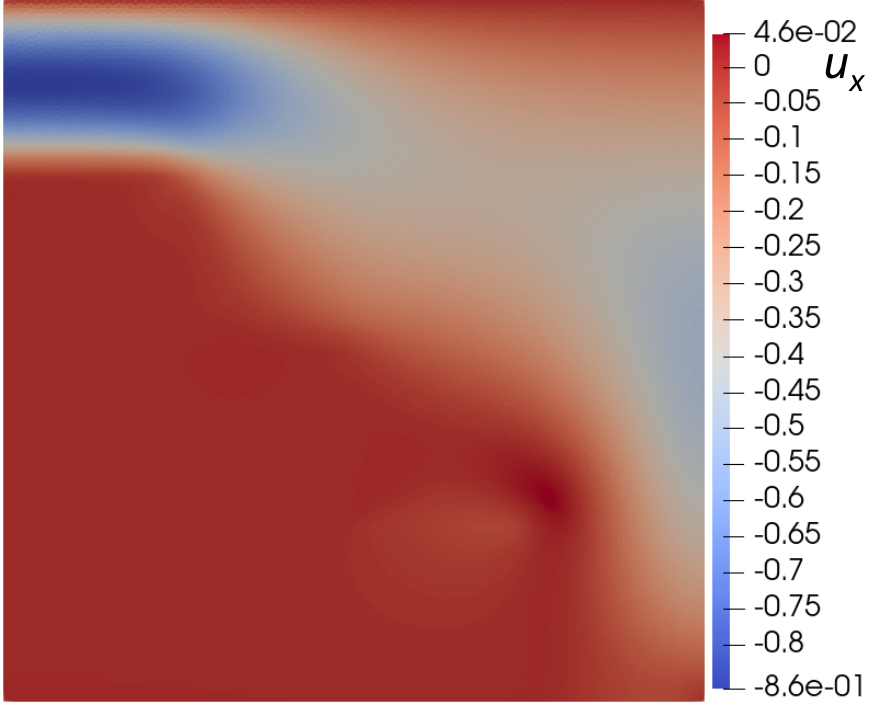}
	\includegraphics[width=0.325\textwidth]{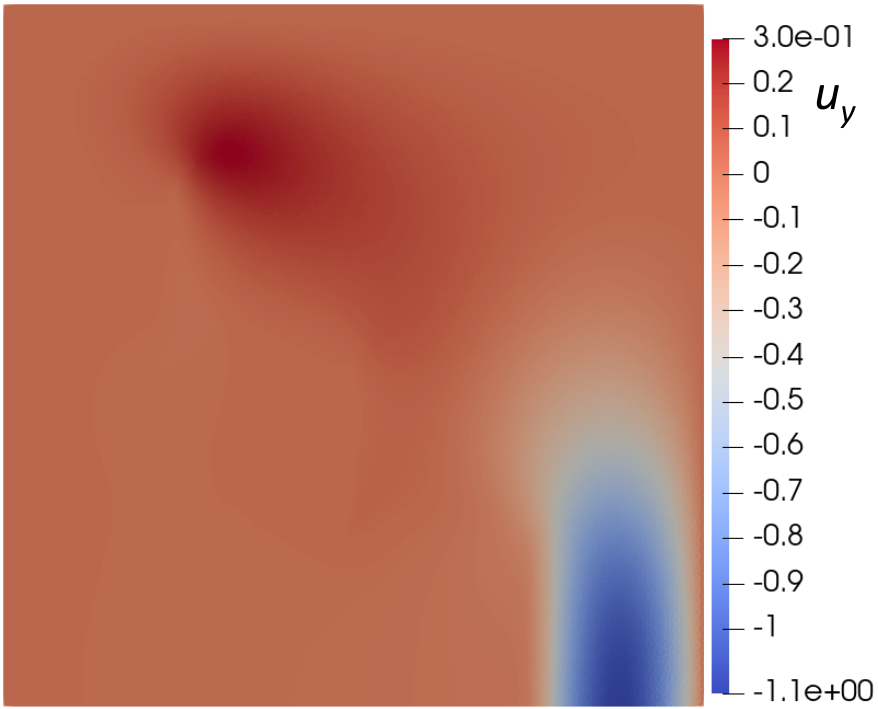}
	\includegraphics[width=0.325\textwidth]{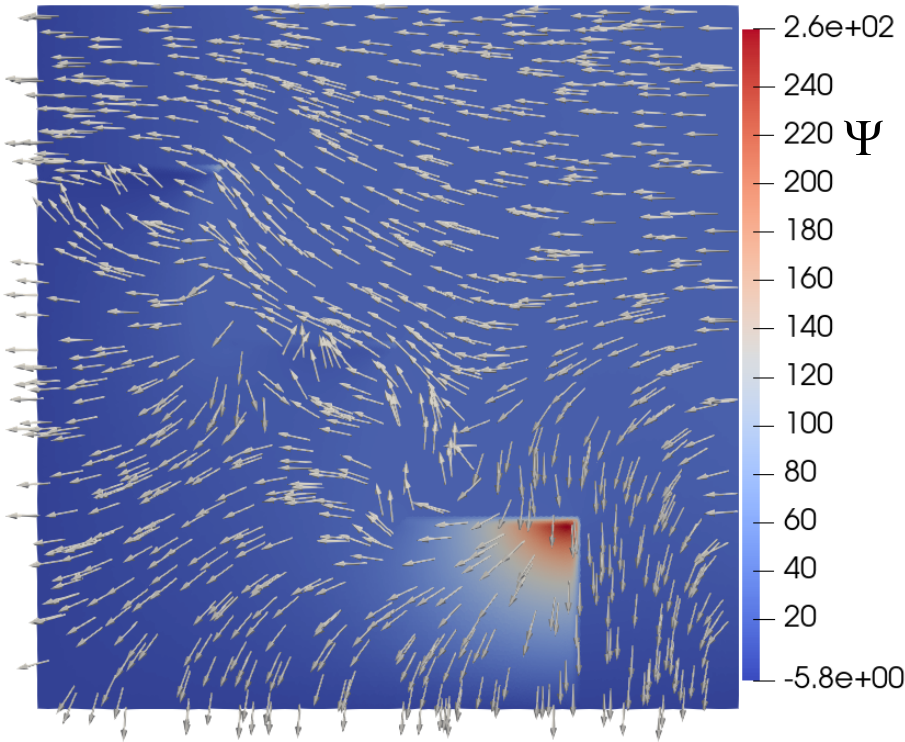}
	\caption{Test 3, computed \(u_x\) (left), \(u_y\) (middle), and \(\Psi\) and velocity vectors (right) with \(\mathfrak{K}_{i,i} = 10000\).}
	\label{test3_velo_comps_k10000}
\end{figure}

\begin{figure}
	\centering
	\includegraphics[width=0.43\textwidth]{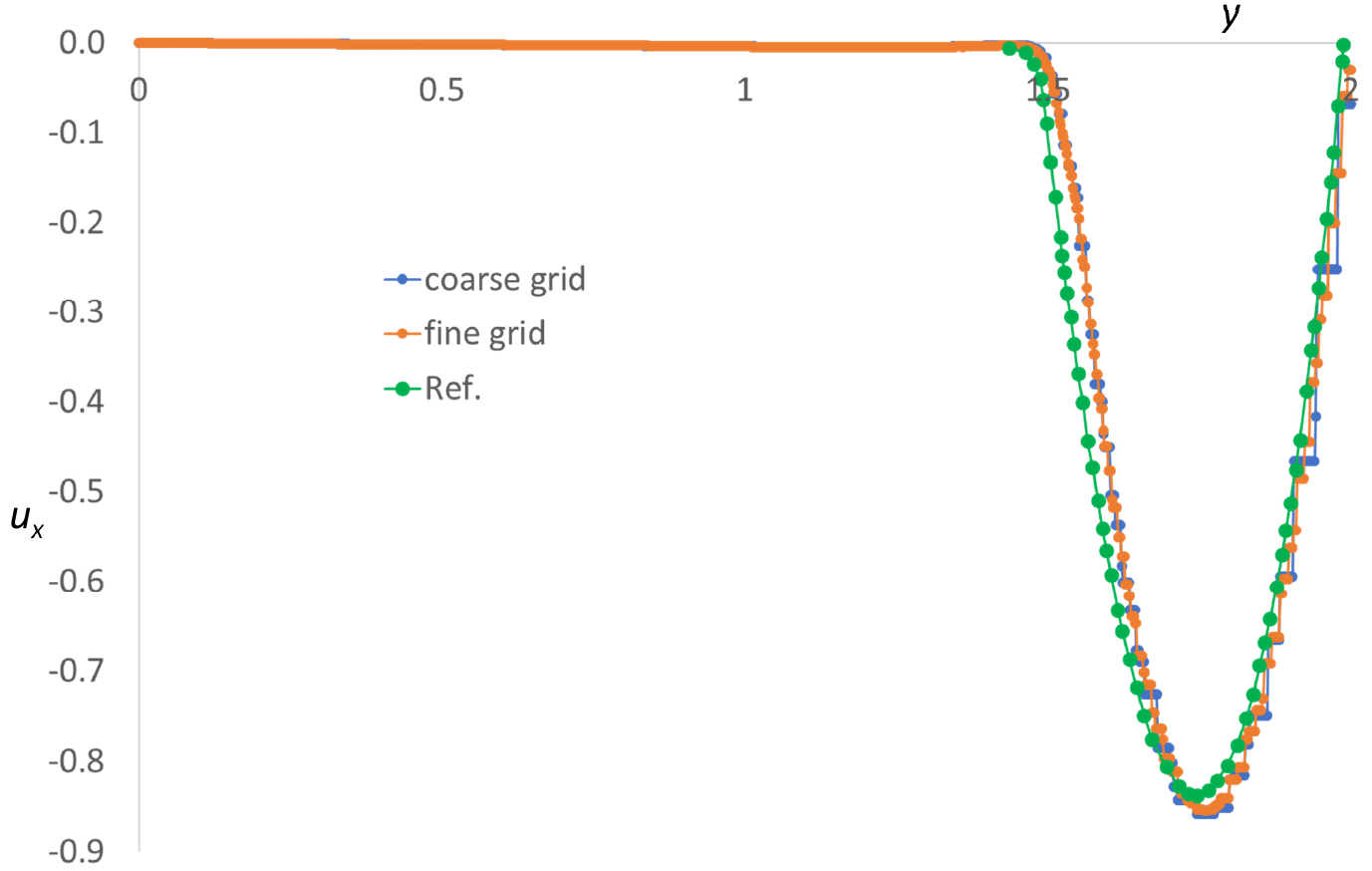}
	\caption{Test 3, profiles of \(u_x\), \(\mathfrak{K}_{i,i} = 10000\). Nomenclature. ``coarse grid'' and ``fine grid'': results of the present numerical method; ``Ref.'': results provided in \cite{CHEN2014650}.}
	\label{test3_u_x_k10000}
\end{figure}

\subsection{Test 4: Study of an intracranial aneurysm (ICA) with an aneurysmatic sac filled by a porous medium} \label{Test 4}
The last test is a ``show-case'' application, where a porous medium fills the sac of an intracranial aneurysm (ICA). ICAs are dilatation of the arterial walls, generally formed in the circle of Willis, which may easily rupture, with consequent serious brain damages, hemorrhages, and high risk of mortality. During the last decade, a technology based on the use of porous media like shape memory polymer foams has been proposed as a novel endovascular ICA treatment, see, \eg \cite{Cabaniss2025} and references therein. The shape memory property allows the polymers to be compressed into a catheter embedded into the ICA's cavity, so that they expand via recovery activation, occupying the initial programmed ICA geometry, occluding the aneurysm cavity, and recovering the original cerebral parent vessel shape. This leads to a strong reduction of the blood circulation in the aneurysm cavity and stress on the internal walls of the sac. Clinical advantages have emerged over surgical treatments and the more traditional coiling techniques, the latter often involving a partial filling of the sac by the porous foam \cite{Cabaniss2025}. \par

We consider a real patient-specific geometry case-id ``C0075'' (patient-id ``P0235'') of the open-source web repository ``Aneurisk'' \cite{AneuriskWeb}, see the computational domain in \cref{test_4_sett}, left. The volume of the aneurysmatic sac is approximately 1180 \(\textnormal{mm}^3\), one of the largest values in the web repository, and the diameter of the parent vessels is in the range \(\left[1.7\e{-03}, 3.4\e{-03}\right]\) m. The computational grid has been created by the open-source VMTK toolkit \cite{Antiga2008} and Netgen \cite{Schberl1997NETGENAA}, with 42915 vertices and 176102 tetrahedra with maximum aspect ratio 6.5. This application is presented to show the capability of the method to handle unsteady flows in very irregular domains. \par

\begin{figure}
	\centering
	\includegraphics[width=0.36\linewidth]{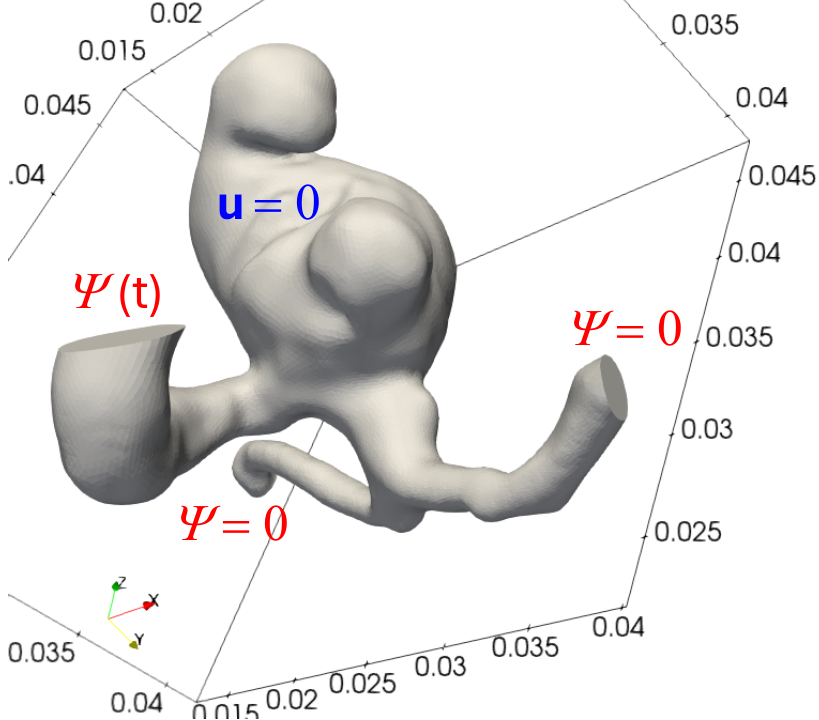}
	\includegraphics[width=0.32\linewidth]{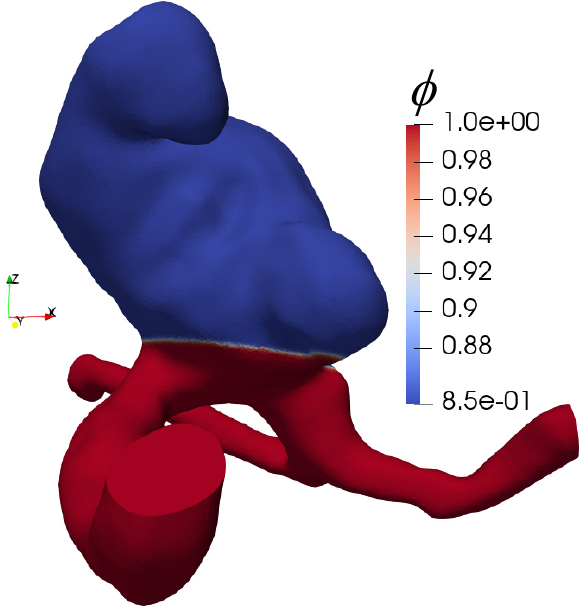}
	\includegraphics[width=0.27\linewidth]{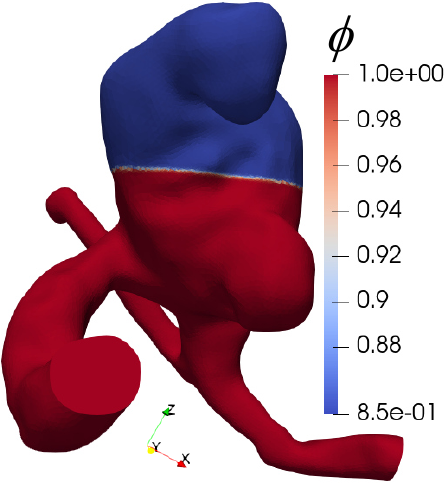}
	\caption{Test 4, computational domain and settings for the numerical runs (left), \(\phi\) values for \(F_r \simeq 1\) (center) and \(\phi\) values for \(F_r \simeq 0.61\) (right).}
	\label{test_4_sett}
\end{figure}

\begin{table} [H] 
	\caption{Test 4, parameters of the porous medium within the aneurysmatic sac.}
	\centering
	\begin{tabular}{c c c c c c c c c c}
		\(\phi_{max}\) & \(\phi_{min}\) & \(\theta_{\psi}\) & \(\theta_{\mathfrak{K}_{i,j}}\) & \(\mathfrak{K}^0_{1,1}\) & \(\mathfrak{K}^0_{2,2}\) & \(\mathfrak{K}^0_{3,3}\) & \(\mathfrak{K}^0_{1,2} = \mathfrak{K}^0_{2,1}\) & \(\mathfrak{K}^0_{1,3} = \mathfrak{K}^0_{3,1}\) & \(\mathfrak{K}^0_{2,3} = \mathfrak{K}^0_{3,2}\) \\
		\hline
		1	& 0.85 & 12000 & 12000 & 670 & 670 & 67000 & 0 & 0 & 0 \\	
		\hline		
	\end{tabular}
	\label{test4_bulk_pm_values}
\end{table}

We consider three scenarios, depending on the filling ratio of the aneurysmatic sac \(F_r\), \ie the ratio between the volume occupied by the porous medium and the total volume of the sac. The first two cases are shown in \cref{test_4_sett} (center and right) corresponding to \(F_r \simeq 1\) and \(F_r \simeq 0.61 \), respectively, while in the third scenario no porous medium fills the sac, \ie \(F_r = 0\). The flow is generated by an unsteady periodic pressure difference \(\Delta \Psi(t) = 1\e{-03}\sin(2 \pi  t)\) between the ends of the parent vessels, as shown in \cref{test_4_sett} (left). We set zero velocity (no slip BCs) on the lateral vessel walls (\cref{test_4_sett}). The ICs are zero velocity and pressure. The porous medium in the sac is homogeneous and anisotropic. The porosity $\phi$ and the inverse permeability $\mathfrak{K}$ are computed as in \eqref{perm_bulk} with parameters given in \cref{test4_bulk_pm_values}. The permeability coefficient along the \(z\) direction is two orders of magnitude smaller than along the other two directions.  \par

We compute the hemodynamic forces on the internal wall of the ICA and on the porous medium filling the sac. The total forces on the ICA wall is computed as
\begin{equation*}
	\mathbf{F}_w=\mathbf{F}_{w,\Psi}+\mathbf{F}_{w,\nu} = \sum_{i=1}^{N_{fb}} \left(\mathbf{f}_{w,\Psi,i}+\mathbf{f}_{w,\nu,i}\right),
\end{equation*} 
where the summation is over the triangles discretizing the ICA wall and \(\mathbf{f}_{w,\Psi,i}\) and \(\mathbf{f}_{w,\nu,i}\) are the pressure and viscous forces, respectively, acting on the \(i\)-th boundary triangle, \(1 \le i \le N_{fb}\), computed as
\begin{equation*}
	\mathbf{f}_{w,\Psi,i} = \phi_i \Psi_i \mathbf{n}_i |e_i|, \qquad f^{x\left(y,z\right)}_{w,\nu,i} = \nu \nabla u_{x\left(y,z\right)}^{n+1} \cdot \mathbf{n}_i |e_i|.
\end{equation*}
Here \(|e_i|\) is the surface of interface \(i\), \(\phi_i\) is the porosity computed in the center of mass of the interface,  \(\Psi_i\) is the intrinsic kinematic pressure computed in the center of mass of \(i\) according to the \(P_1\) pressure distribution within the tetrahedron sharing the boundary interface, and \(\mathbf{n}_i\) is the outward unit normal vector to the interface. \(f^{x\left(y,z\right)}_{w,\nu,i}\) are the components of the viscous force \(\mathbf{f}_{w,\nu,i}\), where the velocity components gradients are obtained according to the linear variation of the velocity components within the tetrahedron sharing the boundary face. \par

The force on the porous medium within the aneurysmatic sac is computed as 
\begin{equation*}
	\mathbf{F}_{pm} = \sum_{i=1}^{N_{T,pm}} \mathbf{f}_{pm,i} = \sum_{i=1}^{N_{T,pm}} \sum_{j=1}^{d+1} \phi_j \mathbf{L}_j \mathbf{u}_j w_j |E_i|,
\end{equation*}
where the summation on $i$ is over the tetrahedra used to discretize the porous domain and \(\mathbf{f}_{pm,i}\) is the force on the \(i\)-th element, computed by the Gaussian numerical integration. The summation on $j$ is over the \(d+1\) Gaussian integration points within the \(i\)-th simplex with volume \(|E_i|\), \(\mathbf{L}_j\)  and \(\mathbf{u}_j\) are the quantities computed at the $j$-th integration point, and \(w_j\) is the quadrature weight. \par

In \cref{test_4_1} we plot, for the case of \(F_r \simeq 1\), the pressure and velocity fields at the simulation times \(T_0 / 4 \) and \(3 T_0 / 4\) (where \(T_0\) is the oscillation period), corresponding to the maximum positive and negative pressure differences \(\Delta \Psi \). The maximum Reynolds number in the vessel, computed according to the maximum velocity vector magnitude \(|\mathbf{u}|\) and the diameter of the parent vessel, is \(Re \simeq 0.3\). The maximum Reynolds number in the porous region is 2-3 orders of magnitude smaller. The velocity vectors are consistent with the pressure difference assigned at the boundary.
\cref{test_4_2} shows the computed pressure and viscous forces acting on the ICA wall (left and middle columns) and the forces acting on the porous medium within the sac (right column). The magnitudes of \(\mathbf{F}_{w,\nu}\) and \(\mathbf{F}_{pm}\) are approximately 2 and 4 orders of magnitude smaller than the pressure forces \(\mathbf{F}_{w,\Psi}\), respectively. As expected, \(\mathbf{F}_{pm}\) is oriented according to the velocity within the porous medium, depending on the boundary assigned pressure differences. We further show in \cref{test_4_3} the differences in \(\mathbf{F}_{w,\Psi}\) computed in the case \(F_r \simeq 1\) with the cases \(F_r \simeq 0.61\) and \(F_r = 0\). As expected, the magnitude of \(\mathbf{F}_{w,\Psi}\) is smaller where the porous medium is in contact with the aneurysmatic wall. \par

\begin{figure}
	\centering
	\includegraphics[width=0.7\linewidth]{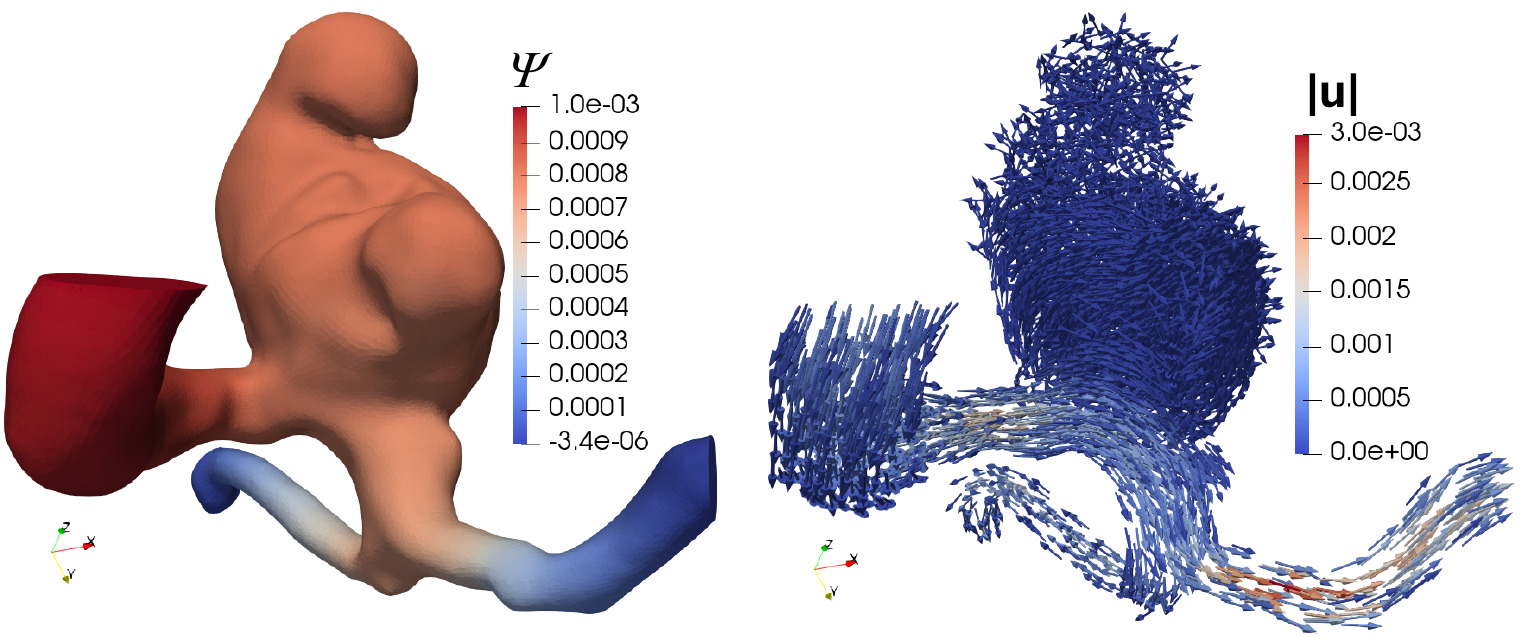}
	\includegraphics[width=0.7\linewidth]{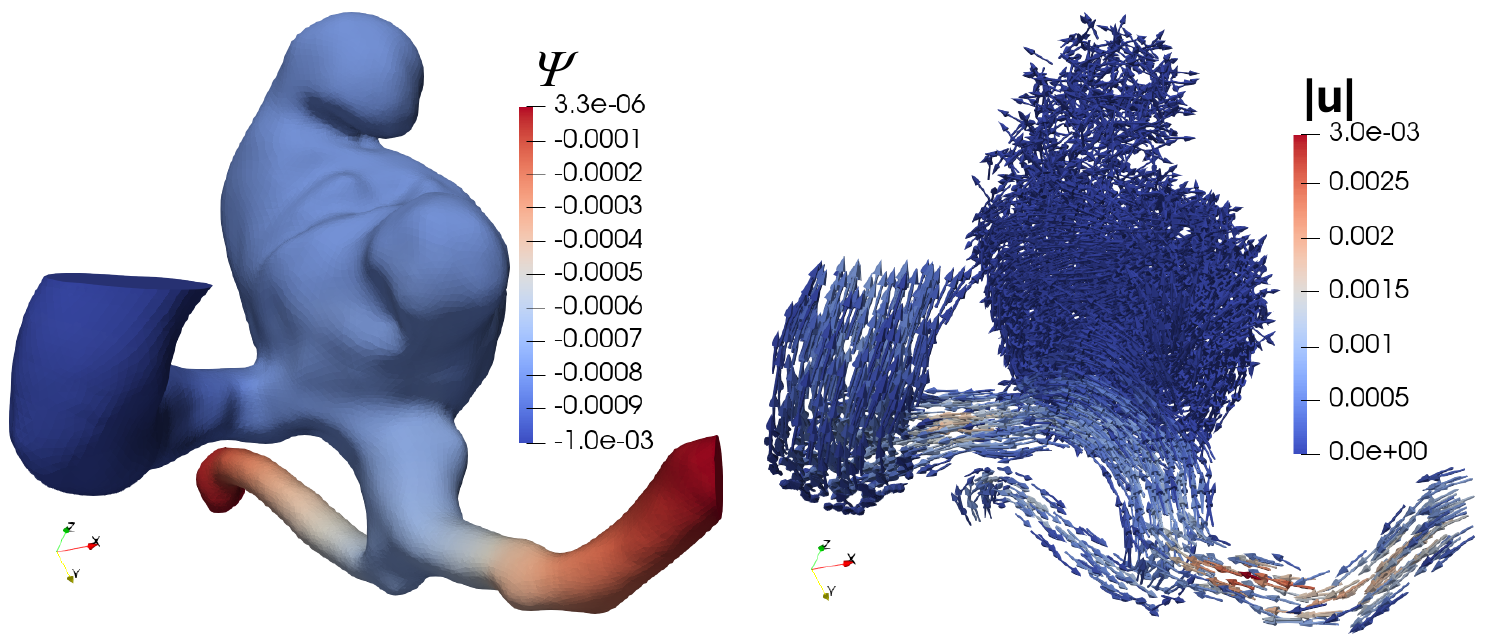}
	\caption{Test 4, computed results with \(F_r \simeq 1\). Left: pressure (\(\textnormal{m}^2 / \textnormal{s}^2\)), right: velocity (\(\textnormal{m} / \textnormal{s}\)). Top row: \(t=T_0 / 4\), bottom row: \(t=3  T_0 / 4\).}.
	\label{test_4_1}
\end{figure}

\begin{figure}
	\centering
	\includegraphics[width=0.8\textwidth]{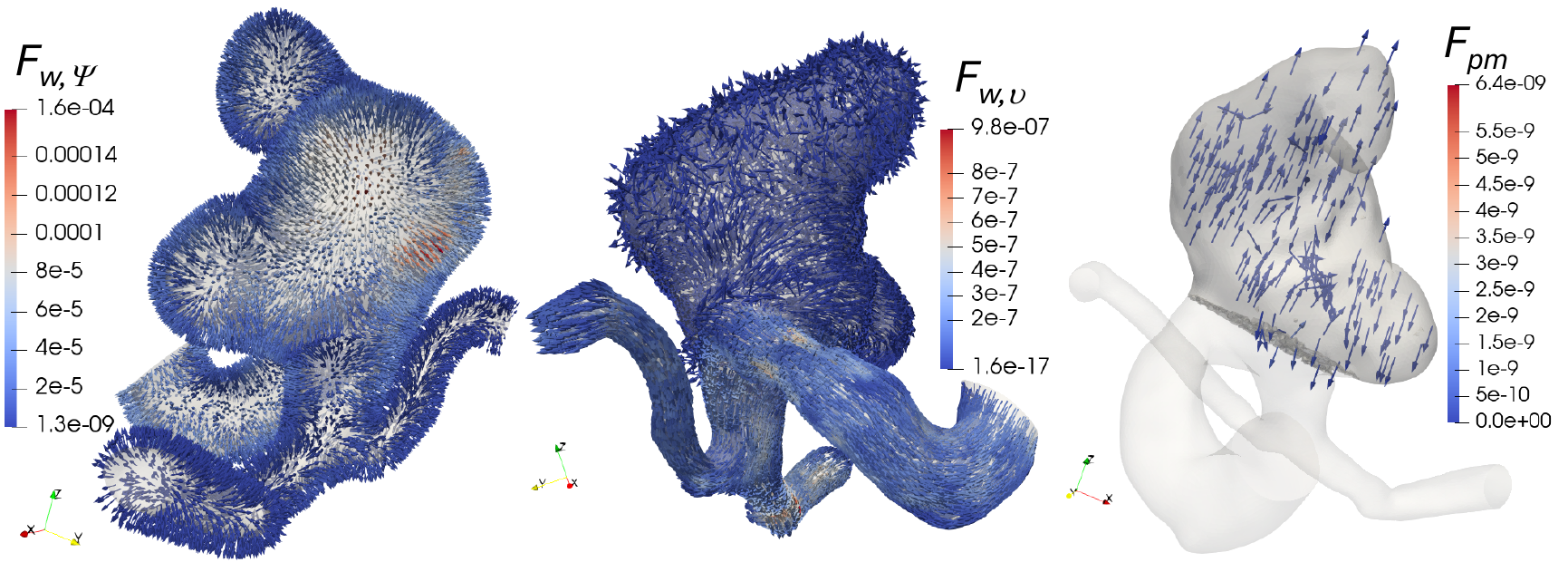}
	\includegraphics[width=0.8\textwidth]{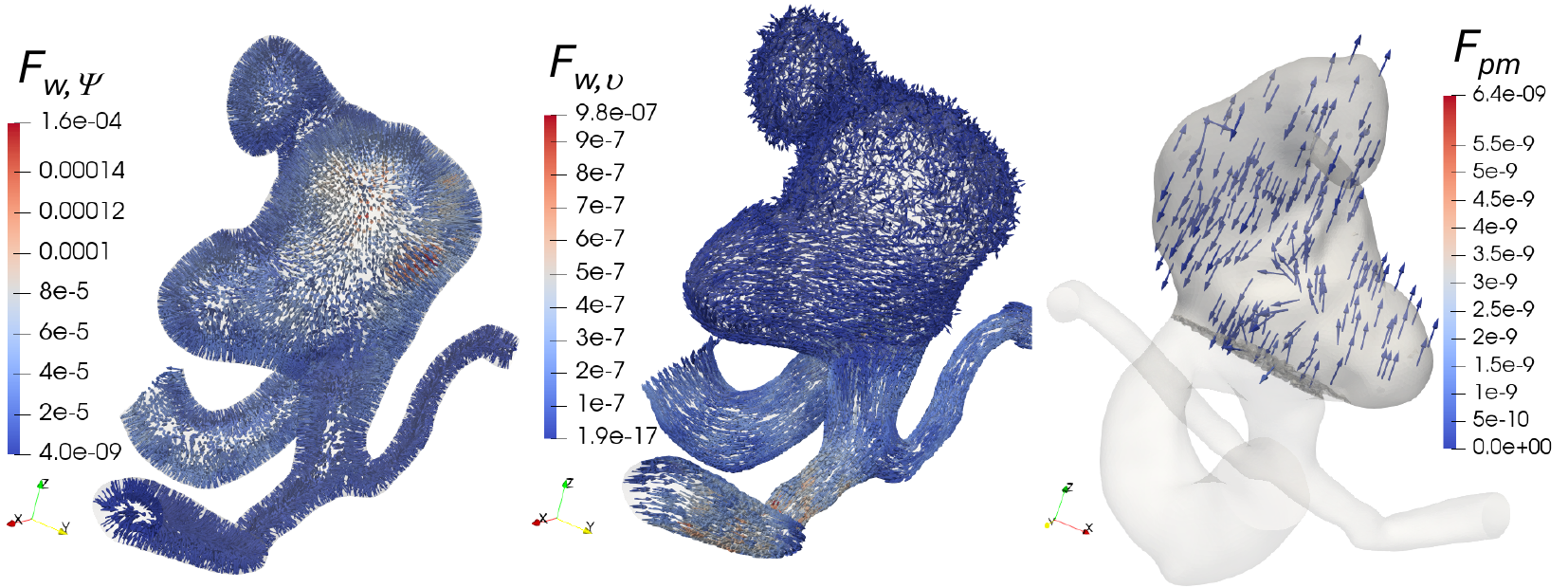}
	\caption{Test 4, computed results with \(F_r \simeq 1\). Left: \(\mathbf{F}_{w,\Psi}\), center: \(\mathbf{F}_{w,\nu}\), right:  \(\mathbf{F}_{pm}\). Top row: \(t=T_0 / 4\), bottom row: \(t=3  T_0 / 4\) (results in \(\textnormal{m}^4 / \textnormal{s}^2\)).}
	\label{test_4_2}
\end{figure}

\begin{figure}
	\centering
	\includegraphics[width=0.6\linewidth]{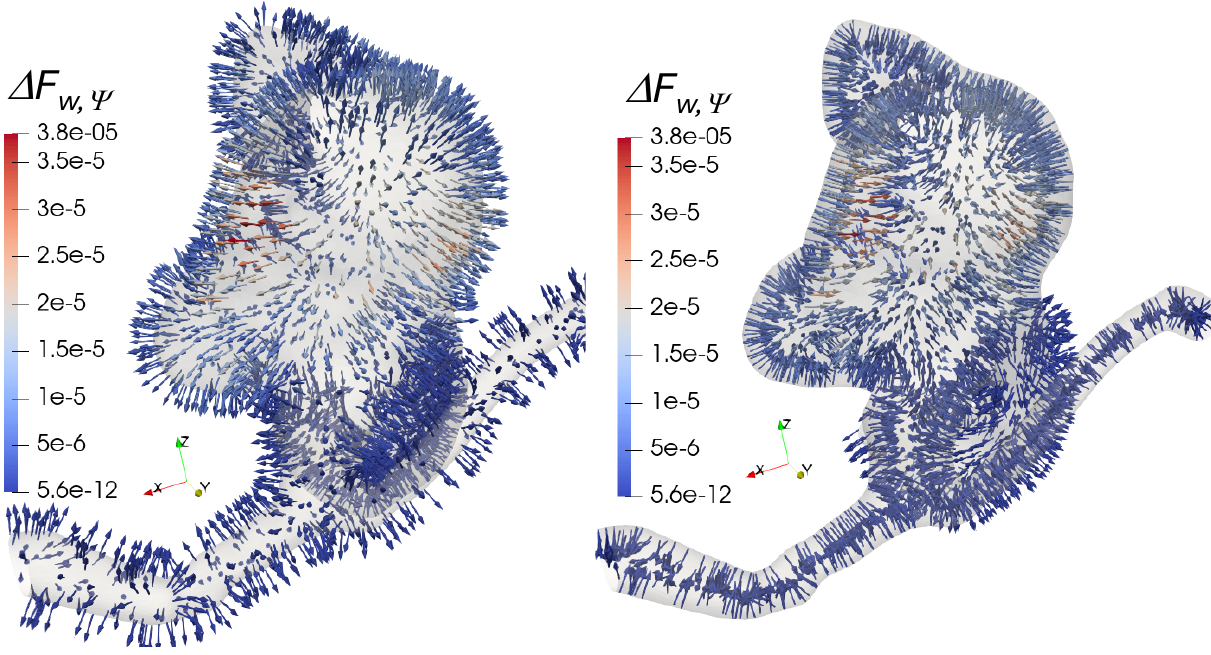}
	\includegraphics[width=0.6\linewidth]{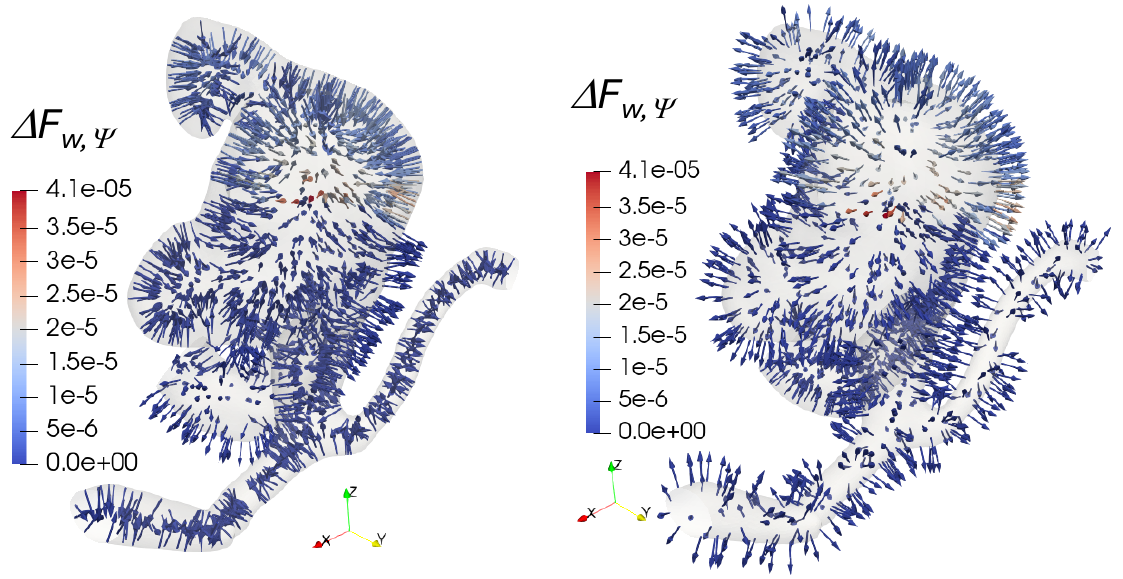}
	\caption{Test 4, computed difference in \(\mathbf{F}_{w,\Psi}\) between the cases \(F_r \simeq 1\) and \(F_r \simeq 0.61\) (left) and between the cases \(F_r \simeq 1\) and \(F_r =0\) (right). Top row: \(t=T_0 / 4\), bottom row: \(t=3  T_0 / 4\) (results in \(\textnormal{m}^4 / \textnormal{s}^2\)).}
	\label{test_4_3}
\end{figure}

\section{Conclusions} \label{conclusion}
We have presented \(H(\text{div})\)-conforming projection-based mixed finite element methods for the solution of the unsteady Brinkman problem. At each time step the methods require the solution of a mixed stress-velocity predictor problem for the viscous effects and a mixed velocity-pressure projection problem to account for the fluid incompressibility. The methods produce a pointwise divergence-free velocity and are robust in both the Stokes and the Darcy regimes. Unconditional stability of the fully discretized algorithm and first order in time accuracy have been established. In the second part of the paper, we proposed a specific method based on the MFMFE methodology on simplicial grids in 2D and 3D. We use the \(RT_1\) mixed finite element pair along with an inexact numerical quadrature rule to obtain mass lumping and local elimination of the viscous stress and velocity in the prediction and projection problem, respectively. This results in symmetric and positive definite algebraic systems with only \(d+1\) unknowns per element, efficiently solved by the preconditioned conjugate gradient method.
The computed pressure and velocity are second order accurate in space. The numerical tests provided in this paper illustrate the efficiency and accuracy of the scheme for different benchmark problems, including challenging applications with highly irregular geometry and heterogeneous porous media with strong discontinuities of the porosity and permeability values. Furthermore, we show that the method is robust in both the Stokes and Darcy regimes and provides an efficient, accurate, and robust ODA alternative for modeling coupled free fluid and porous media flows. 

\section*{Acknowledgments}
This work was partially supported by the German Research Foundation (DFG), by funding Sonderforschungsbereich (SFB) 1313 (Project Number 327154368, Research Project A02), University
of Stuttgart, by funding SimTech via Germany’s Excellence Strategy (EXC 2075–390740016),
University of Stuttgart, and by NSF grant DMS-2410686.

\bibliographystyle{abbrv}
\bibliography{references.bib}
\end{document}